\documentclass[review,onefignum,onetabnum]{siamonline190516}

\usepackage{lipsum}
\usepackage{amsfonts}
\usepackage{graphicx}
\usepackage{epstopdf}
\usepackage{algorithmic}
\ifpdf
  \DeclareGraphicsExtensions{.eps,.pdf,.png,.jpg}
\else
  \DeclareGraphicsExtensions{.eps}
\fi

\usepackage{enumitem}
\setlist[enumerate]{leftmargin=.5in}
\setlist[itemize]{leftmargin=.5in}

\newsiamremark{remark}{Remark}
\newsiamremark{hypothesis}{Hypothesis}
\crefname{hypothesis}{Hypothesis}{Hypotheses}
\newsiamthm{claim}{Claim}

\headers{UQ for the kinetic BGK equation}{Jingwei Hu, Lorenzo Pareschi, and Yubo Wang}

\title{Uncertainty quantification for the BGK model of the Boltzmann equation using multilevel variance reduced Monte Carlo methods\thanks{Submitted to the editors DATE.\funding{J.H.'s research was supported in part by NSF grant DMS-1620250 and NSF CAREER grant DMS-1654152. L.P. would like to thank the Italian Ministry of Instruction, University and Research (MIUR) to support
this research with PRIN Project 2017, No. 2017KKJP4X, "Innovative numerical methods for evolutionary partial
differential equations and applications".}}}

\author{Jingwei Hu\thanks{Department of Mathematics, Purdue University, West Lafayette, IN 47907, USA 
  (\email{jingweihu@purdue.edu}).}
\and Lorenzo Pareschi\thanks{Department of Mathematics and Computer Science, University of Ferrara, Via Machiavelli 30, 44121-Ferrara, Italy 
  (\email{lorenzo.pareschi@unife.it}).}
\and Yubo Wang\thanks{Department of Mathematics, Purdue University, West Lafayette, IN 47907, USA 
  (\email{wang3158@purdue.edu}).}}

\usepackage{amsopn}

\ifpdf
\hypersetup{
  pdftitle={Uncertainty quantification for the BGK model of the Boltzmann equation using multilevel variance reduced Monte Carlo methods},
  pdfauthor={Jingwei Hu, Lorenzo Pareschi, and Yubo Wang}
}
\fi

\nolinenumbers

\usepackage{mathrsfs}

\newcommand{\rd}{\,\mathrm{d}}

\newcommand{\abs}[1]{\lvert #1 \rvert}

\newcommand{\p}{\partial}

\newcommand{\fl}[2]{\frac{#1}{#2}}
\newcommand{\ab}[1]{\lvert #1 \rvert}
\newcommand{\noab}[1]{\lVert #1 \rVert}

\newcommand{\be}{\begin{equation}}
\newcommand{\ee}{\end{equation}}
\newcommand{\bal}{\begin{aligned}}
\newcommand{\eal}{\end{aligned}}
\newcommand{\ba}{\begin{array}}
\newcommand{\ea}{\end{array}}
\newcommand{\bea}{\begin{eqnarray}}
\newcommand{\eea}{\end{eqnarray}}
\newcommand{\beas}{\begin{eqnarray*}}
\newcommand{\eeas}{\end{eqnarray*}}
\newcommand{\beit}{\begin{itemize}}
\newcommand{\eit}{\end{itemize}}

\newtheorem{assumption}{Assumption}[section]

\renewcommand{\vec}[1]{\boldsymbol{\mathrm{#1}}}
\newcommand{\bmat}[1]{\begin{bmatrix} #1 \end{bmatrix}}
\newcommand{\mat}[1]{\boldsymbol{#1}}

\providecommand{\mU}{\ensuremath{\mat{U}}}

\providecommand{\va}{\ensuremath{\vec{a}}}

\providecommand{\vn}{\ensuremath{\vec{n}}}

\providecommand{\vv}{\ensuremath{\vec{v}}}

\providecommand{\vx}{\ensuremath{\vec{x}}}

\providecommand{\h}{q}

\begin{document}

\maketitle

% REQUIRED
\begin{abstract}
We propose a control variate multilevel Monte Carlo method for the kinetic BGK model of the Boltzmann equation subject to random inputs. The method combines a multilevel Monte Carlo technique with the computation of the optimal control variate multipliers derived from local or global variance minimization problems. Consistency and convergence analysis for the method equipped with a second-order positivity-preserving and asymptotic-preserving scheme in space and time is also performed. Various numerical examples confirm that the optimized multilevel Monte Carlo method outperforms the classical multilevel Monte Carlo method especially for problems with discontinuities.
\end{abstract}

% REQUIRED

% REQUIRED
\overfullrule=0pt
\section{Introduction}

Kinetic theory, from a statistical physics viewpoint \cite{chapman1990mathematical},  represents an essential tool to model the non-equilibrium dynamics in a variety of fields including rarefied gases, semiconductors, plasmas, and even large particle systems in biological and social sciences \cite{cercignani1994r,villani2002review, ParTosbook}. The most fundamental kinetic equation, the Boltzmann equation, describes the statistical behavior of a thermodynamic system by taking into account particle transport and binary collisions \cite{cercignani1988boltzmann}:
\be
\label{eqn:boltzmann}
\p_{t} f +\vv \cdot \nabla_{\vx} f=\frac{1}{\varepsilon}\mathcal{Q}(f,f),\ \ \vx\in D \subset \mathbb{R}^3,\ \vv \in \mathbb{R}^3, \ t>0,
\ee
where $f=f(\vx,\vv,t)$ is the phase space distribution function of position $\vx$, velocity $\vv$ and time $t$. The collision term $\mathcal{Q}(f, f)$ is a high-dimensional, quadratic integral operator. $\varepsilon$ is the Knudsen number, defined as the ratio of the mean free path and the typical length scale. In most applications, $\varepsilon$ varies from $O(1)$, the kinetic regime, to $\varepsilon \ll 1$, the fluid regime. Albeit wide applicability of the Boltzmann equation, the complexity of the collision operator $\mathcal{Q}(f,f)$ makes both analysis and computation of the equation extremely challenging. Hence many simplified collisional models have been introduced to mimic the properties of the full Boltzmann operator. Among these, the Bhatnagar-Gross-Krook (BGK) model \cite{bhatnagar1954model}, which assumes a simple relaxation to equilibrium, has been widely used. The model reads as follows:
\be
\label{eqn:bgk}
\p_{t} f +\vv \cdot \nabla_{\vx} f=\frac{1}{\varepsilon}(M[f]-f),\ \ \vx\in D \subset \mathbb{R}^3,\ \vv \in \mathbb{R}^3, \ t>0,
\ee
where $M[f]$ is the so-called Maxwellian equilibrium function given by
\be
\label{eqn:maxwellian}
M[f](\vx,\vv,t)=\fl {\rho(\vx,t)} {(2\pi T(\vx,t))^{\fl {3} 2}}\exp\left({-\fl {\ab{\vv-\mU(\vx,t)}^2}{2T(\vx,t)}}\right),
\ee
where $\rho(\vx,t)$, $\mU(\vx,t)$, $T(\vx,t)$ are the density, bulk velocity, and temperature defined through the moments of $f$:
\be
\bal
\label{eqn:macroquant}
&\rho(\vx,t)=\int_{\mathbb{R}^3} f(\vx,\vv,t) \ {\rm d}\vv ,\text{\quad \quad} \mU(\vx,t)=\fl 1 {\rho(\vx,t)}\int_{\mathbb{R}^3} \vv f(\vx,\vv,t) \ {\rm d}\vv ,\\
&T(\vx,t)=\fl 1 {3\rho(\vx,t)} \int_{\mathbb{R}^3} \ab {\vv-\mU(\vx,t)}^2 f(\vx,\vv,t)\ {\rm d}\vv.
\eal
\ee
Let \(\Phi(\vv)=\left[1, \vv, \fl {1} 2 |\vv|^2\right]^T\), then one has the following conservation property:
\be
\label{eq:localconservation}
\int_{\mathbb{R}^3} M[f](\vx,\vv,t)\Phi(\vv) {\rm d}\vv=\int_{\mathbb{R}^3} f(\vx,\vv,t)\Phi(\vv){\rm d}\vv=\bmat{\rho\\ \rho\mU \\ \fl {{3}}{2}\rho T+\fl 1 2 {\rho{\ab {\mU}}^2}}=:\bmat{\rho\\ \mathbf{m} \\ E},
\ee
where $\mathbf{m}$ is the momentum and $E$ is the total energy. Using \cref{eq:localconservation}, if we multiply \cref{eqn:bgk} by \(\Phi(\vv)\) and integrate over $\vv$, we obtain the following local conservation law:
\be
\label{eq:localconservationsep}
\left\{
\bal
&\p_t \int_{\mathbb{R}^3} f {\rm d} \vv +\nabla_{\vx} \cdot \int_{\mathbb{R}^3} \vv f {\rm d} \vv =0,\\
&\p_t \int_{\mathbb{R}^3} \vv f {\rm d} \vv +\nabla_{\vx} \cdot \int_{\mathbb{R}^3} \vv \otimes \vv f {\rm d} \vv =0,\\
&\p_t \int_{\mathbb{R}^3}  \fl 1 2 \abs{\vv}^2 f {\rm d} \vv +\nabla_{\vx} \cdot \int_{\mathbb{R}^3} \fl 1 2\vv \abs{\vv}^2 f {\rm d} \vv =0.
\eal
\right.
\ee
When $\varepsilon\rightarrow 0$, formally we have $f \rightarrow M[f]$ from \cref{eqn:bgk}. Replacing $f$ by $M[f]$ in the above local conservation law yields the compressible Euler equations:
\be
\label{Euler}
\left\{
\bal
&\p_t  \rho +\nabla_{\vx} \cdot (\rho \mU) =0,\\
&\p_t (\rho \mU)  +\nabla_{\vx} \cdot (\rho \mU\otimes\mU+\rho TI) =0,\\
&\p_t E +\nabla_{\vx} \cdot ((E+\rho T)\mU )=0.\\
\eal
\right.
\ee

In the last decades, research activities in kinetic theory have focused mainly on deterministic kinetic equations, both theoretically and numerically \cite{cercignani1994r,villani2002review,dimarco2014numerical}, ignoring the presence of uncertain/random inputs. In reality, uncertainties may arise in initial/boundary conditions and other parameters, like the details of the microscopic interaction, because of incomplete knowledge or imprecise measurement. Recently, there has been a significant interest to study the impact of these random inputs in kinetic equations, see \cite{HJ17} and the whole collection \cite{jin2018uncertainty} for an overview. In order to quantify the above uncertainties, the construction of numerical methods for kinetic equations has been mostly oriented on stochastic Galerkin approximation based on generalized Polynomial Chaos expansion (gPC-sG), already successfully applied to many physical and engineering problems \cite{GS,Xiu}. We mention that recently gPC-sG methods have been successfully applied also to direct simulation Monte Carlo methods for the Boltzmann equation \cite{pareschi2020monte}. Despite the fact that gPC-sG methods have been able to show spectral accuracy for smooth solutions, they suffer the drawback of the curse of dimensionality and their highly intrusive nature. On one hand, existing codes for simulating the deterministic kinetic problems need to be completely reconfigured to implement the gPC-sG method. On the other hand, intrusiveness can induce some non-physical approximations even when the deterministic numerical solvers possess the correct physical properties. For example, due to the gPC expansion, the methods may induce approximations with non positive density; furthermore, close to fluid regimes, it is well-known that the gPC-sG system may lose hyperbolicity and lead to spurious solutions \cite{despres2013robust}. 

Another class of methods for uncertainty quantification is based on statistical Monte Carlo (MC) sampling, where the random space is sampled and the underlying deterministic PDE is solved for each sample. The non-intrusiveness of the method enables the approximated solutions to inherit properties, like positivity preservation, of the existing deterministic solvers and makes the parallel computing feasible for implementation. However, the asymptotic convergence rate is non-improvable by the central limit theorem and accelerated algorithms are obtained through variance reduction techniques \cite{caflisch_1998}. In this context, multifidelity methods for kinetic equations have been recently introduced in \cite{Dimarco2018, Dimarco2019, LiuXu}, see also the recent survey \cite{surveyPWG} for an introduction to the topic. These methods are capable to provide a significant speedup of the convergence properties of the Monte Carlo solver using as control variates simplified surrogate models that are cheaper to solve than the full model. A related line of research is based on the use of multilevel Monte Carlo (MLMC) methods (see for instance \cite{Mishra13, Mishra2016} and references therein for these methods applied to hyperbolic conservation laws), where the approximation of statistical expectation breaks up into telescopic sums of expectations of consecutive mesh sizes. These methods are closely related to multifidelity methods, since they essentially use in a recursive way the solution of the full model with various coarser meshes as surrogate models. 

In this manuscript, following the above analogy we develop multilevel Monte Carlo methods in a control variate setting for the multiscale kinetic equations. Therefore, in our MLMC method each level in the telescopic sum depends on an additional parameter which is computed in order to minimize the variance of the solver. We will perform this strategy, both locally between two different levels, as well as globally among all levels. As a prototype kinetic equation to design our methodology we consider the BGK model (\ref{eqn:bgk}) of the Boltzmann equation subject to random inputs.
Following the well-posedness results in \cite{Perthame1989,Perthame1993}, we provide a direct analogue of the former to the BGK equation with random parameters. Due to the non-intrusiveness of MC type methods, approximations of the statistical moments can preserve properties from the deterministic solvers. We adopt the Implicit-Explicit Runge-Kutta (IMEX-RK) scheme from \cite{Hu2018} to construct a second-order positivity-preserving (the distribution $f$ is positive for all $\varepsilon$) and asymptotic-preserving (the scheme becomes a solver for the limiting Euler sytem (\ref{Euler}) when $\varepsilon$ goes to zero) scheme for time and spatial discretizations. Various numerical examples confirm the good performance of MLMC methods compared to MC methods and that the control variate MLMC method outperforms the classical MLMC method especially for problems with discontinuities.
 
The rest of this paper is organized as follows. In the next section, we introduce the BGK equation with random inputs and establish the well-posedness of the equation. The Monte Carlo methods and analysis are presented in  \cref{sec:mc}, whereas in \cref{sec:cvmlmc} we discuss their multilevel extension in a standard and control variate setting. In \cref{sec:numericalresults} we show the numerical results obtained with standard MC, MLMC and control variate MLMC methods. Finally some conclusions are drawn in \cref{sec:conclusions}. In a separate \cref{appendix} we report the details of the dimension reduction method and the numerical scheme adopted to solve the deterministic BGK equation.

%%%%%%%%%%%%%%%%%%%%%%%%%%%%%%%%%%%%%%%%%%

\section{The BGK equation with random inputs}
\label{sec:bgkeqn}

In this section we formulate systematically the BGK equation with random inputs and establish the well-posedness of the equation by extending the results in \cite{Perthame1989,Perthame1993}.

\subsection{Setup of the problem}

In the BGK equation, due to the uncertain parameters coming from the initial condition or boundary condition, the resulting solution $f$ would be a random variable taking values in the functional space, where the solution of the BGK equation (\ref{eqn:bgk}) lies in. In most circumstances, it is the physical observables or macroscopic quantities (such as $\rho$, $\mU$, $T$) at certain time that are of interest, hence we will mainly consider the random variables taking values in $L^1(D)$, where $D$ is the physical domain. Following the discussion in \cite{Mishra2016}, we first present some basic concepts from probability theory and functional analysis.

Let $(\Omega,\mathscr{F},\mathbb{P})$ be a probability space with $\Omega$ being the set of elementary events, $\mathscr{F}$ the corresponding $\sigma$-algebra, and $\mathbb{P}$ the probability measure mapping $\Omega$ into $[0,1]$ such that $\mathbb{P}(\Omega)=1$.
A random variable taking values in $L^1(D)$, a separable Banach space, is defined to be any mapping $X: \Omega \rightarrow L^1(D)$ such that the set $\{w \in \Omega: X(w)\in A\}=\{X \in A\}\in \mathscr{F}$ for any $A \in \mathscr{G}$, where $(L^1(D), \mathscr{G})$ is a measurable space.

To define the expectation and variance of random variables in $L^1(D)$, we need the concept of Bochner integral by extending the Lebesgue integral theory. The strong measurable mapping $X:\Omega \rightarrow L^1(D) $ is $Bochner$  $integrable $ if, for any probability measure $\mathbb{P}$ on the measurable space $(\Omega,\mathscr{F})$,
\be
\int_{\Omega} \noab{X(w)}_{L^1(D)}  \  {\rm d}\mathbb{P}(w) <  \infty.
\ee
Moreover, any Bochner integrable random variable $X:\Omega \rightarrow {L^1(D)}$ can be approximated by a sequence of simple random variables $\{X_n\}_{n\in \mathbb{N}}$ defined as follows,
\be
X_n=\sum\limits_{i=1}^{N} x_{n,i}  \chi_{A_{n,i}},  \quad A_{n,i} \in \mathscr{F}, \ x_{n,i} \in {L^1(D)}, \ N < \infty.
\ee
To get moments like expectation or central moments like variance, similar as the derivation of the Lebesgue integral, the Bochner integral is defined by taking the limit of sequences of simple random variables $\{X_n(w)\}$, for example, the $k$-th order moments is defined as
\be
\mathbb{E}[X^k]:=\int_{\Omega} X^k(w) \  {\rm d}\mathbb{P}(w)=\mathop{lim}\limits_{n \rightarrow \infty} \int_{\Omega} X^k_n(w)\  {\rm d}\mathbb{P}(w),
\ee
and the variance is defined as
\be
\mathbb{V}[X]:=\mathbb{E}[(X-\mathbb{E}[X])^2]=\int_{\Omega} (X(w)-\mathbb{E}[X])^2 \  {\rm d}\mathbb{P}(w)=\mathbb{E}[X^2]-(\mathbb{E}[X])^2.
\ee
For the error analysis, we need to introduce the Banach space $L^p(\Omega,\mathscr{F},\mathbb{P};{L^1(D)})$ with the norm
\be
\noab{X}_{L^p(\Omega;{L^1(D)})}:=(\mathbb{E}[\noab{X}_{L^1(D)}^p])^{\fl 1 p} < \infty,\  1 \leq p < \infty;
\ee
and $L^{\infty}(\Omega,\mathscr{F},\mathbb{P};{L^1(D)})$ with the norm
\be
\noab{X}_{L^\infty(\Omega;{L^1(D)})}:={\text{ess}\sup}_{w\in \Omega}\noab{X}_{L^1(D)}.
\ee

The BGK equation with random inputs hence reads
\be
\bal
\label{eqn:bgkwithrandomness}
\p_{t} f(w; \vx,\vv,t) +\vv \cdot \nabla_{\vx} f(w; \vx,\vv,t)=\frac{1}{\varepsilon}(M[f](w; \vx,\vv,t)-f(w; \vx,\vv,t)), \\ w \in \Omega, \  \vx\in D \subset \mathbb{R}^3,\ \vv \in \mathbb{R}^3, \ t>0,\\
\eal
\ee
where
\be
\label{eqn:maxwellianwithrandomness}
M[f](w;\vx,\vv,t)=\fl {\rho(w;\vx,t)} {(2\pi T(w;\vx,t))^{\fl {3} 2}}\exp\left({-\fl {\ab{\vv-\mU(w;\vx,t)}^2}{2T(w;\vx,t)}}\right),
\ee
with
\be
\bal
\label{eqn:macroquantwithrandomness}
&\rho(w;\vx,t)=\int_{\mathbb{R}^3} f(w;\vx,\vv,t) \ {\rm d}\vv ,\text{\quad \quad} \mU(w;\vx,t)=\fl 1 {\rho(w;\vx,t)}\int_{\mathbb{R}^3} \vv f(w;\vx,\vv,t) \ {\rm d}\vv ,\\
&T(w;\vx,t)=\fl 1 {3\rho(w;\vx,t)} \int_{\mathbb{R}^3} \ab {\vv-\mU(w;\vx,t)}^2 f(w;\vx,\vv,t)\ {\rm d}\vv.
\eal
\ee
The initial condition is given as
\be \label{eqn:initialcondition}
f(w;\vx,\vv,0)=f_0(w;\vx,\vv), \quad w \in \Omega, \  \vx\in D \subset \mathbb{R}^3,\ \vv \in \mathbb{R}^3.
\ee
For the boundary condition, we consider one of the following:
\begin{itemize}
\label{BC}
\item periodic boundary: $f(w;\vx+\va,\vv,t)=f(w;\vx,\vv,t)$ for $\vx \in  \p D$ and some $\va \in \mathbb{R}^3$; 
\item Dirichlet boundary: $f(w;\vx,\vv,t)=g(w;\vx,\vv,t)$ for $\vx\in \p D$;
\item purely diffusive Maxwell boundary: for $\vx \in \p D$,
\be
f(w;\vx,\vv,t)=M_w(w;\vx,\vv,t),  \  \vv \cdot \vn <0,
\ee
where $\vn$ is outward normal of $\p D$ and $M_w$ is given by
\be
 M_w(w;\vx,\vv,t)=\fl {\rho_w(w;\vx,t)} {(2\pi T_w(w;\vx,t))^{\fl {3} 2}}\exp\left({-\fl {\ab{\vv}^2}{2T_w(w;\vx,t)}}\right),
\ee
where $T_w(w;\vx,t)$ is the wall temperature and $\rho_w(w;\vx,t)$ is chosen such that
\be
\label{eqn:conserdiffbc}
\int_{\vv\cdot \vn > 0}\vv\cdot \vn \,f(w;\vx,\vv,t)\ {\rm d}\vv=-\int_{\vv\cdot \vn < 0}\vv\cdot \vn \,M_w(w;\vx,\vv,t)\ {\rm d} \vv.
\ee
\end{itemize}

\subsection{Well-posedness of the equation and some estimates of the macroscopic quantities} 

In the following, we establish the well-posedness of the BGK equation \cref{eqn:bgkwithrandomness} with random inputs. We also obtain some estimates for the macroscopic quantities $\rho$, $\mU$ and $T$. For simplicity, we assume the periodic boundary condition and consider the uncertainty only arising in the initial condition $f_0$.

First of all, some general estimates on the macroscopic quantities can be obtained point-wise in $w$ following \cite{Perthame1993} for the deterministic BGK equation.
\begin{proposition}[\cite{Perthame1993}]
\label{prop:macroestimate}
Suppose $f(w;\vx,\vv,t) \geq 0${}. Define $\rho(w;\vx,t)$, $\mU(w;\vx,t)$, $T(w;\vx,t)$ according to \cref{eqn:macroquantwithrandomness}. Moreover, set
\be
N_q(f):=\sup_{\vv}f(w;\vx,\vv,t)\abs{\vv}^q,  \quad q\geq0.
\ee
Then the following estimates hold:
\be
\label{inequality3}
\fl {\rho(w;\vx,t)} {{T(w;\vx,t)}^{\fl {3} 2}} \leq C N_0(f),
\ee
\be
\label{inequality4}
\rho(w;\vx,t)(3T(w;\vx,t)+\ab{\mU(w;\vx,t)}^2)^{\fl{q-3}{2}}\leq C N_{q}(f), \quad \text{for  } q>5,
\ee
where $C$ is a positive constant depending only on $q$.
\end{proposition}

Based on the above estimates, one can obtain the existence and uniqueness of the solution to (\ref{eqn:bgkwithrandomness}) also following \cite{Perthame1993} in a point-wise manner in $w$.
\begin{theorem}[\cite{Perthame1993}]
\label{thm:unqandext}
Set
\be
\mathbb{N}_{q}(f) := \sup\limits_{\vx}\sup\limits_{\vv}f(w;\vx,\vv,t)(1+\ab{\vv}^{q}).
\ee
Suppose the initial condition $f_0(w;\vx,\vv) \geq 0${} and that for some  $q>5$,
\be
\bal
\label{ass:unqiandext}
&\mathbb{N}_{q}(f_0) = \sup\limits_{\vx}\sup\limits_{\vv}f_0(w;\vx,\vv)(1+\ab{\vv}^{q}):=A(w),\\
&\sup_{w}A(w)\leq A_0 <\infty,
\eal
\ee
and
\be
\label{ass:uniandextini}
\bal
&\gamma(w;\vx,t):=\int\limits_{\mathbb{R}^3} f_0(w;\vx-\vv t,\vv) \ {\rm d}\vv, \\
&\inf_w\inf_{\vx}\inf_{t}\gamma(w;\vx,t) \geq \inf_wC(w)\geq C_0 >0,
\eal
\ee
then, for fixed Knudsen number $\varepsilon >0$, there exists a unique mild solution of the initial-value problem \cref{eqn:bgkwithrandomness}-\cref{eqn:initialcondition} with periodic boundary condition. % in space $L^{\infty}_{\text{loc}}(0,+\infty;L^1(D\times\mathbb{R}^3;{\rm d}\vx(1+\abs{v}^2){\rm d}\vv))$.

Moreover, for all $t>0$, the following bounds hold:
\be
\label{inequality1}
\mathbb{N}_0(f(t)),\mathbb{N}_{q}(f(t)) \leq A_0\exp\left(\fl {C} {\varepsilon} t \right),
\ee
\be
\label{inequality2}
\inf_{\vx}\rho(w;\vx,t) \geq C_0\exp\left(-\fl t {\varepsilon}\right),
\ee
where $C$ is a constant depending only on $q$.
\end{theorem} 
%Note that in a bounded periodic domain $D$ assumptions \cref{ass:uniandextini} are automatically satisfied. 

As a direct consequence of \cref{prop:macroestimate} and \cref{thm:unqandext}, we have the following corollary on the upper bounds of the macroscopic quantities.
\begin{corollary}
\label{cor: uppbound}
Suppose the conditions in \cref{thm:unqandext} hold. We also assume the Knudsen number $\varepsilon \geq \varepsilon_0>0$. 
Then
for all $t>0$, the following bounds hold:
\be
\sup_w\sup_{\vx} \left\{\rho(w;\vx,t),\ \abs{\mU(w;\vx,t)},\ T(w;\vx,t)\right\}\leq C_1\exp\left(\fl {C_2}{\varepsilon_0}t\right),
\ee
where $C_1$ and $C_2$ are positive constants depending only on $A_0$, $C_0$ and $q$.
\end{corollary}
\begin{proof}
By \cref{inequality4}, \cref{inequality1} and \cref{inequality2}, we have
\be
(3T(w;\vx,t)+\ab{\mU(w;\vx,t)}^2)^{\fl{q-3}{2}}\leq \frac{C N_{q}(f)}{\rho(w;\vx,t)}\leq C_1\exp\left(\frac{C_2}{\varepsilon}t\right).
\ee
Hence
\be
T(w;\vx,t) \leq C_1\exp\left(\frac{C_2}{\varepsilon}t\right), \quad |\mU(w,\vx,t)|\leq C_1\exp\left(\frac{C_2}{\varepsilon}t\right).
\ee

By \cref{inequality3}, \cref{inequality1} and \cref{inequality2}, we have
\be \label{Tlow}
T(w;\vx,t)^{\fl 3 2}\geq \frac{C \rho(w;\vx,t)}{N_0(f)}\geq C_1\exp\left(-\frac{C_2}{\varepsilon}t\right).
\ee
Again using \cref{inequality3}, we have
\be
\rho(w;\vx,t)\leq \frac{CN_0(f)}{T(w;\vx,t)^{\fl 32}}.
\ee
Finally, by (\ref{Tlow}) and (\ref{inequality1}), we have
\be
\rho(w;\vx,t)\leq C_1\exp\left(\frac{C_2}{\varepsilon}t\right).
\ee
\end{proof}

%%%%%%%%%%%%%%%%%%%%%%%%%%%%%%%%%%%%%%%%%%

\section{Standard Monte Carlo method}
\label{sec:mc}

In this section, we describe the basic Monte Carlo sampling method to solve the BGK equation (\ref{eqn:bgkwithrandomness}) and establish some error estimates. For simplicity, we will consider that the uncertainty only comes from the initial condition. The case for the random boundary condition is similar.

\subsection{Monte Carlo method}
\label{subsec:mc}

Suppose we generate $M$ independent and identically distributed (i.i.d.) random samples $f_0^i$, $i=1,\dots,M,$ according to the random initial condition $f_0(w;\vx,\vv)$. Then each ${f}_0^i(w;\vx,\vv)$ will yield a unique analytical solution to (\ref{eqn:bgkwithrandomness}) at time $t$, denoted by ${f}^{i}(w;\vx,\vv,t)$.
From ${f}^{i}(w;\vx,\vv,t) $, we can easily compute
\be
\label{def:macroquant}
\bal
&{{\rho}^i(w;\vx,t)}={\int_{\mathbb{R}^3} {f}^{i}(w;\vx,\vv,t)\, {\rm d}\vv}, \quad {{\textbf{m}}^i(w;\vx,t)}={\int_{\mathbb{R}^3} \vv{f}^{i}(w;\vx,\vv,t)\, {\rm d}\vv},\\
&{{E}^i(w;\vx,t)}=\int_{\mathbb{R}^3} \fl {|\vv|^2} 2 {f}^{i}(w;\vx,\vv,t)\, {\rm d}\vv,
\eal
\ee
then $\mU^i$ and $T^i$ are given by
\be
\label{def:bulk u}
{\mU}^{i}(w;\vx,t)=\fl{{\textbf{m}}^{i}(w;\vx,t)}{{\rho}^i(w;\vx,t)}, \quad  {T}^{i}(w;\vx,t)=\fl {2{\rho}^i(w;\vx,t)E^i(w;\vx,t)-|{\textbf{m}}^i(w;\vx,t)|^2}{3({\rho}^i(w;\vx,t))^2}.
\ee
Since it is the macroscopic quantities we are interested in, in the following, without further notice we will use a single variable $q$ to denote $\rho$, $|\mU|$ or $T$. 

Given the samples $q^i$, $i=1,\dots,M$, the MC estimate of the expectation $\mathbb{E}[\h(w;\vx,t)]$ is given by
\be
\label{def:mc}
\mathbb{E}[\h(w;\vx,t)]\approx E_M[\h(w;\vx,t)]:=\fl 1 M \sum\limits_{i=1}^{M} {\h}^{i}(w;\vx,t).
\ee
To estimate the error between $\mathbb{E}[\h(w;\vx,t)]$ and $E_M[\h(w;\vx,t)]$, we need the following lemma.
\begin{lemma}
\label{lemma:l2idt}
For every finite sequence ${\{Y_j\}}_{j=1}^M$ of independent random variables with zero mean in $L^2(\Omega;L^2(D))$,
\be
\Big\|\sum\limits_{j=1}^{M}Y_j\Big\|^2_{L^2(\Omega;L^2(D))}=\sum\limits_{j=1}^{M}\|Y_j\|_{L^2(\Omega;L^2(D))}^2.
\ee
\end{lemma}
\begin{proof} 
From independence of ${\{Y_j\}}_{j=1}^M$ and that \(\mathbb{E}[Y_j]=0\),
\be
\bal
\Big\|\sum\limits_{j=1}^{M}Y_j\Big\|^2_{L^2(\Omega;L^2(D))}&=\int_D\mathbb{E} [(\sum\limits_{j=1}^{M}Y_j) ^2] \ {\rm  d} \vx=\int_D\mathbb{V} [\sum\limits_{j=1}^{M}Y_j] \ {\rm  d} \vx\\
&=\int_D\sum\limits_{j=1}^{M}\mathbb{V} [Y_j ] \ {\rm  d} \vx=\sum\limits_{j=1}^{M}\int_D\mathbb{E} [Y_j ^2]\  {\rm  d} \vx =\sum\limits_{j=1}^{M}\noab{Y_j}_{L^2(\Omega;L^2(D))}^2.
\eal
\ee
\end{proof}
We have the following consistency theorem.
\begin{theorem}
\label{thm:mcconsis}
For any M $\in \mathbb{N}^+$, at time $t=t_1$,
\be
 \noab {\mathbb{E}[\h(w;\vx,t_1)]-E_M[\h(w;\vx,t_1)]}_{L^2(\Omega;L^1(D))}
\leq  M^{-\fl 1 2} \ab{D}^{\fl 1 2} \noab{\mathbb{V}[\h(w;\vx,t_1)]}_{L^1(D)}^{\fl 1 2}.
\ee
\end{theorem}
\begin{proof}
We interpret the $M$ samples $\{{f}_0^i\}_{i=1}^{M}$ as unique realizations of $M$ independent samples of $f_0$ in the probability space $(\Omega,\mathscr{F},\mathbb{P})$. In other words, $\{{f}_0^i\}_{i=1}^{M}$ are i.i.d. copies of $f_0 \in L^1(D\times\mathbb{R}^3)$. 
As a result, the corresponding copies of macroscopic quantities ${\{{\h}^{i}(w;\vx,t_1) \}}_{i=1}^{M}$ derived from the initial data $\{{f}_0^i\}_{i=1}^{M}$ are also independent in $ L^2(\Omega;L^1(D))$.

Denote $\mathbb{E}[\h(w;\vx,t_1)]-{\h}^{i}(w;\vx,t_1)$ by $\Delta \h^i(w,\vx,t_1)$, then
\be
\mathbb{E}[\Delta \h^i(w,\vx,t_1)]=0,
\ee
and 
\be
\noab {\mathbb{E}[\h(w;\vx,t_1)]-E_M[\h(w;\vx,t_1)]}_{L^2(\Omega;L^1(D))}=M^{- 1 }\Big\|\sum_{i=1}^M\Delta \h^i(w,\vx,t_1)\Big\|_{L^2(\Omega;L^1(D))}.
\ee
Using the boundedness of domain $D$,
\be
\label{mcproof1}
\Big\|\sum_{i=1}^M\Delta \h^i(w,\vx,t_1)\Big\|^2_{L^1(D)}\leq \abs{D}\Big\|\sum_{i=1}^M\Delta \h^i(w,\vx,t_1)\Big\|^2_{L^2(D)}.
\ee
Taking the expectation, noting that $\Delta \h^i$ are independent and using \cref{lemma:l2idt}, we have
\be
\bal
&\Big\|\sum_{i=1}^M\Delta \h^i(w,\vx,t_1)\Big\|_{L^2(\Omega;L^1(D))}\leq \abs{D}^{\fl 1 2}\Big\|\sum_{i=1}^M\Delta \h^i(w,x,t_1)\Big\|_{L^2(\Omega;L^2(D))}\\
=&\abs{D}^{\fl 1 2}\sqrt{\sum_{i=1}^M\noab{\Delta \h^i(w,\vx,t_1)}^2_{L^2(\Omega;L^2(D))}}= \abs{D}^{\fl 1 2} M^{\fl 1 2} \noab{\Delta \h^i(w,\vx,t_1)}_{L^2(\Omega;L^2(D))}\\
=&\abs{D}^{\fl 1 2} M^{\fl 1 2}\noab{\mathbb{V}[\h(w;\vx,t_1)]}_{L^1(D)}^{\fl 1 2}.
\eal
\ee
\end{proof}

As a direct result of \cref{thm:mcconsis} and \cref{cor: uppbound},  we have the following convergence theorem.
\begin{theorem}
\label{thm:convmc}
 Under  assumptions of \cref{thm:unqandext} and \cref{cor: uppbound}, for $ 0<t_1<\infty$, as $M \rightarrow \infty$, the MC estimate $E_M[\h(w;\vx,t_1)]$ converges in $L^2(\Omega;L^1(D))$  to $\mathbb{E}[\h(w;\vx,t_1)]$. Furthermore, for any M $\in \mathbb{N}^+$,  there holds the error bound
\be
 \noab {\mathbb{E}[\h(w;\vx,t_1)]-E_M[\h(w;\vx,t_1)]}_{L^2(\Omega;L^1(D))}
\leq C_1 \abs{D} \exp\left(\fl {C_2} {\varepsilon_0} t_1\right)M^{-\fl 1 2}.
\ee
\end{theorem}

\begin{proof}
It only needs to note that
\be
\noab{\mathbb{V}[\h(w;\vx,t_1)]}_{L^1(D)}^{\fl 1 2}\leq \noab{\mathbb{E}[\h^2(w;\vx,t_1)]}_{L^1(D)}^{\fl 1 2}\leq \abs{D}^{\fl 1 2} C_1 \exp\left({\fl {C_2}{\varepsilon_0}t_1}\right).
\ee 
\end{proof}

%%%%%%%%%%%%%%%%%%%%%%%%%%%%%%%%%%%%%%%%%%

\subsection{Monte Carlo method with fully discrete scheme}

To complete the error analysis, we need to consider the Monte Carlo method coupled with the fully discrete scheme for the BGK equation, which includes discretization in time, physical space and velocity space. The details are given in the \cref{appendix}. Simply speaking, we are using the Gauss quadrature in the velocity space, second order IMEX-RK scheme for time discretization, and second order MUSCL finite volume scheme for spatial discretization (under the hyperbolic CFL condition $\Delta t\leq C\Delta x$). Overall, this leads to a second order positivity-preserving and asymptotic-preserving scheme for the deterministic BGK equation. In the following, we assume that the velocity discretization is accurate enough and ignore the work and error in velocity space. It is then reasonable to assume the numerical solution ${\h}_{\Delta x,\Delta t}(w;\vx,t_1)$, computed with mesh size $\Delta x$ and time step $\Delta t$ corresponding to initial data $f_0(w;\vx,\vv)$ up to time $t_1$, satisfies the following error estimate point-wise in $w$:
\begin{assumption}
\label{ass:conv}
For $0<t_1 < \infty$, we have
\be
\noab {{\h}(w;\vx,t_1)-{\h}_{\Delta x,\Delta t}(w;\vx,t_1)}_{L^1(D)}\leq C(w)\left({(\Delta x)}^2+{(\Delta t)}^2\right)\leq C_w(\Delta x)^2,
\ee
where the constant $C(w)$ has an upper bound $C_w$.
\end{assumption}

The MC estimate of the expectation $\mathbb{E}[\h(w;\vx,t)]$ is now given by
\be
\label{def:nummc}
\mathbb{E}[\h(w;\vx,t)]\approx E_M[\h_{\Delta x,\Delta t}(w;\vx,t)]:=\fl 1 M \sum\limits_{i=1}^{M}{\h}^{i}_{\Delta x,\Delta t}(w;\vx,t).
\ee
We have
\begin{theorem}
\label{thm:nummcconsis}
For any M $\in \mathbb{N}^+$, at time $t=t_1$,
\be
 \noab {\mathbb{E}[\h(w;\vx,t_1)]-E_M[\h_{\Delta x,\Delta t}(w;\vx,t_1)]}_{L^2(\Omega;L^1(D))}
\leq  M^{-\fl 1 2} \ab{D}^{\fl 1 2} \noab{\mathbb{V}[\h(w;\vx,t_1)]}_{L^1(D)}^{\fl 1 2}+C_w(\Delta x)^2.
\ee
\end{theorem}
\begin{proof}
\be
\bal
\noab {\mathbb{E}[\h(w;\vx,t_1)]-E_M[\h_{\Delta x,\Delta t}(w;\vx,t_1)]}_{L^2(\Omega;L^1(D))}&\leq \noab {\mathbb{E}[\h]-E_M[\h]}_{L^2(\Omega;L^1(D))}\\
&+\noab{E_M[\h]-E_M[\h_{\Delta x,\Delta t}]}_{L^2(\Omega;L^1(D))}.
\eal
\ee
It is enough to apply \cref{thm:mcconsis} and \cref{ass:conv}.
\end{proof}

The following corollary is a direct result of \cref{thm:nummcconsis}. 
\begin{corollary}
\label{cor:convnummc}
Under assumptions of \cref{thm:unqandext} and \cref{cor: uppbound}, for $0<t_1<\infty$, as $M \rightarrow \infty$ and $\Delta x$, $\Delta t \rightarrow 0$, the MC estimate $E_M[\h_{\Delta x,\Delta t}(w;\vx,t_1)]$ converges in $L^2(\Omega;L^1(D))$ to $\mathbb{E}[\h(w;\vx,t_1)]$. Furthermore, for any M $\in \mathbb{N}^+$,  there holds the error bound:
\be
\noab {\mathbb{E}[\h(w;\vx,t_1)]-E_M[\h_{\Delta x,\Delta t}(w;\vx,t_1)]}_{L^2(\Omega;L^1(D))}
\leq C_1\abs{D} \exp\left(\fl {C_2}{\varepsilon_0}t_1\right)M^{-\fl 1 2}+C_w(\Delta x)^2.
\ee
\end{corollary}

%%%%%%%%%%%%%%%%%%%%%%%%%%%%%%%%%%%%%%%%%%

\section{Control variate multilevel Monte Carlo method}
\label{sec:cvmlmc}

In this section we first introduce the multilevel Monte Carlo method and then following \cite{Dimarco2019} we discuss the use of control variate techniques to optimize its variance reduction properties locally using two subsequent levels or globally among all levels.

\subsection{Multilevel Monte Carlo method}

The MLMC method is defined as a multilevel discretization in $\vx$ and $t$ with a level $l$ dependent number of samples $M_l$. Suppose we have a nested triangulation $\{\mathcal{T}_l\}_{l=1}^L$ of the spatial domain $D$ ($L\in \mathbb{N}^{+}$ is the number of levels) such that the mesh size $\Delta x_l$ at level $l$ satisfies
\be
\Delta x_l=\sup\{diam(K):K\in \mathcal{T}_l\}\searrow  \text{   as  } l \nearrow.
\ee
Set $\h^{i}_{\Delta x_0,\Delta t_0}(w;\vx,t):=0$, then given a target level $L$ of spatial resolution, the MLMC estimate of the expectation $\mathbb{E}[\h(w;\vx,t)]$ is given as follows
\be
\label{def:nummlmc}
\bal
\mathbb{E}[\h(w;\vx,t)]&\approx E^L[\h_{\Delta x_L,\Delta t_L}(w;\vx,t)]\\
&:= \sum\limits_{l=1}^{L}E_{M_l}\left[ {\h}^{}_{\Delta x_l,\Delta t_l}(w;\vx,t)- {{\h}^{}_{\Delta x_{l-1},\Delta t_{l-1}}}(w;\vx,t)\right]\\
&=\sum\limits_{l=1}^{L}\sum\limits_{i=1}^{M_l} \fl 1 {M_l} \left[ {\h}^{i}_{\Delta x_l,\Delta t_l}(w;\vx,t)- {{\h}^{i}_{\Delta x_{l-1},\Delta t_{l-1}}}(w;\vx,t)\right].
\eal
\ee
Hence what we really sample is the difference of solutions at two consecutive levels. At each level $l$, we generate $M_l$ i.i.d. samples $f_0^i$, $i=1,\dots,M_l$, of the initial data $f_0$ on meshes $\Delta x_l$ and $\Delta x_{l-1}$ respectively, and then use the fully discrete scheme for the BGK equation (\ref{eqn:bgkwithrandomness}) to advance solutions ${\h}^{i}_{\Delta x_l,\Delta t_l}$ and ${\h}^{i}_{\Delta x_{l-1},\Delta t_{l-1}}$ to certain time $t$.

To simplify the notation, we set $\h_{\Delta x_0,\Delta t_0}(w;\vx,t):=0$ and define the random variable $Y_l := {\h}^{}_{\Delta x_l,\Delta t_l}(w;\vx,t)- {{\h}^{}_{\Delta x_{l-1},\Delta t_{l-1}}}(w;\vx,t)$, and the specific samples $Y_l^i := {\h}^{i}_{\Delta x_l,\Delta t_l}(w;\vx,t)- {{\h}^{i}_{\Delta x_{l-1},\Delta t_{l-1}}}(w;\vx,t)$. We have the following consistency and convergence results for the estimator \cref{def:nummlmc}.
\begin{theorem}
\label{thm:mlmcconsis}
For any $M_l \in \mathbb{N}^+$, $l=1,\dots,L$, at time $t=t_1$,
\be
\begin{split}
 \noab {\mathbb{E}[\h(w;\vx,t_1)]-E^L[\h_{\Delta x_L,\Delta t_L}(w;\vx,t_1)]}_{L^2(\Omega;L^1(D))} &\leq C_w{({\Delta x}_L)}^2\\
 &+ \abs{D}^{\fl 1 2} \sum\limits_{l=1}^{L}{{M_l}^{-\fl 1 2}}\noab{\mathbb{V}[Y_l]}_{L^1(D)}^{\fl 1 2}.
 \end{split}
\ee
\end{theorem}
\begin{proof}
\be
\bal
&\noab {\mathbb{E}[\h]-E^{L}[\h_{\Delta x_L,\Delta t_L}]}_{L^2(\Omega;L^1(D))}=\noab {\mathbb{E}[\h]-\sum\limits_{l=1}^{L}E_{M_l}^{}[Y_l]}_{L^2(\Omega;L^1(D))}\\
&\leq \noab {\mathbb{E}[\h]-\sum\limits_{l=1}^{L}\mathbb{E}[Y_l]}_{L^2(\Omega,L^1(D))}+\noab{\sum\limits_{l=1}^{L}E_{M_l}^{}[Y_l]-\sum\limits_{l=1}^{L}\mathbb{E}[Y_l]}_{L^2(\Omega;L^1(D))}\\
&\leq \noab{\mathbb{E}[\h]-\mathbb{E}[ {\h}_{\Delta x_L,\Delta t_L}]}_{L^1(D)}+\abs{D}^{\fl 1 2}\sum\limits_{l=1}^{L}\noab{E_{M_l}^{}[Y_l]-\mathbb{E}[Y_l]}_{L^2(\Omega;L^2(D))}\\
&=I+II.
\eal
\ee
For part $I$, \cref{ass:conv} yields,
\be
I= \noab {{\h}(w;\vx,t_1)-{\h}_{\Delta x_L,\Delta t_L}(w;\vx,t_1)}_{L^1(\Omega;L^1(D))}\leq C_w{({\Delta x}_L)}^2.
\ee
For part $II$, using \cref{lemma:l2idt},
\be
{II}=\abs{D}^{\fl 1 2} \sum\limits_{l=1}^{L}{{M_l}^{-\fl 1 2}}\noab{Y_l^i-\mathbb{E}[Y_l]}_{L^2(\Omega;L^2(D))}=\abs{D}^{\fl 1 2} \sum\limits_{l=1}^{L}{{M_l}^{-\fl 1 2}}\noab{\mathbb{V}[Y_l]}_{L^1(D)}^{\fl 1 2}.
\ee
\end{proof}

\begin{theorem}
\label{thm:convmlmc}
Under the assumptions of \cref{thm:unqandext} and \cref{cor: uppbound}, 
for $0<t_1<\infty$, as $M_l \rightarrow \infty$ and $\Delta x$, $\Delta t \rightarrow 0$, the MLMC estimate $E^L[\h_{\Delta x_L,\Delta t_L}(w;\vx,t_1)]$ converges in $L^2(\Omega;L^1(D))$  to $\mathbb{E}[\h(w;\vx,t_1)]$. Furthermore, there holds the error bound:
\be
\bal
&\noab {\mathbb{E}[\h(w;\vx,t_1)]-E^L[\h_{\Delta x_L,\Delta t_L}(w;\vx,t_1)]}_{L^2(\Omega;L^1(D))}\\
&\leq C_w {({\Delta x}_L)}^2
+\left(C_w\abs{D}^{\fl 1 2}(\Delta x_1)^2+C_1\abs{D}\exp\left(\fl {C_2}{\varepsilon_0}t_1\right)\right)M_1^{-\fl 1 2}\\
&+\sum\limits_{l=2}^{L}C_w\abs{D}^{\fl 1 2}\left((\Delta x_l)^2+(\Delta x_{l-1})^2\right){M_{l}^{-\fl 1 2 }}.
\eal
\ee
\end{theorem}
\begin{proof}
From \cref{thm:mlmcconsis}
for $l=1$,
\be
\bal
\noab{Y_1^i-\mathbb{E}[Y_1]}_{L^2(\Omega;L^2(D))}&=\noab{{\h}^i_{\Delta x_1,\Delta t_1}-\mathbb{E}[{\h}^i_{\Delta x_1,\Delta t_1}]}_{L^2(\Omega;L^2(D))}\\
&\leq \noab{{\h}^i_{\Delta x_1,\Delta t_1}}_{L^2(\Omega;L^2(D))}\\
&\leq \noab{{\h}^i_{\Delta x_1,\Delta t_1}-{\h}^i}_{L^2(\Omega;L^2(D))}+\noab{{\h}^i}_{L^2(\Omega;L^2(D))}\\
&\leq C_w(\Delta x_1)^2+\abs{D}^{\fl 1 2}C_1\exp\left(\fl {C_2} {\varepsilon_0} t_1\right),
\eal
\ee
and similarly  for $l \geq 2$,
\be
\bal
\noab{Y_l^i-\mathbb{E}[Y_l]}_{L^2(\Omega;L^2(D))}&\leq \noab{Y_l^i}_{L^2(\Omega;L^2(D))}\\
&=\noab{{\h}^i_{\Delta x_l,\Delta t_l}-{\h}^i_{\Delta x_{l-1},\Delta t_{l-1}}}_{L^2(\Omega;L^2(D))}\\
&\leq \noab{{\h}^i_{\Delta x_l,\Delta t_l}-{\h}^i}_{L^2(\Omega;L^2(D))}+\noab{{\h}^i-{\h}^i_{\Delta x_{l-1},\Delta t_{l-1}}}_{L^2(\Omega;L^2(D))}\\
&\leq  C_w((\Delta x_l)^2+(\Delta x_{l-1})^2).
\eal
\ee

\end{proof}

\subsection{Quasi-optimal and optimal multilevel Monte Carlo method}

In this section we generalize the previous MLMC method following \cite{Dimarco2018}.
To start with, take the $2$ level MLMC method for example. Suppose we have a low fidelity (coarse mesh) approximation $\h_1$ and a high fidelity (fine mesh) approximation $\h_2$, then the 2 level MLMC method with control variate (assume $M_2 \ll M_1$) reads as follows
\be
\label{2lcvmlmc}
\mathbb{E}[\h] \approx \lambda E_{M_1}[\h_1]+E_{M_2}[\h_2-\lambda \h_1],
\ee
where the multiplier $\lambda$ has to be determined in order to minimize the variance. It can be shown that for independent samples the optimal value of $\lambda$ is given by
\begin{equation}
\lambda=\frac{\text{Cov}[\h_1,\h_2]}{\mathbb{V}[\h_1]} \approx \frac{\sum\limits_{i=1}^{M_2} (q^i_1-\bar{q}_1)(q^i_2-\bar{q}_2)}{\sum\limits_{i=1}^{M_2}(q^i_1-\bar{q}_1)^2},
\label{eq:opt}
\end{equation}
where $\bar{q}_1=E_{M_2}[\h_1]$, $\bar{q}_2=E_{M_2}[\h_2]$ and in the above expression the covariance and variance are estimated directly from the Monte Carlo samples.

Generally, suppose we have $L$ levels of solutions $\{\h_{\Delta x_i,\Delta t_i}\}_{i=1,...,L}$, from coarsest level $\h_{\Delta x_1,\Delta t_1}$ to finest level $\h_{\Delta x_L,\Delta t_L}$. Then the MLMC method with control variates is given by
\be
\label{def:cvmlmce}
\bal
\mathbb{E}[\h(w;\vx,t)]&\approx E_{CV}^L[\h_{\Delta x_L,\Delta t_L}]\\
 &:=\prod\limits_{i=1}^L \lambda_i E_{M_1}[\h_{\Delta x_1,\Delta t_1}]+\sum\limits_{l=2}^L\prod\limits_{i=l}^L\lambda_iE_{M_l}[\h_{\Delta x_l,\Delta t_l}-\lambda_{l-1}\h_{\Delta x_{l-1},\Delta t_{l-1}}].
\eal
\ee
Note that \(\{\lambda_l\}_{l=1}^L\) here are the coefficients to be determined and \(\lambda_L=1\). If we only consider the variance reduction for each pair of consecutive levels, then we can easily get the analogy of \eqref{eq:opt} to estimate $\{\lambda_l\}$, which we refer to as the \emph{quasi-optimal MLMC method}:
\be
\lambda_{l-1}=\fl {\text{Cov}[\h_{\Delta x_{l},\Delta t_{l}},\h_{\Delta x_{l-1},\Delta t_{l-1}}]}{\mathbb{V}[\h_{\Delta x_{l-1},\Delta t_{l-1}}]}\approx \fl {\sum\limits_{i=1}^{M_l}(\h^{i}_{\Delta x_{l},\Delta t_{l}}-{{\bar{\h}}_{\Delta x_{l},\Delta t_{l}}})(\h^i_{\Delta x_{l-1},\Delta t_{l-1}}-{{\bar{\h}}_{\Delta x_{l-1},\Delta t_{l-1}}})}{\sum\limits_{i=1}^{M_l}(\h^i_{\Delta x_{l-1},\Delta t_{l-1}}-{\bar{\h}_{\Delta x_{l-1},\Delta t_{l-1}}})^2},
\label{eq:quasi}
\ee
where \(\bar{\h}_{\Delta x_{l},\Delta t_{l}}=E_{M_l}[\h_{\Delta x_{l},\Delta t_{l}}]$.

However, if we focus on minimizing the overall variance of the estimator \cref{def:cvmlmce} and assume that the levels are independent, then denoting 
\begin{equation}
\hat{\lambda}_l=\prod\limits_{i=l}^{L}\lambda_i,\quad l=1,\ldots, L,
\end{equation}
the optimality conditions yield a tridiagonal system for $\hat{\lambda}_l$: 
\be
\label{eqn:optimallambda}
\bal
&\hat{\lambda}_l\mathbb{V}[\h_{\Delta x_l, \Delta t_l}]-\hat{\lambda}_{l+1}\frac{M_{l}}{M_l+M_{l+1}}\text{Cov}[\h_{\Delta x_{l+1}, \Delta t_{l+1}},\h_{\Delta x_l, \Delta t_l}]\\
&-\hat{\lambda}_{l-1}\frac{M_{l+1}}{M_l+M_{l+1}}\text{Cov}[\h_{\Delta x_{l-1}, \Delta t_{l-1}},\h_{\Delta x_l, \Delta t_l}]=0, \quad l=1,\dots,L-1,
\eal
\ee
where we assumed \(\hat{\lambda}_0=0\), \(\hat{\lambda}_L=1\) and \(\h_{\Delta x_0,\Delta t_0}=0\). A practical way to solve the above tridiagonal system is to rewrite \cref{eqn:optimallambda} in terms of original $\lambda_i$. For simplicity, we denote \(\mathbb{V}[\h_{\Delta x_l, \Delta t_l}]\) by \(\mathbb{V}_l\) and \(\text{Cov}[\h_{\Delta x_{l+1}, \Delta t_{l+1}},\h_{\Delta x_l, \Delta t_l}]\) by \(\text{Cov}_l\) to get
\be
\label{eqn:forwardsub}
\bal
&\lambda_1\mathbb{V}_1-\fl {M_1}{M_1+M_2}\text{Cov}_1=0,\\
&\lambda_2\mathbb{V}_2-\fl {M_2}{M_2+M_3}\text{Cov}_2-\lambda_1\lambda_2 \fl {M_3}{M_2+M_3}\text{Cov}_1=0,\\
&\lambda_3\mathbb{V}_3-\fl {M_3}{M_3+M_4}\text{Cov}_3-\lambda_2\lambda_3 \fl {M_4}{M_3+M_4}\text{Cov}_2=0,\\
&...\\
&\lambda_{L-1}\mathbb{V}_{L-1}-\fl {M_{L-1}}{M_{L-1}+M_{L}}\text{Cov}_{L-2}-\lambda_{L-2}\lambda_{L-1} \fl {M_{L}}{M_{L-1}+M_{L}}\text{Cov}_{L-2}=0,
\eal
\ee
which can be easily solved by recursive substitution. This is what we refer to as the \emph{optimal MLMC method}.

Denote the correlation coefficient of $\h_{\Delta x_{l}, \Delta t_{l}}$ and $\h_{\Delta x_{l+1}, \Delta t_{l+1}}$ by
\be
\label{correlation}
r_{l}=\fl {\text{Cov}[\h_{\Delta x_{l},\Delta t_{l}},\h_{\Delta x_{l+1},\Delta t_{l+1}}]}{\left(\mathbb{V}[\h_{\Delta x_{l+1},\Delta t_{l+1}}]\mathbb{V}[\h_{\Delta x_{l},\Delta t_{l}}]\right)^{\fl 1 2}}, 
\ee
we can prove the following consistency and convergence results for the estimator \cref{def:cvmlmce}:
\begin{theorem}
\label{thm:cvmlmcconsis}
For any $M_l\in \mathbb{N}^+$, $l=1,\dots,L$, if $\{\lambda_l\}$ are quasi-optimal and exact, i.e.,
\be
\lambda_l=\fl{\text{Cov}[\h_{\Delta x_{l},\Delta t_{l}},\h_{\Delta x_{l+1},\Delta t_{l+1}}]}{\mathbb{V}[\h_{\Delta x_{l},\Delta t_{l}}]},
\ee
then at time $t=t_1$,
\be
\bal
&\noab {\mathbb{E}[\h(w;\vx,t_1)]-E_{CV}^L[\h_{\Delta x_L,\Delta t_L}(w;\vx,t_1)]}_{L^2(\Omega;L^1(D))}\\
&\leq C_w{({\Delta x}_L)}^2+ \abs{D}^{\fl 1 2} {{M_1}^{-\fl 1 2}}\hat{\lambda}_1\noab{\mathbb{V}[\h_{\Delta x_{1},\Delta t_{1}}]}_{L^1(D)}^{\fl 1 2}\\
&+\abs{D}^{\fl 1 2} \sum\limits_{l=2}^{L}{{M_l}^{-\fl 1 2}}\hat{\lambda}_l(1-r_{l-1}^2)^{\fl 1 2}\noab{\mathbb{V}[\h_{\Delta x_{l},\Delta t_{l}}]}_{L^1(D)}^{\fl 1 2}.
\eal
\ee
\end{theorem}
\begin{proof}
The proof is similar to \cref{thm:mlmcconsis}. All we need is to note that when $\lambda$ is quasi-optimal, we have for $l \geq 2$,
\be
\bal
\mathbb{V}[\h_{\Delta x_l,\Delta t_l}-\lambda_{l-1}\h_{\Delta x_{l-1}, \Delta t_{l-1}}]&=\mathbb{V}[\h_{\Delta x_l,\Delta t_l}]+\lambda_{l-1}^2 \mathbb{V}[\h_{\Delta x_{l-1},\Delta t_{l-1}}]
\\
&-2\lambda_{l-1}\text{Cov}[\h_{\Delta x_l,\Delta t_l},\h_{\Delta x_{l-1},\Delta t_{l-1}}]\\
&= (1-r_{l-1}^2)\mathbb{V}[\h_{\Delta x_l,\Delta t_l}].
\eal
\ee
\end{proof}

\begin{theorem}
Under the assumptions of \cref{thm:unqandext} and \cref{cor: uppbound}, and if $\{\lambda_l\}$ are quasi-optimal and exact, we have
for $0<t_1<\infty$, as $M_l \rightarrow \infty$ and $\Delta x$, $\Delta t \rightarrow 0$, the quasi-optimal MLMC estimate $E_{CV}^L[\h_{\Delta x_L,\Delta t_L}(w;\vx,t_1)]$ converges in $L^2(\Omega;L^1(D))$  to $\mathbb{E}[\h(w;\vx,t_1)]$ with the error bound
\be
\bal
&\noab {\mathbb{E}[\h(w;\vx,t_1)]-E_{CV}^L[\h_{\Delta x_L,\Delta t_L}(w;\vx,t_1)]}_{L^2(\Omega;L^1(D))}\\
&\leq C_w {({\Delta x}_L)}^2+\sum\limits_{l=2}^{L}C_w\abs{D}^{\fl 1 2}\hat{\lambda}_l{M_{l}^{-\fl 1 2 }}(1-r_{l-1}^2)^{\fl 1 2}(\Delta x_l)^2\\
&+\left(C_w\abs{D}^{\fl 1 2}(\Delta x_1)^2+C_1\abs{D} \exp\left(\fl {C_2}{\varepsilon_0}t_1\right)\right)M_1^{-\fl 1 2}\hat{\lambda}_1.
\eal
\ee
\end{theorem}

\begin{remark}
We emphasize that the computational cost for quasi-optimal and optimal MLMC is the same as the standard MLMC method. In fact, we can use the data from MLMC to estimate the $\lambda_l$  using \cref{eq:quasi} or \cref{eqn:forwardsub}. 
\end{remark}

%%%%%%%%%%%%%%%%%%%%%%%%%%%%%%%%%%%%

\section{Numerical results}
\label{sec:numericalresults}

In this section, we present several numerical examples for the BGK equation (\ref{eqn:bgkwithrandomness}) with random initial condition or random boundary condition. The details of the deterministic solver are provided in \cref{appendix}. Simply speaking, we are solving a reduced system \cref{eqn:1dbgkchuphi} and \cref{eqn:1dbgkchupsi}, which is equivalent to the full BGK equation in one spatial dimension. We use the IMEX-RK scheme for time discretization and finite volume scheme for spatial discretization so that the overall method is second order in both time and space. We choose $x\in[0,1]$ and $v\in[-5,5]$, where $40$ Legendre-Gauss quadrature points are used in the velocity space to ensure that the error in velocity is negligible. The CFL condition is fixed as $\Delta t=0.1\Delta x$.

\subsection{Error evaluation}

In the following, we assume the uncertainties come from either the initial condition or boundary condition. Since the solution is a random field, the numerical error is a random quantity as well. For error analysis, we therefore compute a statistical estimator by averaging numerical errors from several independent experiments.

More precisely, for each method we perform $K=40$ experiments, and get the corresponding approximations $\{q^{(j)}(x,t)\}$, $j=1,\ldots,K$, where $q$ can be $\rho$, $U$ or $T$. We approximate the overall error in norm $\noab{\cdot}_{L^2(\Omega;L^1(D))}$ via
\be
\label{def:overallerror}
E(t)=\sqrt{\fl 1 K \sum\limits_{j=1}^K\noab{q^{(j)}(\cdot,t)-q_{\text{ref}}(\cdot,t)}^2_{L^1(D)}},
\ee
where $q_{\text{ref}}(x,t)$ is the reference solution obtained using the stochastic collocation method \cite{Xiu} with $120$ Legendre-Gauss collocation points and $N_x=1280$ spatial points. We are also interested in the error at each spatial point:
\be
\label{def:spatialerror}
E_{\Delta x}(x,t)=\sqrt{\fl 1 K \sum\limits_{j=1}^K({q^{(j)}(x,t)-q_{\text{ref}}(x,t)})^2}.
\ee

Sometimes to better evaluate the error from the random domain, we would like to ignore the error induced by spatial discretization. To achieve so, we consider another kind of reference solution, $q_{\text{rel}}(x,t)$, obtained again using the stochastic collocation with $120$ collocation points, while in the spatial domain we use the same finest mesh $\Delta x_L$ as that in the corresponding MLMC method to obtain $q^{(j)}(x,t)$. Therefore, we can assess the error as
\be
\label{def:relativespatialerror}
E_{\text{rel}\Delta x}(x,t)=\sqrt{\fl 1 K \sum\limits_{j=1}^K({q^{(j)}(x,t)-q_{\text{rel}}(x,t)})^2}.
\ee

\subsection{Test 1: Smooth random initial condition}

We first consider the BGK equation subject to random initial condition:
\begin{equation}
\label{ini:noneqini}
f^0(\vx,\vv,z)=0.5M_{\rho,\mU,T}+0.5M_{\rho,-\mU,T},
\end{equation}
with
\be
M_{\rho,\mU, T}(\vx,\vv,z)=\fl {\rho(\vx,z)} {(2\pi T(\vx,z))^{\fl {3} 2}}\exp\left({-\fl {\ab{\vv-\mU(\vx,z)}^2}{2T(\vx,z)}}\right),
\ee
where
\be
\bal
\label{initialcondition}
&\rho(\vx,z)=\fl {2+\sin(2\pi x)+\fl 1 2 \sin(4\pi x)z} 3, \quad \mU(\vx)=(0.2,0,0), \\
&T(\vx,z)=\fl {3+\cos(2\pi x)+\fl 1 2 \cos(4\pi x)z} 4,
\eal
\ee
and the random variable $z$ obeys the uniform distribution on $[-1,1]$. The periodic boundary condition is used and the Knudsen number $\varepsilon=1$.

To determine the number of samples needed in MC and MLMC methods, we proceed as follows.

In the MC method, we consider a series of spatial discretizations: $N=10$, $20$, $30$, $40$, and for each case, we vary the sample size as $M=5$, $10$, $15$, ... The results are shown in  \cref{fig:mctest} (left), where we plot the error \cref{def:overallerror}. It can be observed that when the number of samples is few, the statistical error dominates and when there are enough number of samples, the spatial error dominates. Therefore, we can roughly determine the best number of samples so that the statistical error $O(M^{-\fl 12})$ balances with the spatial/temporal error $O(\Delta x^2)$:
\begin{itemize}
  \item $N=10$, $M$ $\approx 40$.
  \item $N=20$, $M$ $\approx 640$.
  \item $N=30$, $M$ $\approx 3300$.
  \item $N=40$, $M$ $\approx 10240$.
\end{itemize}

In the MLMC method, we consider three levels of spatial discretizations: $N_1=10$, $N_2=20$, $N_3=40$ and the corresponding number of samples at each level are chosen as $M_1$, $M_2=\fl {M_1}{4}$ and $M_3=\fl {M_1} {16}$. We then vary the starting sample size as $M_1=16$, $32$, $48$, ... The results are shown in \cref{fig:mctest} (right). Roughly we can see that $M_1 \approx 10240$ gives the smallest error (the error saturates when the sample size further increases).

\begin{figure}[tb]
\begin{center}
\includegraphics[width=2.3in]{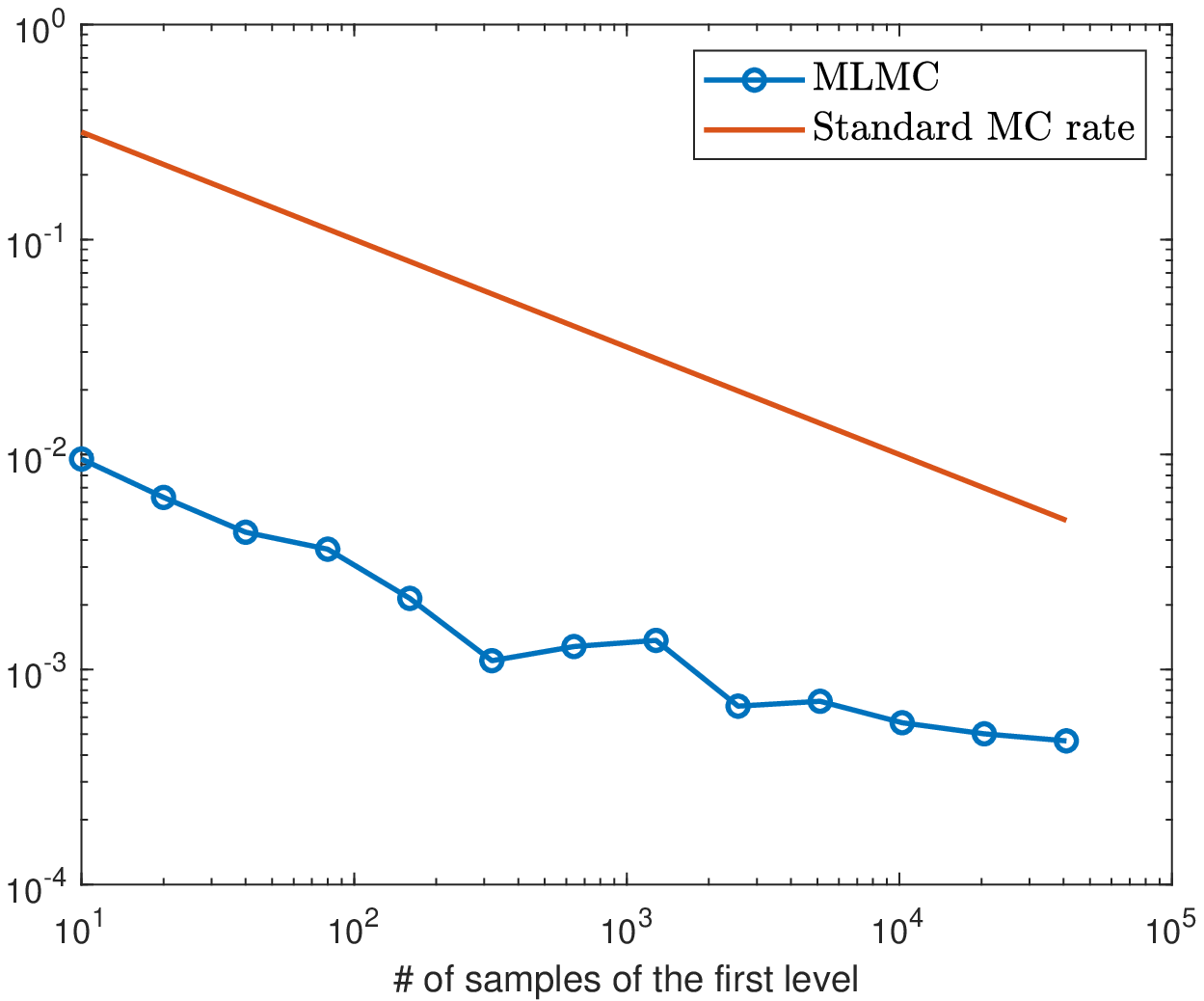}
\includegraphics[width=2.3in]{alltest}
\caption{Test 1: Error \cref{def:overallerror} (density $\rho$) of MC method (left) and MLMC method (right) v.s. number of samples (for MLMC, it is the number of samples at the first level).}
\label{fig:mctest}
\end{center}
\end{figure}

In \cref{fig:alltest} we combine all the previous MC and MLMC results under the scale of workload. Since we are essentially solving 1D BGK problem, the workload for one deterministic run up to certain time with $N$ spatial points is $O(N^2)$. Then for the MC method with $M$ samples, the total work is $O(MN^2)$. For the MLMC method, the amount of work is $O(M_1N_1^2+M_2N_2^2+M_3N_3^2=3M_1N_1^2)$. As we can see clearly from  \cref{fig:alltest}, with the same workload, the MLMC method can achieve better accuracy compared to various MC. 

\begin{figure}[tb]
\begin{center}
\includegraphics[width=2.58in]{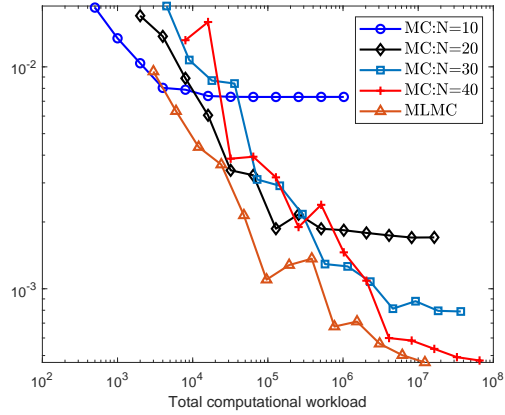}
\caption{Test 1: Error \cref{def:overallerror} (density $\rho$) of MC and MLMC methods v.s. computational workload.}
\label{fig:alltest}
\end{center}
\end{figure}

Now we fix the mesh sizes $N_1=10$, $N_2=20$, $N_3=40$, and sample sizes $M_1=10240$, $M_2=2560$, $M_3=640$ in the MLMC method. We then find the number of samples in the MC method such that they have the same workload. This means
\begin{itemize}
  \item $N=10$, $M=30720$.
  \item $N=20$, $M=7680$.
  \item $N=30$, $M=3413$.
  \item $N=40$, $M=1920$.
  %\item MLMC: 3 levels with $N_1=10,N_2=20, N_3=40$ with  $M_1=10240$, $M_2=2560$ and $M_3=640$.
\end{itemize}
Note that comparing with the numbers we found earlier, for $N=10$ and $20$, the number of samples are far beyond the best number of samples, while for $N=30$, $M$ is around the best number of samples. Finally for $N=40$, the number of samples here is not enough to achieve the best accuracy in the MC method. Using the above parameters, we compare the errors of the standard MC method and three MLMC methods, namely, the standard MLMC, the quasi-optimal MLMC, and optimal MLMC. The results are shown in \cref{fig:timeevopb}, from which we clearly see the better accuracy of MLMC methods compared to standard MC. On the other hand, the difference of three MLMC methods are not obvious in this example.

\begin{figure}[tb]
\begin{center}
\includegraphics[width=2.58in]{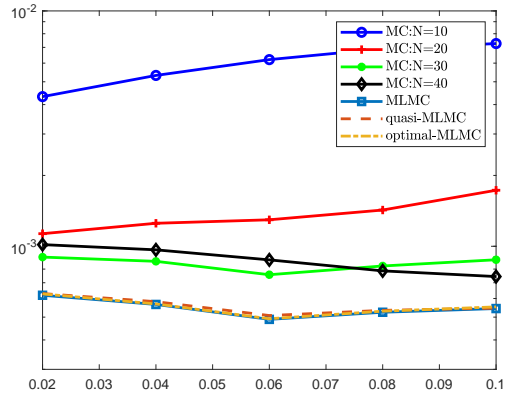}
\caption{Test 1: Time evolution of the errors \cref{def:overallerror} (density $\rho$) using MC and various MLMC methods.}
\label{fig:timeevopb}
\end{center}
\end{figure}

Next we examine the errors of the three MLMC methods as defined in \cref{def:spatialerror},\cref{def:relativespatialerror}. The results are gathered in \cref{fig:3mlmc}. We can see that the three methods perform equally well in this test (at some points of the domain, the errors of quasi-optimal MLMC and optimal MLMC are slightly better than standard MLMC), largely because the solution is smooth. 

\begin{figure}[!tb]
\begin{center}
\includegraphics[width=1.89in]{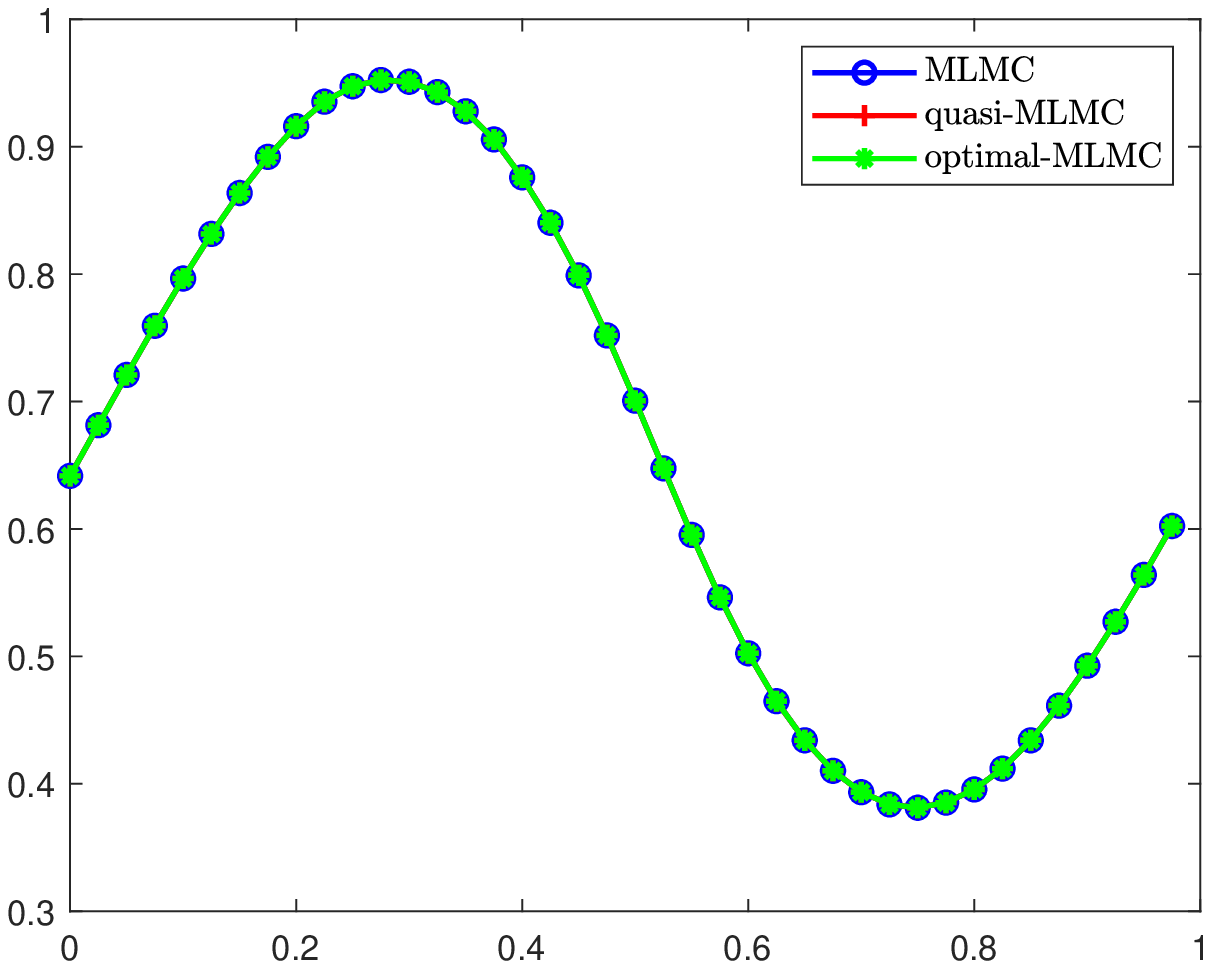}
\includegraphics[width=1.89in]{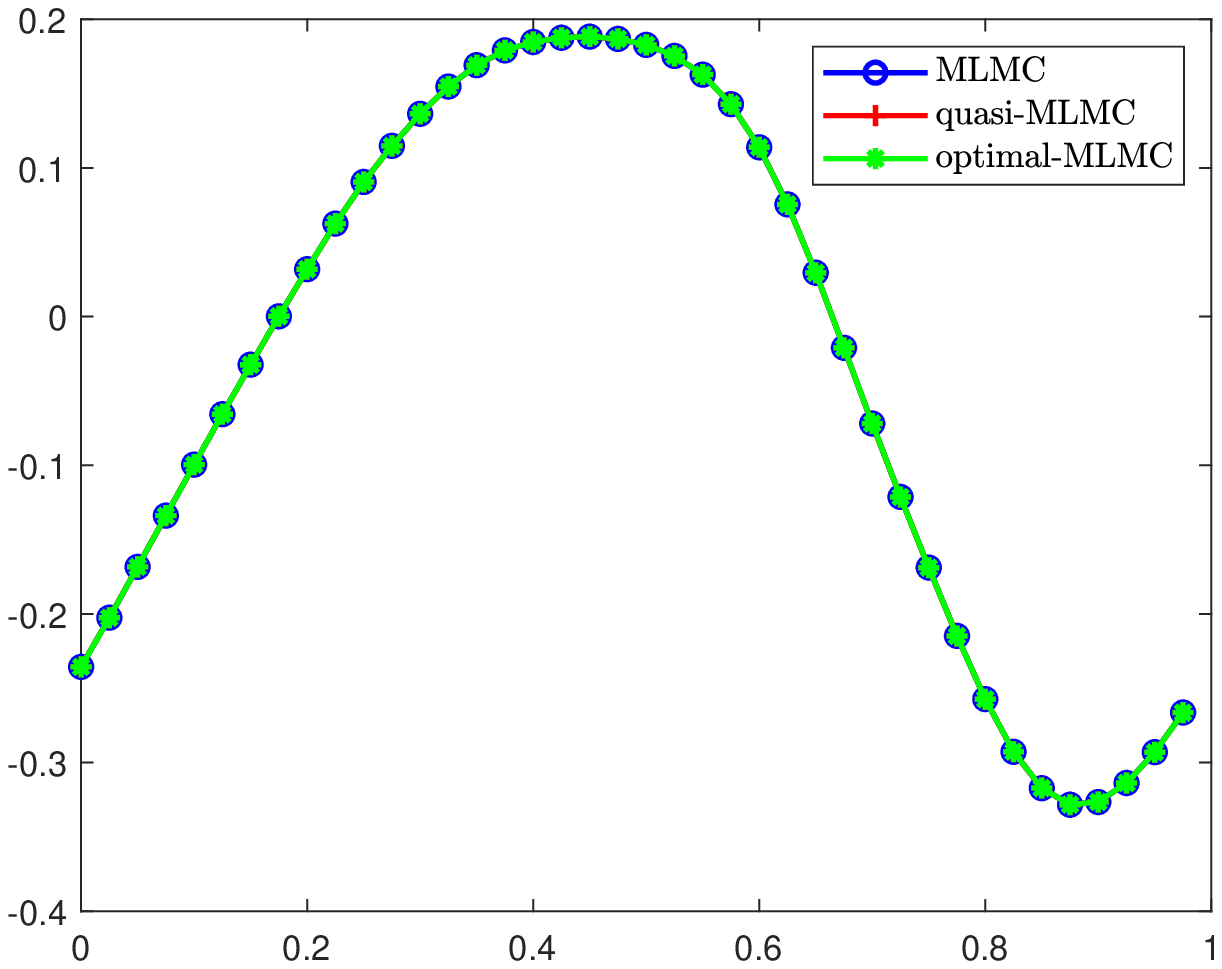}
\includegraphics[width=1.89in]{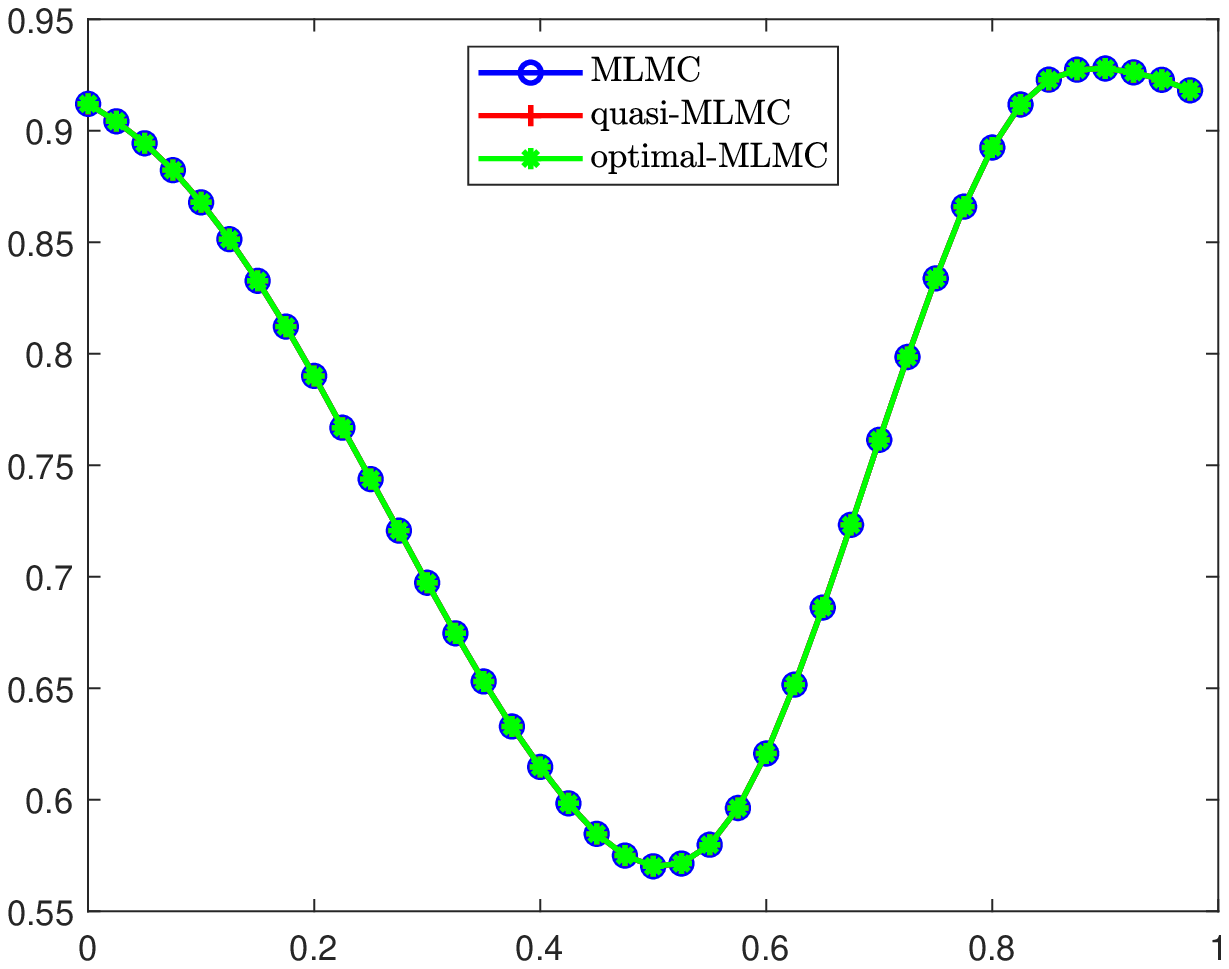}

\includegraphics[width=1.89in]{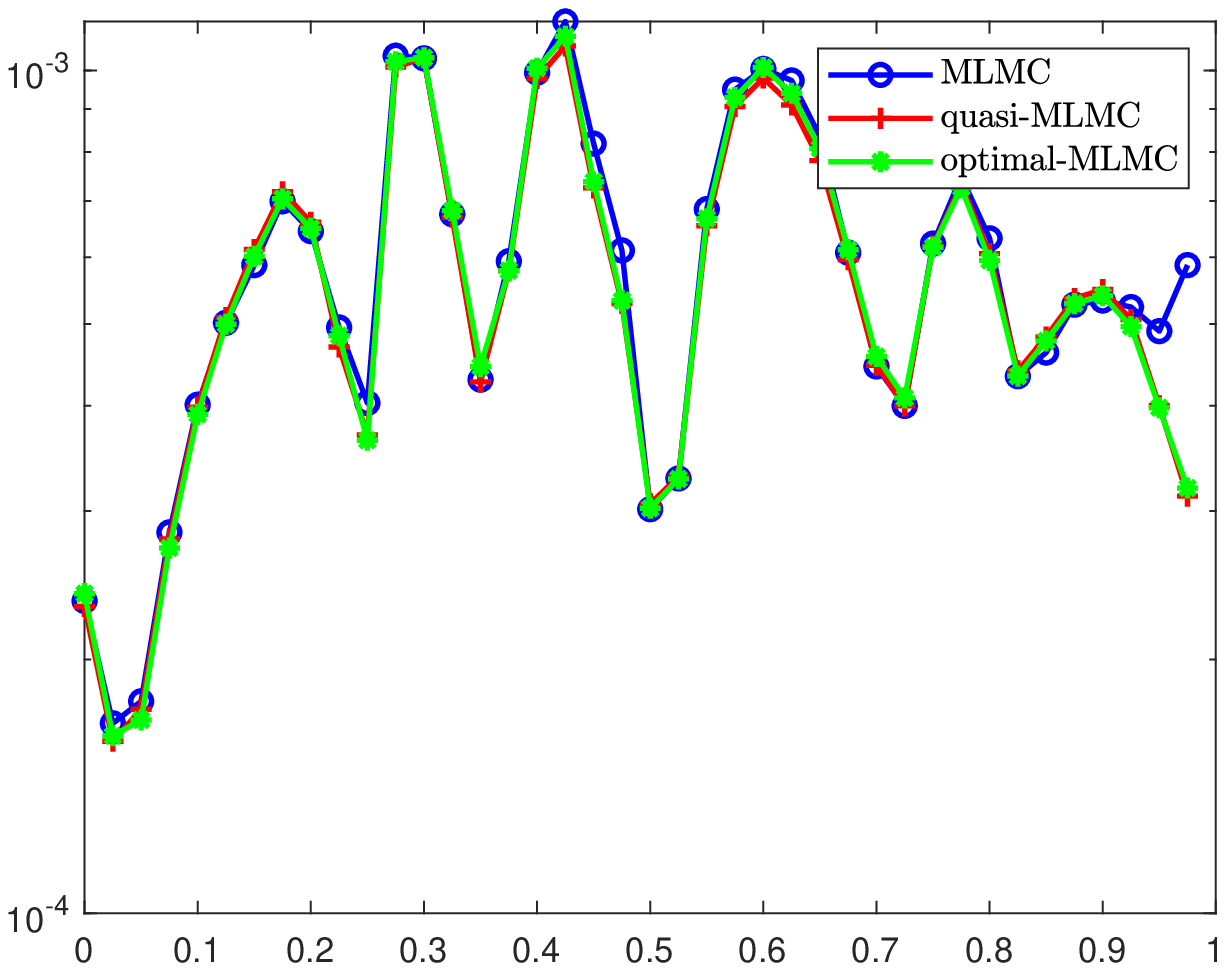}
\includegraphics[width=1.89in]{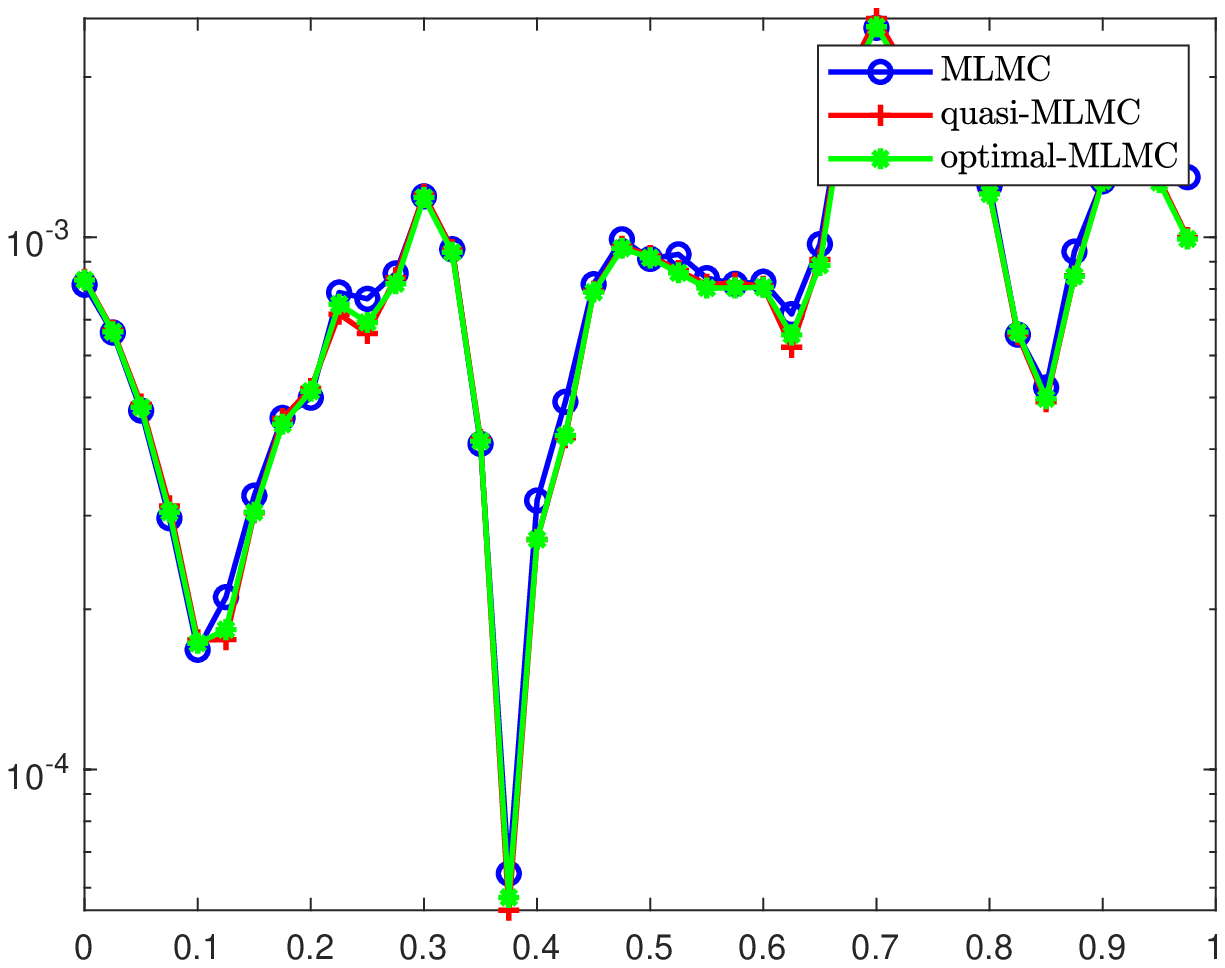}
\includegraphics[width=1.89in]{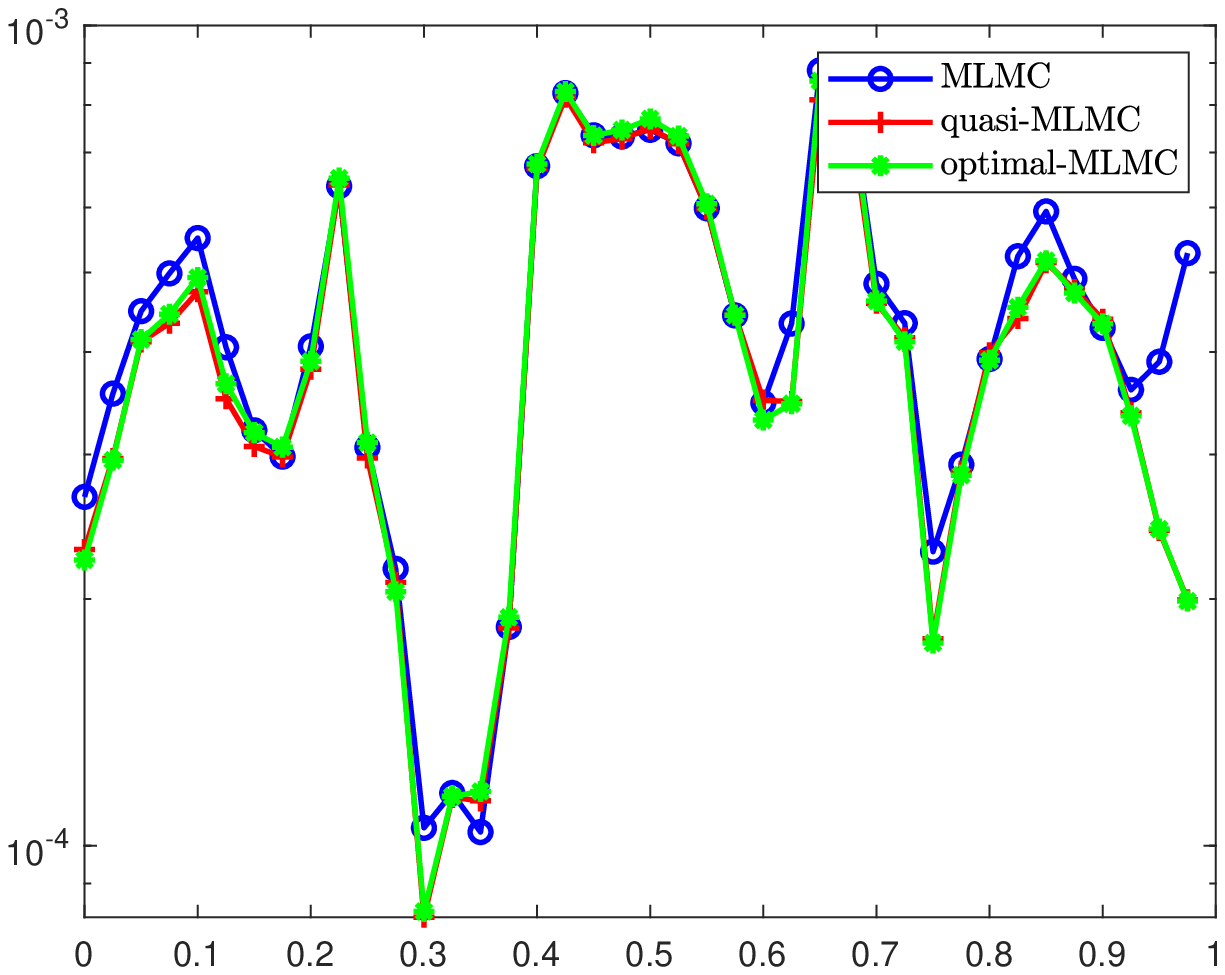}
\includegraphics[width=1.89in]{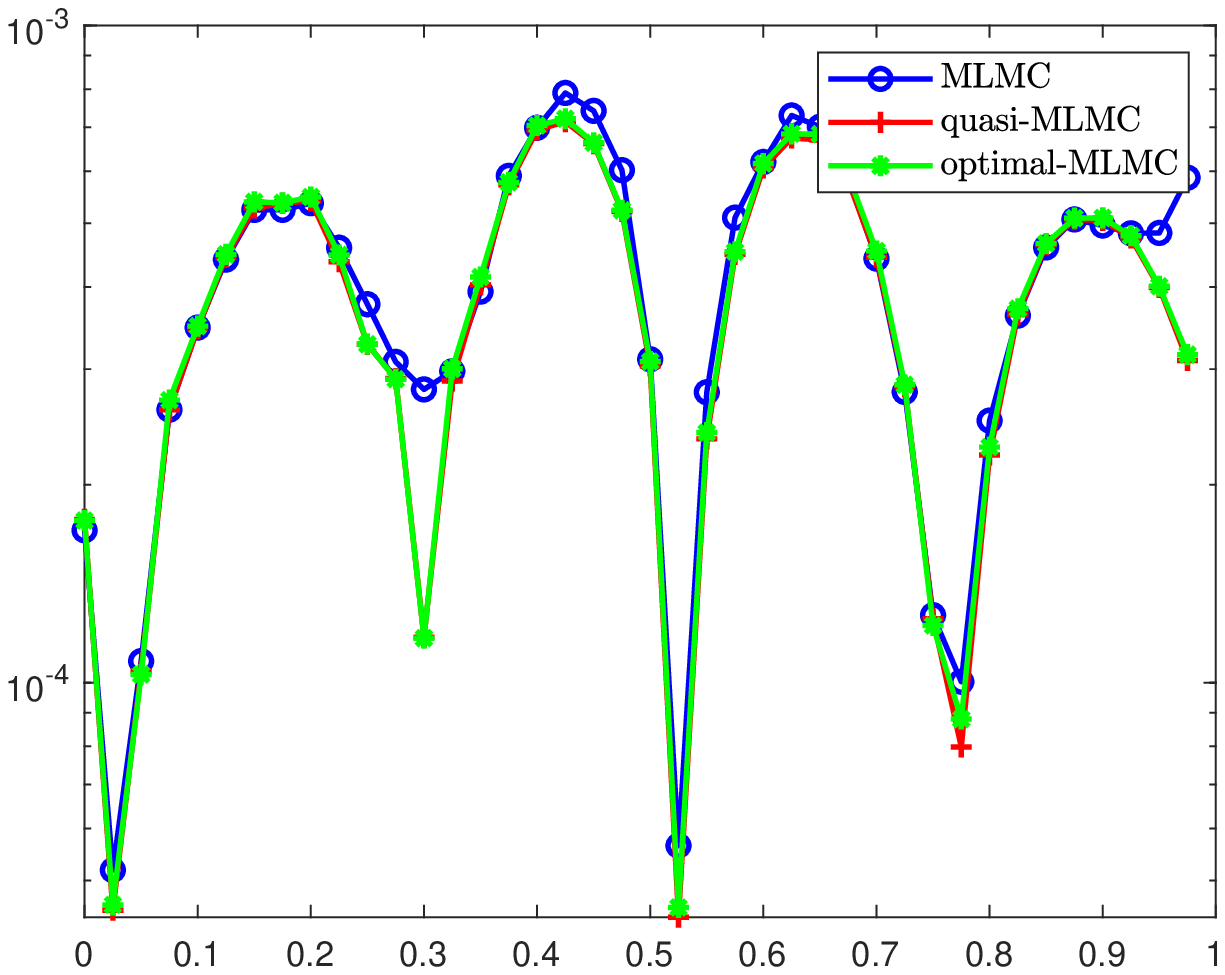}
\includegraphics[width=1.89in]{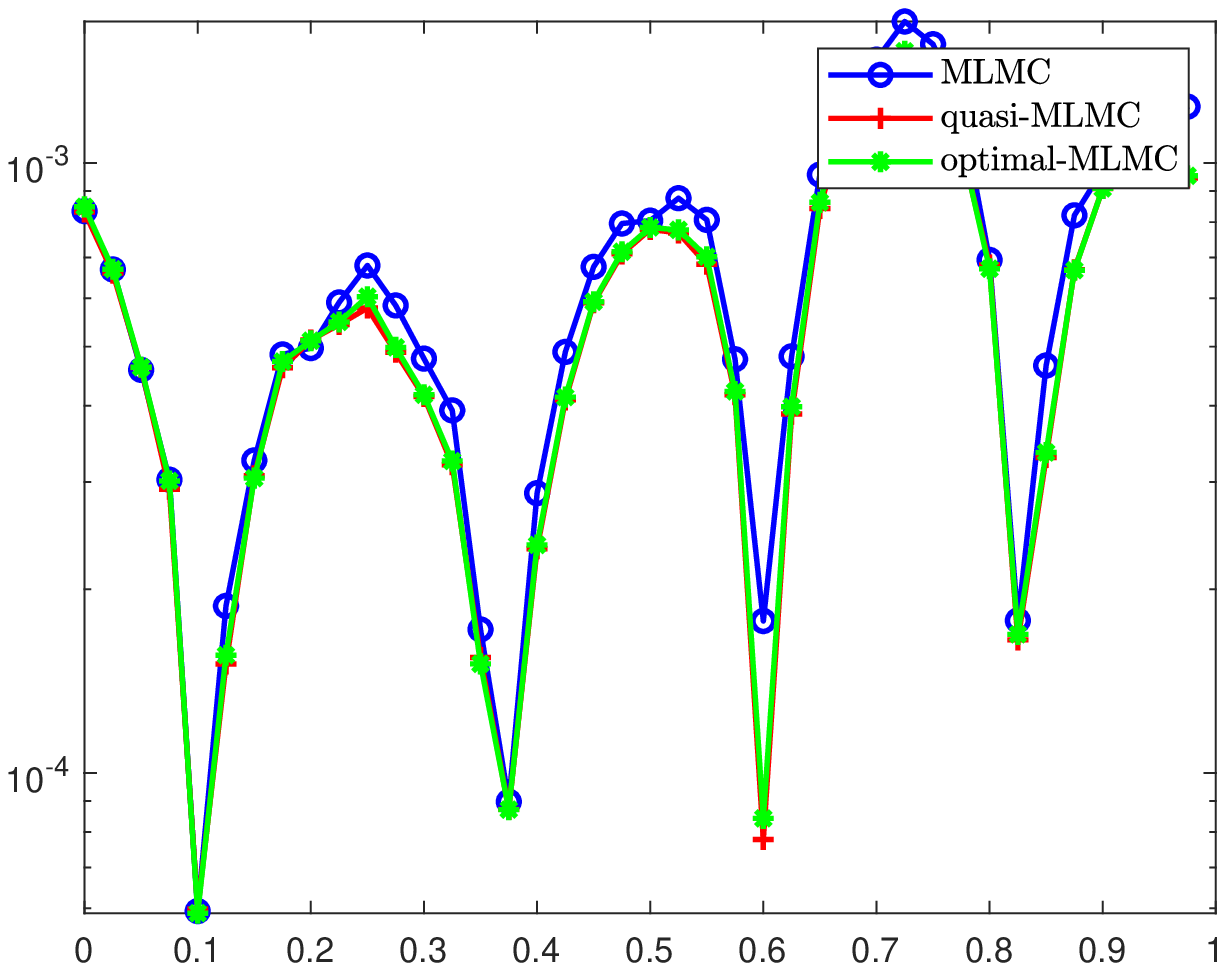}
\includegraphics[width=1.89in]{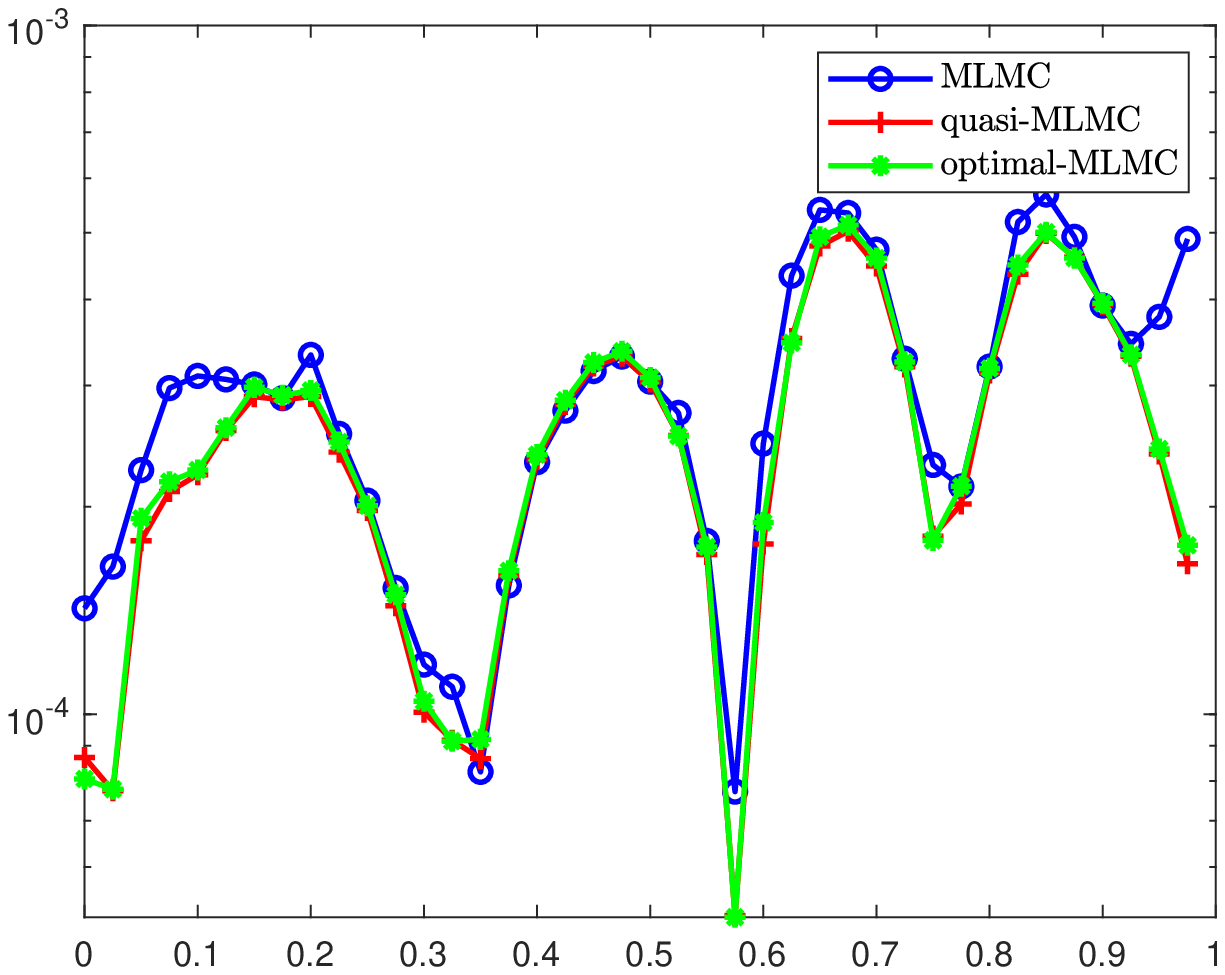}

\caption{Test 1: Approximated expectation of density $\mathbb{E}[\rho]$ (left), velocity $\mathbb{E}[U]$ (middle) and temperature $\mathbb{E}[T]$ (right) using MLMC, quasi-optimal MLMC and optimal MLMC methods at time $t=0.1$ (top row). Error \cref{def:spatialerror} of expectation of density (left), velocity (middle) and temperature (right) using three MLMC methods (middle row). Relative error \cref{def:relativespatialerror} of expectation of density (left), velocity (middle) and temperature (right) using three MLMC methods (bottom row).}
\label{fig:3mlmc}
\end{center}
\end{figure}

To better understand this, we plot the values of $\lambda_1 $ and $\lambda_2$ in the quasi-optimal and optimal MLMC methods in \cref{fig:pblambda}. We can see that almost all values are not far from $1$, which means the methods are not far from the standard MLMC. 

\begin{figure}[tb]
\begin{center}
\includegraphics[width=2.2in]{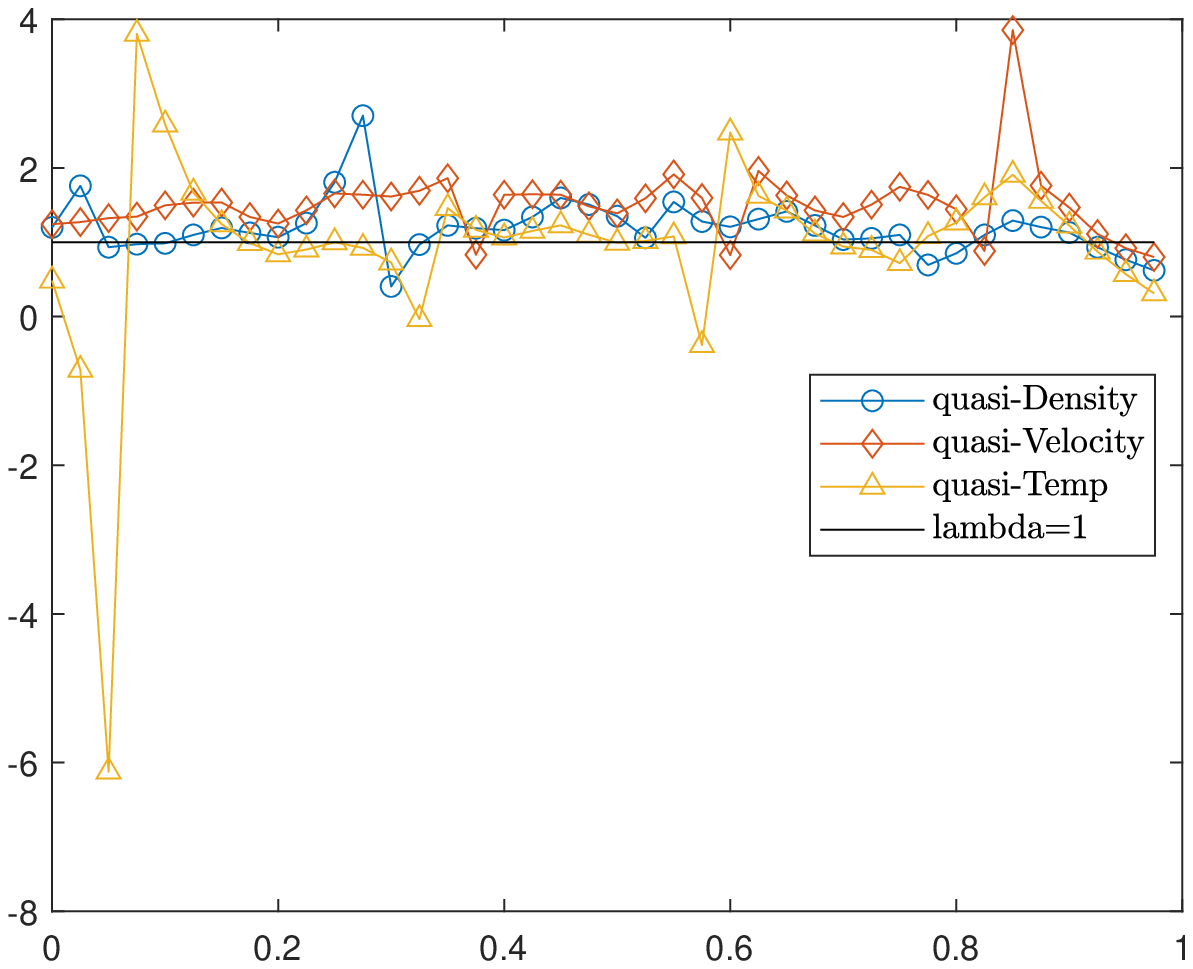}
\includegraphics[width=2.2in]{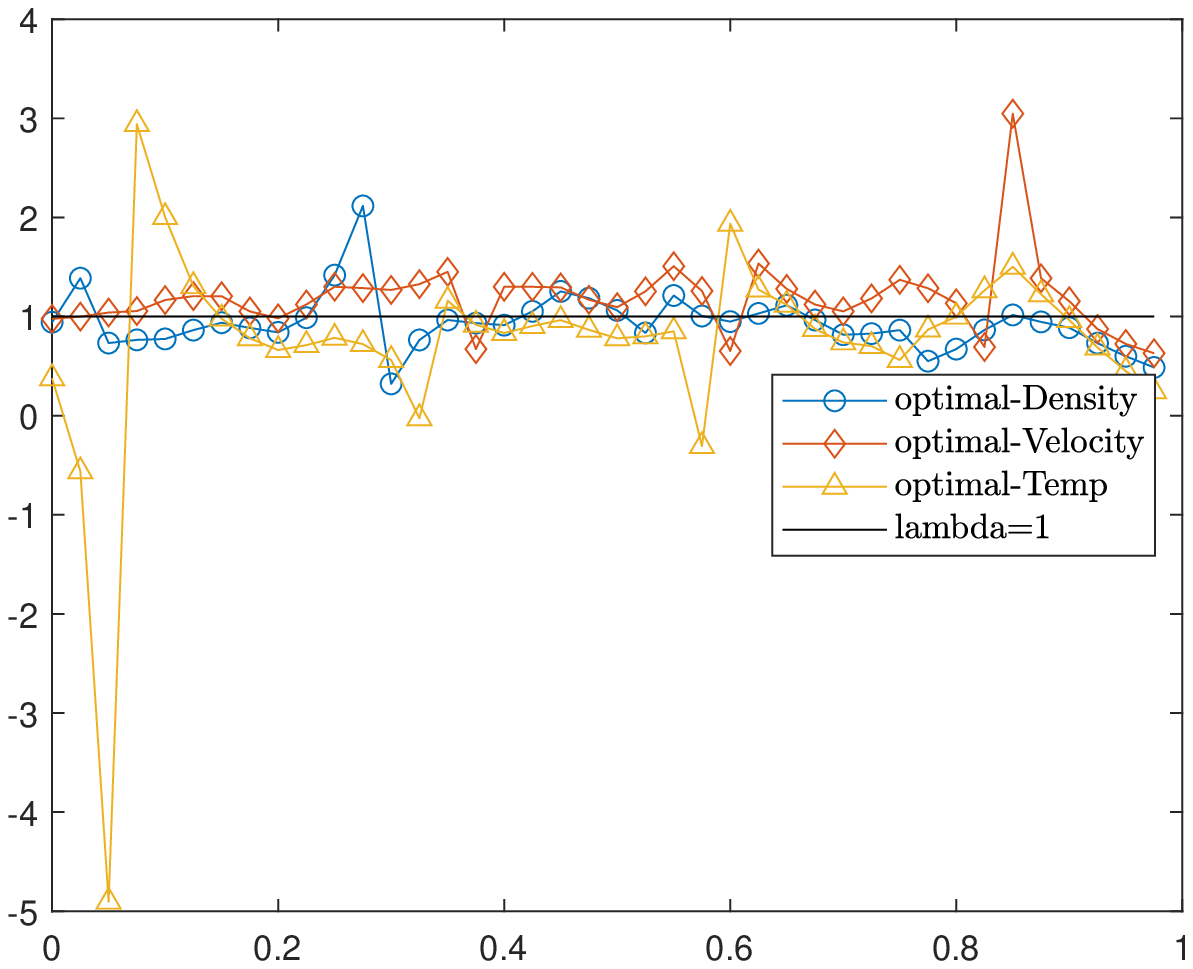}
\includegraphics[width=2.2in]{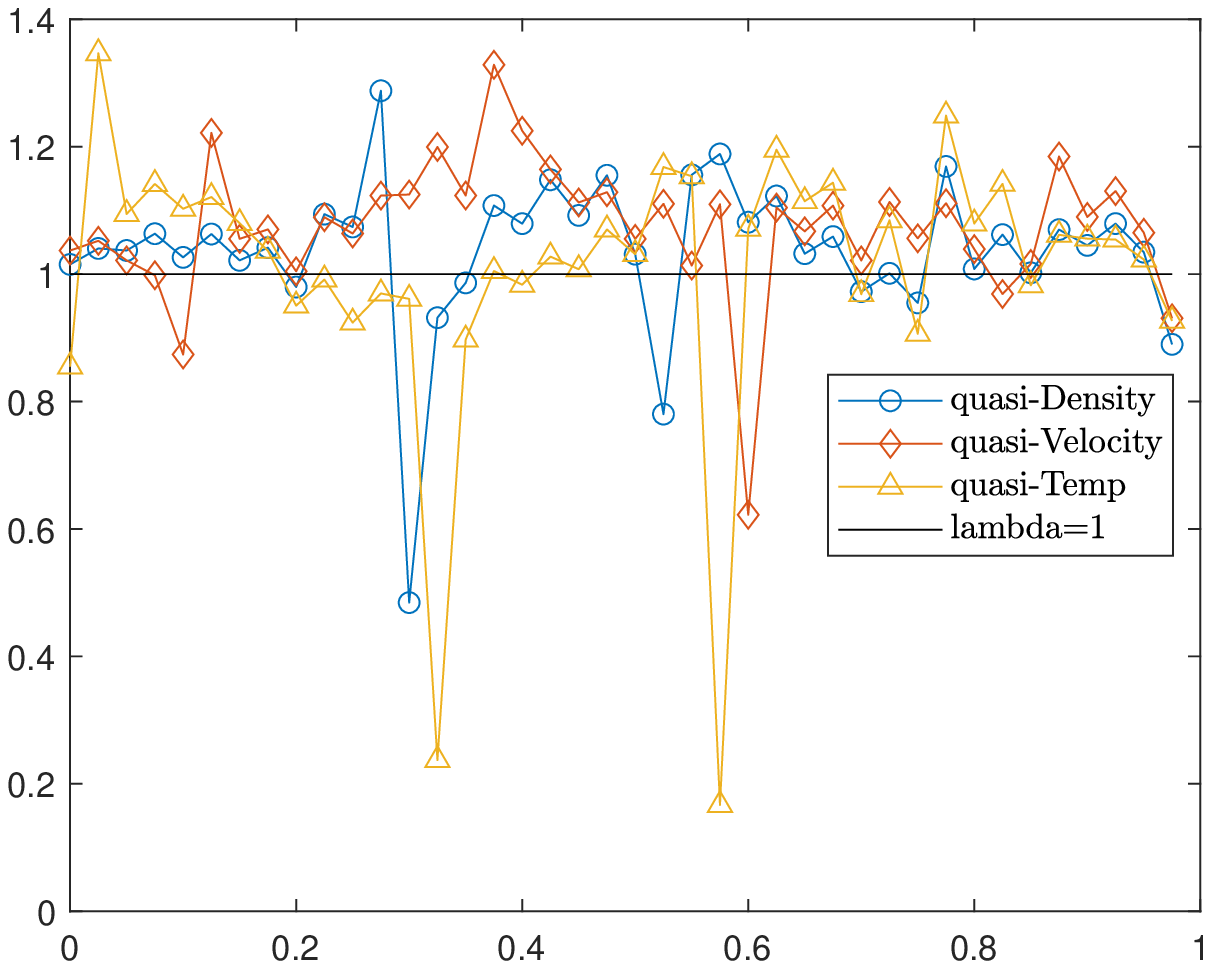}
\includegraphics[width=2.2in]{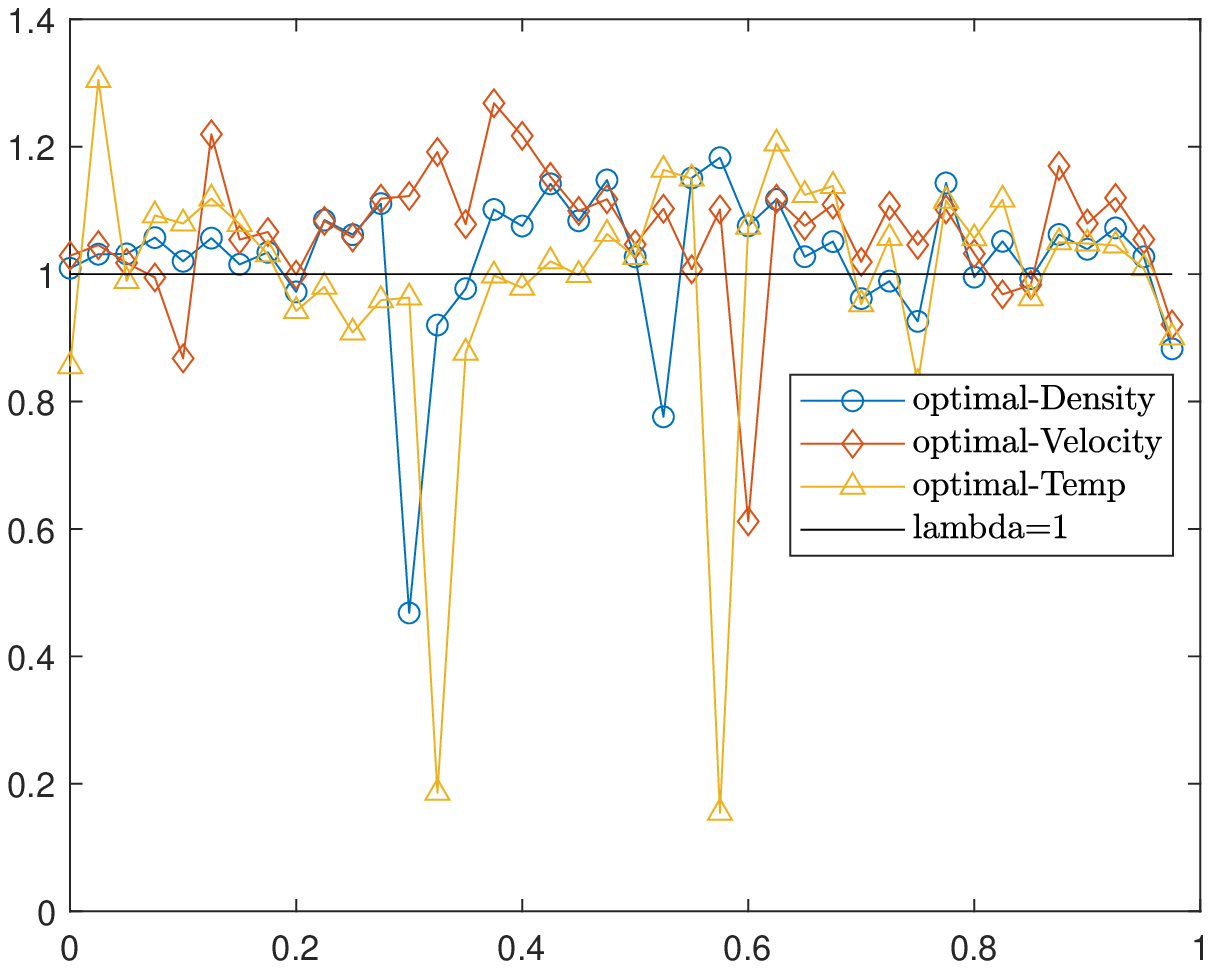}
\caption{Test 1: Values of $\lambda_1$ in quasi-optimal (left) and optimal (right) MLMC methods (top row). Values of $\lambda_2$ in quasi-optimal (left) and optimal (right) MLMC methods (bottom row).}
\label{fig:pblambda}
\end{center}
\end{figure}

\subsection{Test 2: Shock tube problem}

In this test, we consider two kinds of shock tube problems with random initial condition. The first one has uncertainty in the interface location:
\be
I:\left\{
\bal
&\rho_l=1, &\mU_l=(0,0,0)&,\ \ \   T_l=1,& &f_0=M_{\rho_l,\mU_l,T_l}&\quad &x \leq 0.5+0.05z,\\
&\rho_r=0.125, &\mU_r=(0,0,0)&,\ \ \   T_r=0.25,& &f_0=M_{\rho_r,\mU_r,T_r}&\quad &x > 0.5+0.05z.
\eal
\right.
\label{ini:shocktubeini1}
\ee
The second one has uncertainty in the state variables:
\be
II:\left\{
\bal
&\rho_l=1+0.1(z+1), &\mU_l=(0,0,0)&,\ \ \   T_l=1,& &f_0=M_{\rho_l,\mU_l,T_l}&\quad &x \leq 0.5,\\
&\rho_r=0.125, &\mU_r=(0,0,0)&,\ \ \   T_r=0.25,& &f_0=M_{\rho_r,\mU_r,T_r}&\quad &x > 0.5.
\eal
\right.
\label{ini:shocktubeini2}
\ee
The random variable $z$ obeys the uniform distribution on $[-1,1]$. We set the Knudsen number $\varepsilon=1.0e-6$ so that the problem is close to the fluid regime.

Similarly as the previous example, we need to perform a series of tests to determine the number of samples needed in MC and MLMC methods. \cref{fig:mcsttest} shows the analogous tests as those in \cref{fig:mctest}. The main difference from the previous example is that the errors saturate much quicker as the number of samples increases. This is due to the low regularity of the solution so that the error from spatial/temporal discretization dominants easily. In \cref{fig:allsttest} we combine both MC and MLMC results under the scale of workload. Similarly as we observed in \cref{fig:alltest}, with the same workload, the MLMC method can achieve the best accuracy.

\begin{figure}[tb]
\begin{center}
\includegraphics[width=2.3in]{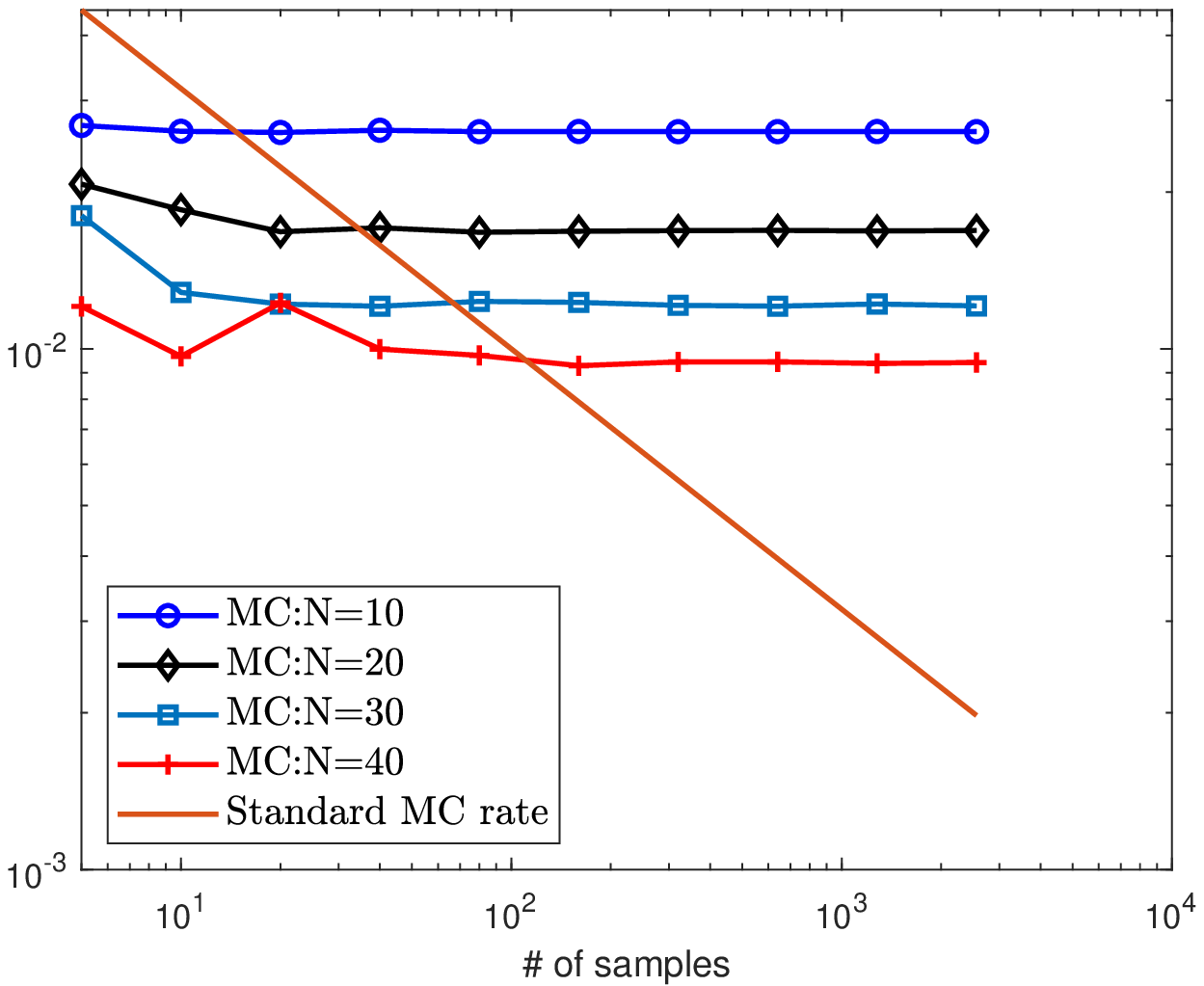}
\includegraphics[width=2.3in]{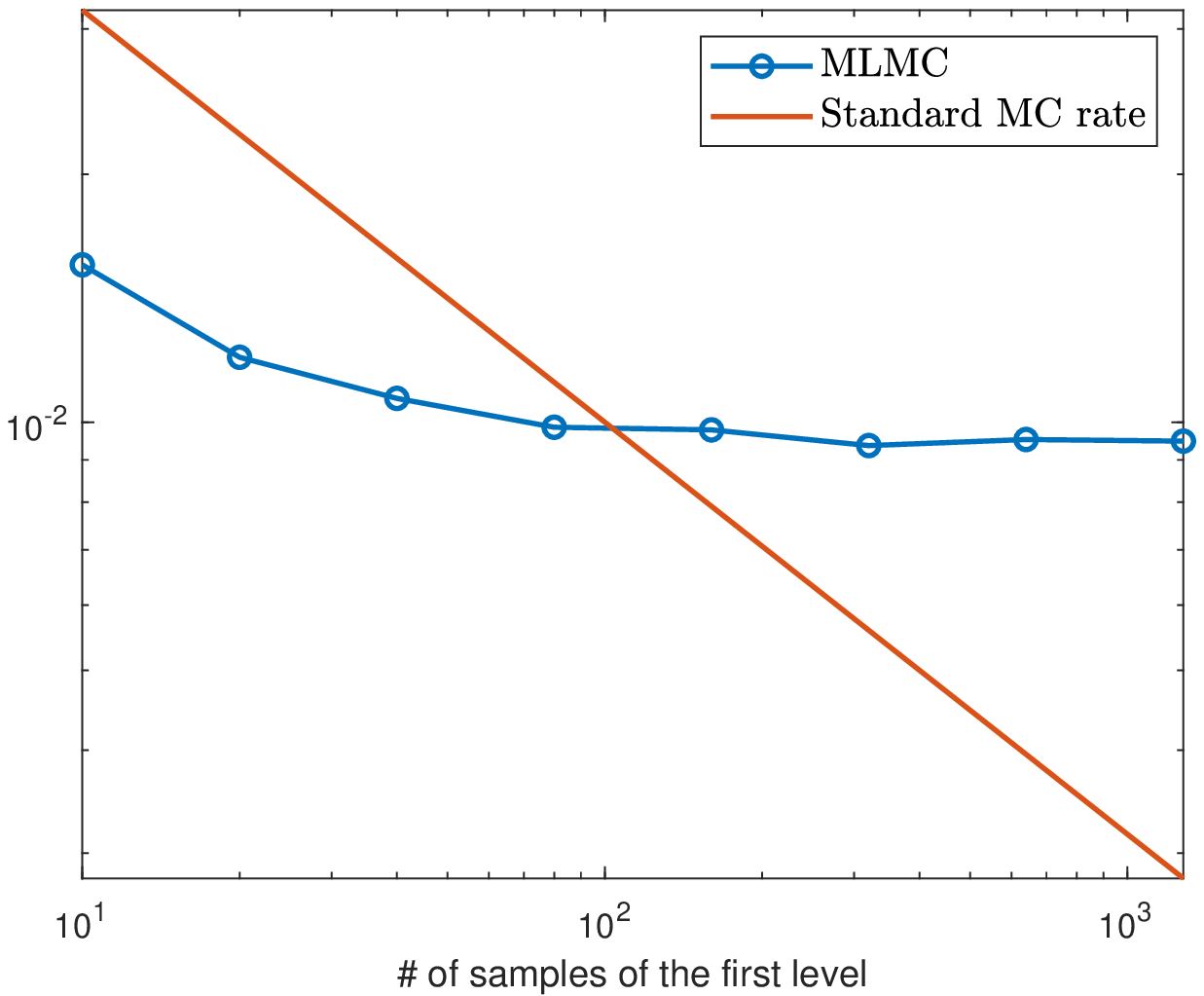}
\caption{Test 2 (I): Error \cref{def:overallerror} (density $\rho$) of MC method (left) and MLMC method (right) v.s. number of samples.}
\label{fig:mcsttest}
\end{center}
\end{figure}

\begin{figure}[!tb]
\begin{center}
\includegraphics[width=2.58in]{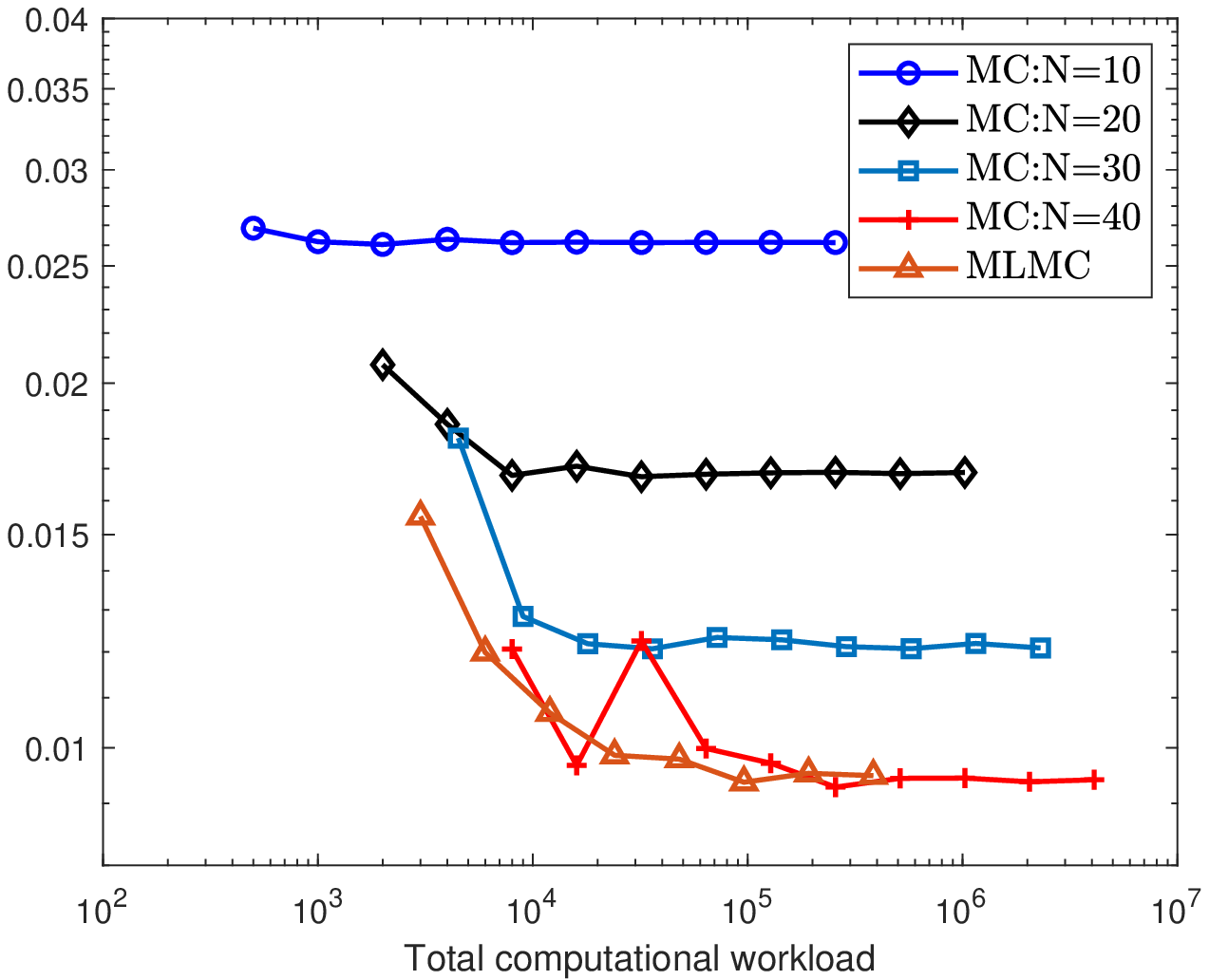}
\caption{Test 2 (I): Error \cref{def:overallerror} (density $\rho$) of MC and MLMC methods v.s. computational workload.}
\label{fig:allsttest}
\end{center}
\end{figure}

From the right plot in \cref{fig:mcsttest}, we also see that $M_1\approx 320$ is the best number of samples at level 1 of MLMC method. Therefore, we fix the set of parameters in the following tests of MLMC methods: mesh sizes  $N_1=10$, $N_2=20$, $N_3=40$, and sample sizes $M_1=320$, $M_2=80$, $M_3=20$. In Figures \ref{fig:st1exp}-\ref{fig:st2var}, we report the results obtained using standard MLMC, quasi-optimal MLMC, and optimal MLMC methods. We examine the approximations to both expectation $\mathbb{E}[q]$ and variance $\mathbb{V}[q]$ of the macroscopic quantities $\rho$, $U$ and $T$. Note that the MLMC methods are based on the linearity of the expectation operator, not the variance operator. Hence to approximate the variance, we approximate separately two different expectations $\mathbb{E}[q^2]$ and $\mathbb{E}[q]$ and use them to obtain $\mathbb{V}[q]=\mathbb{E}[q^2]-(\mathbb{E}[q])^2$. The results clearly show that both control variate MLMC methods outperform the standard MLMC in regions where the solution presents strong variations, namely close to the shock position. On the other hand, we did not observe any relevant gain using the optimal MLMC method over the quasi-optimal MLMC.

\begin{figure}[tb]
\begin{center}
\includegraphics[width=1.89in]{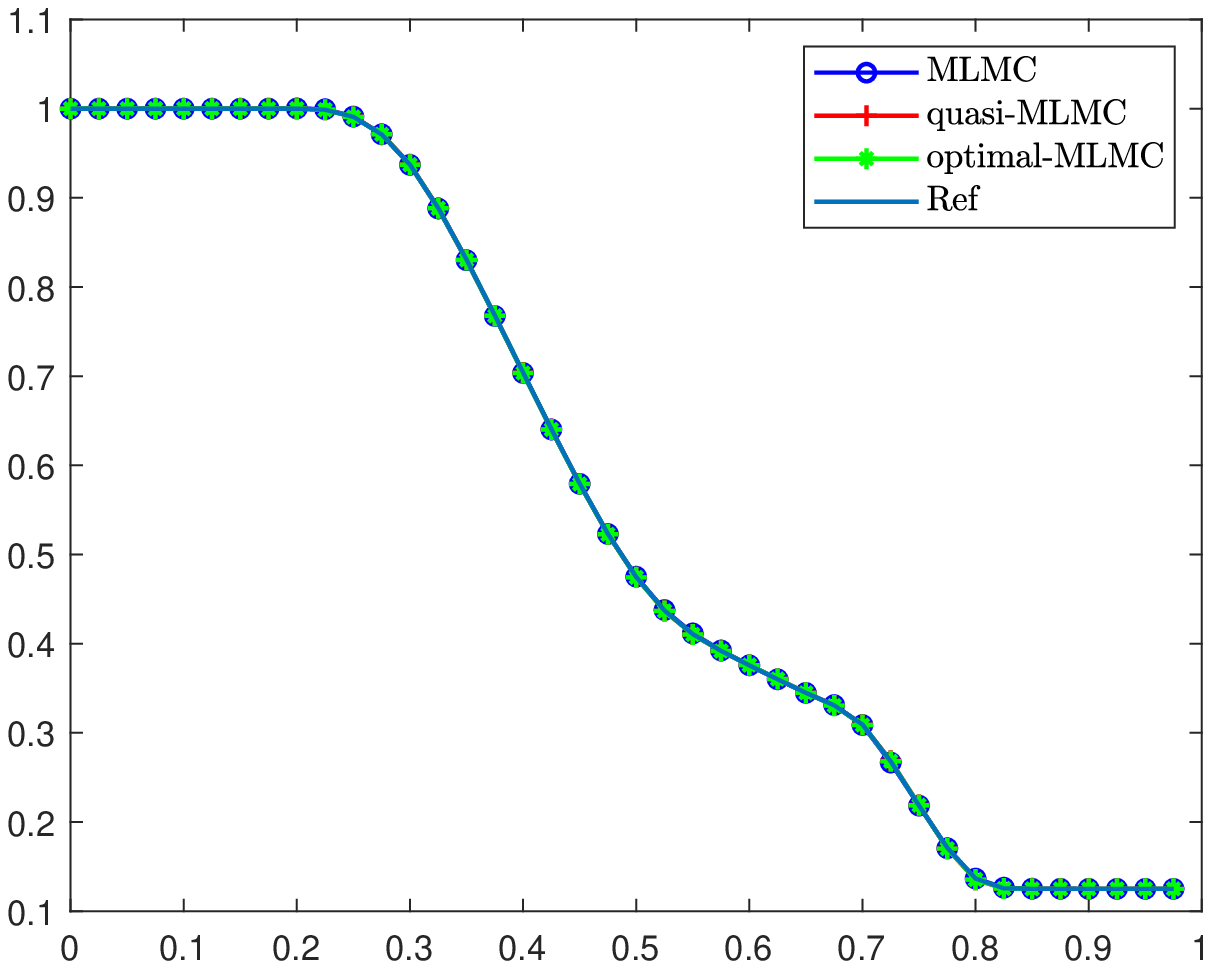}
\includegraphics[width=1.89in]{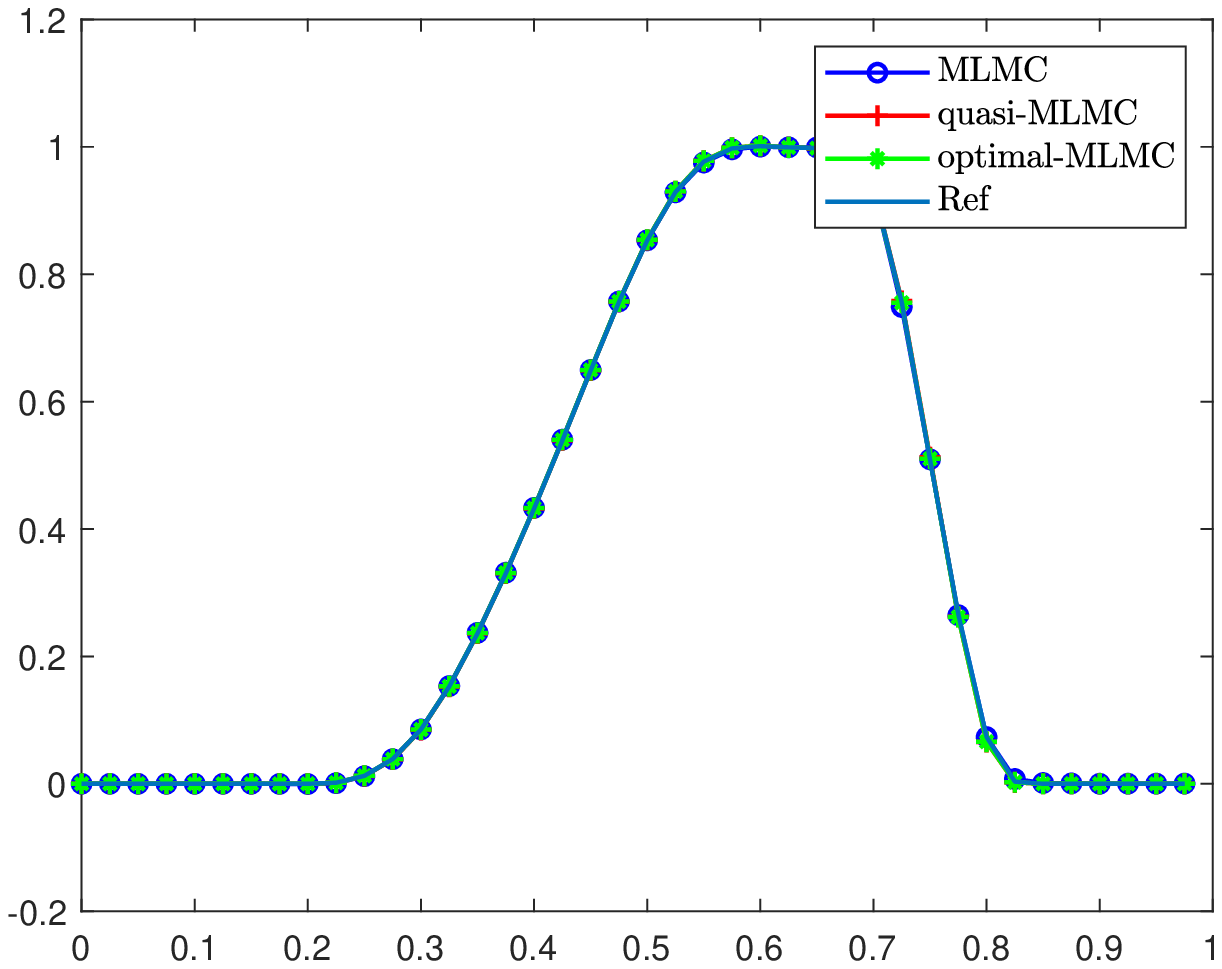}
\includegraphics[width=1.89in]{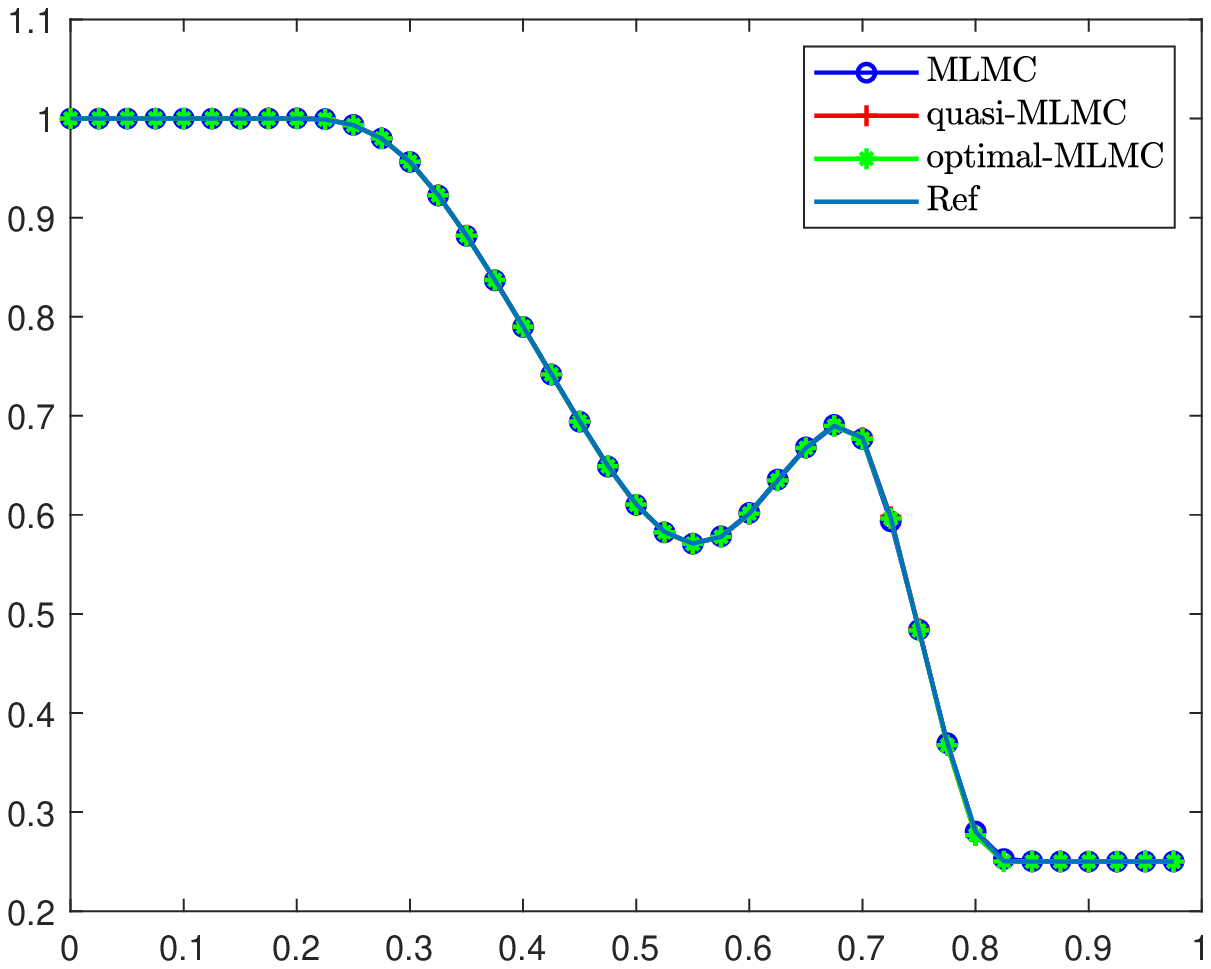}
\includegraphics[width=1.89in]{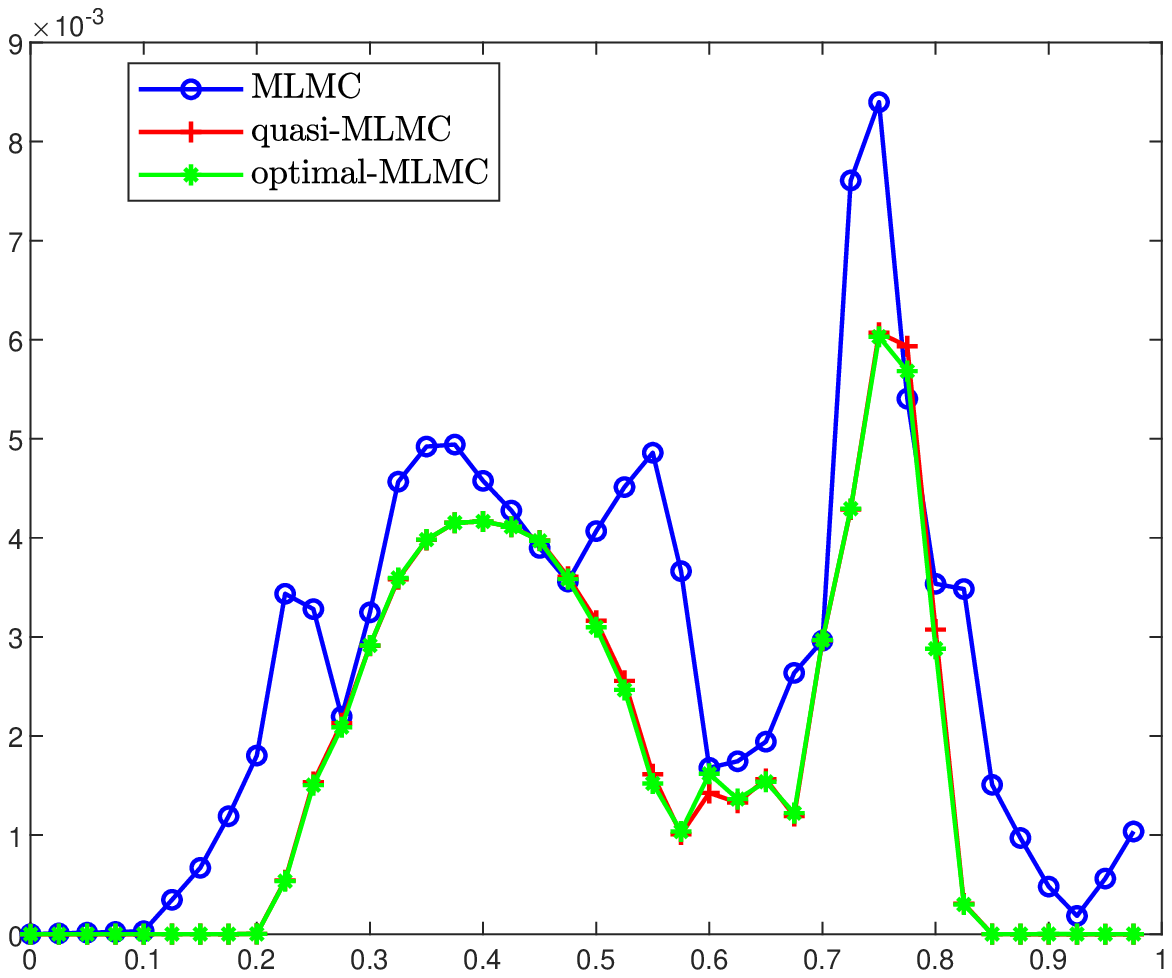}
\includegraphics[width=1.89in]{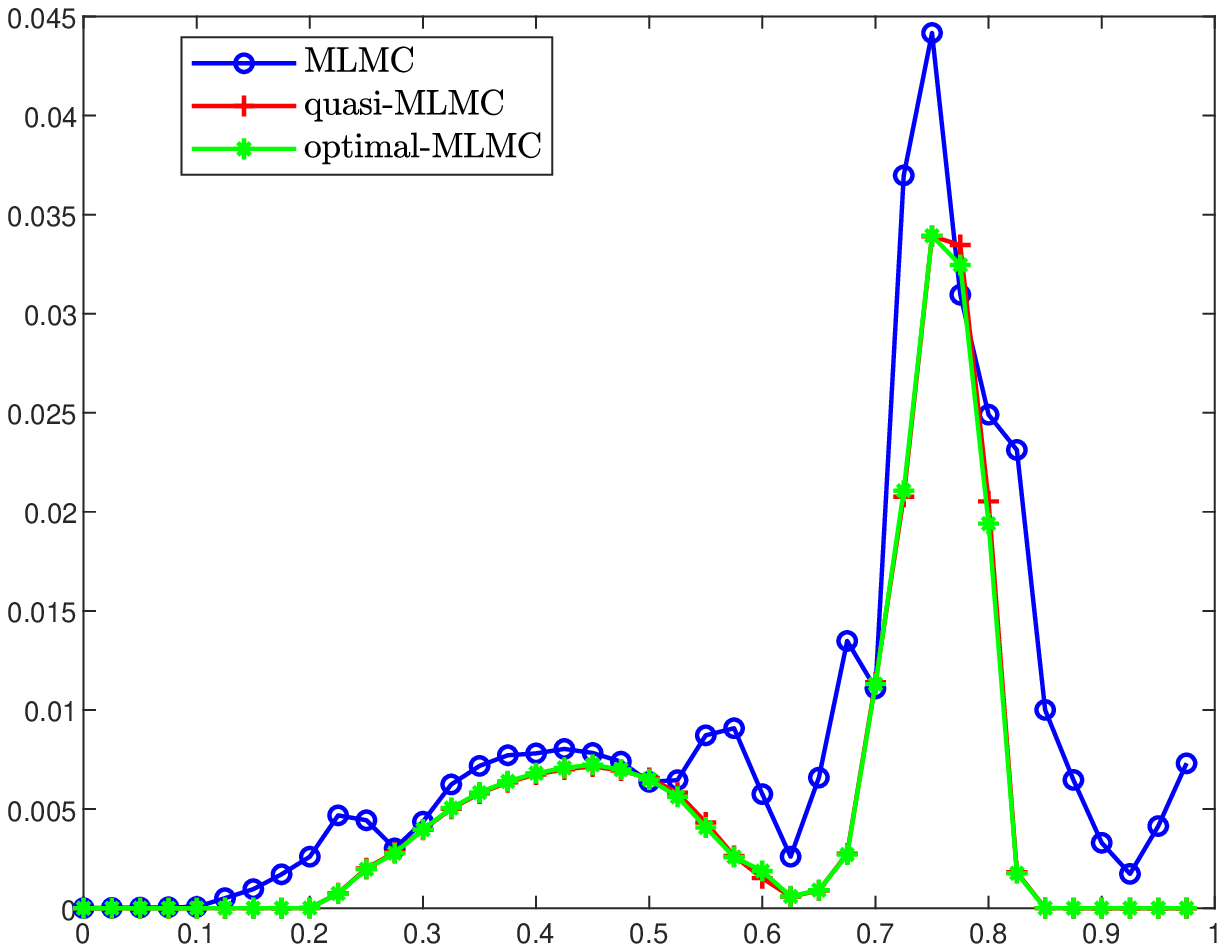}
\includegraphics[width=1.89in]{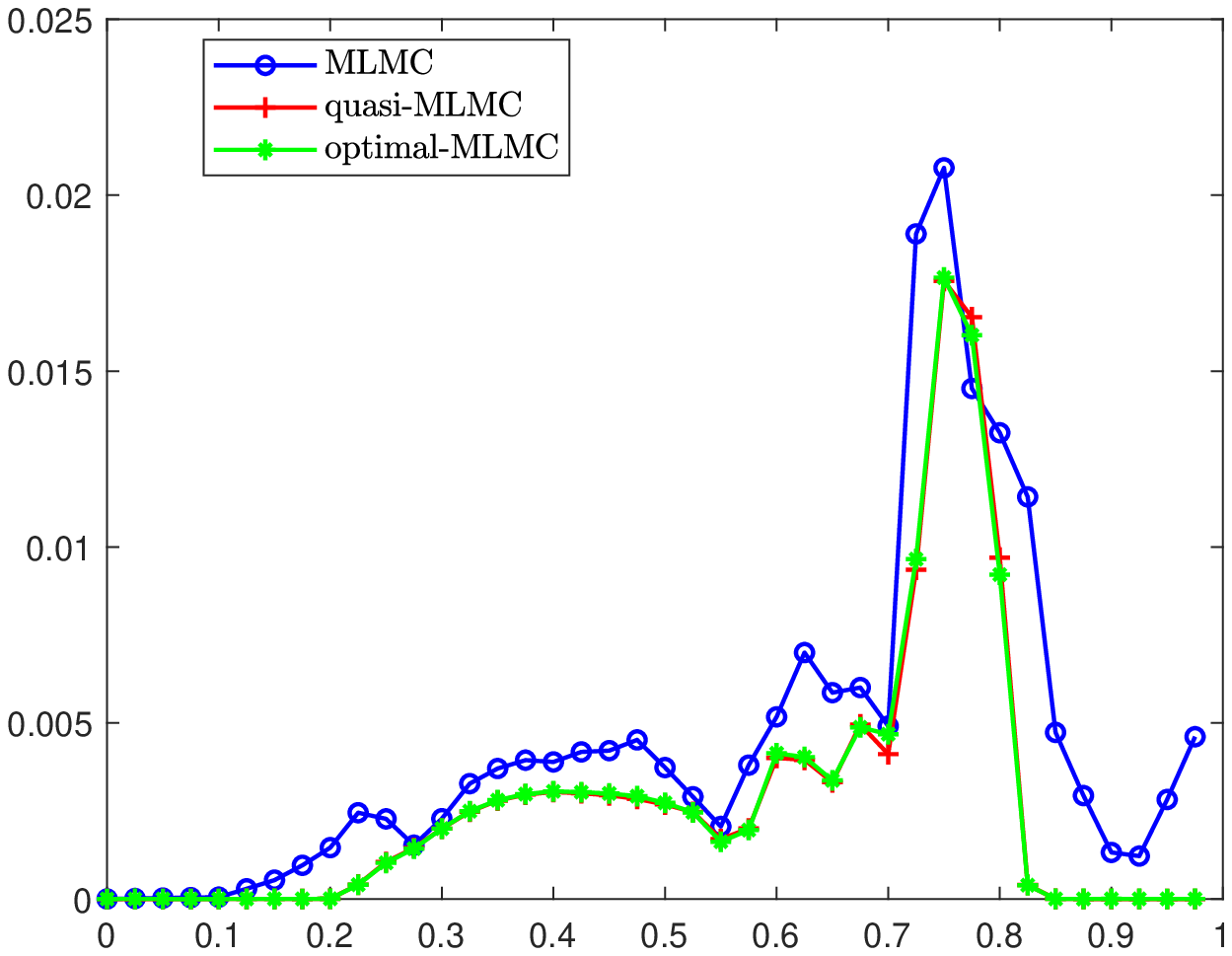}
\caption{Test 2 (I):  Approximated expectation of density $\mathbb{E}[\rho]$ (left), velocity $\mathbb{E}[U]$ (middle) and temperature $\mathbb{E}[T]$ (right) using MLMC, quasi-optimal MLMC and optimal MLMC methods at time $t=0.15$ (top row). Relative error \cref{def:relativespatialerror} of expectation of density (left), velocity (middle) and temperature (right) using three MLMC methods (bottom row).}
\label{fig:st1exp}
\end{center}
\end{figure}

\begin{figure}[tb]
\begin{center}
\includegraphics[width=1.89in]{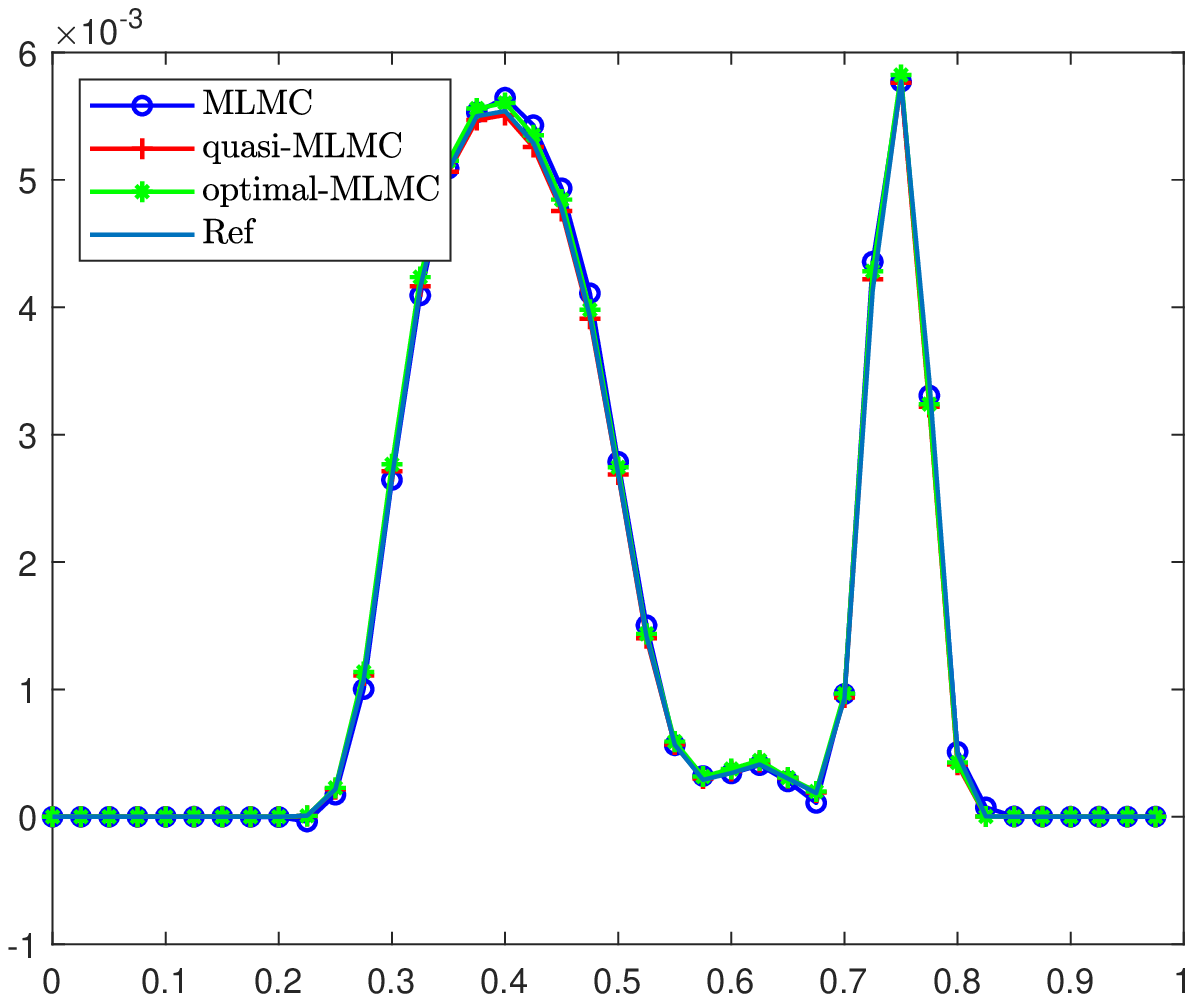}
\includegraphics[width=1.89in]{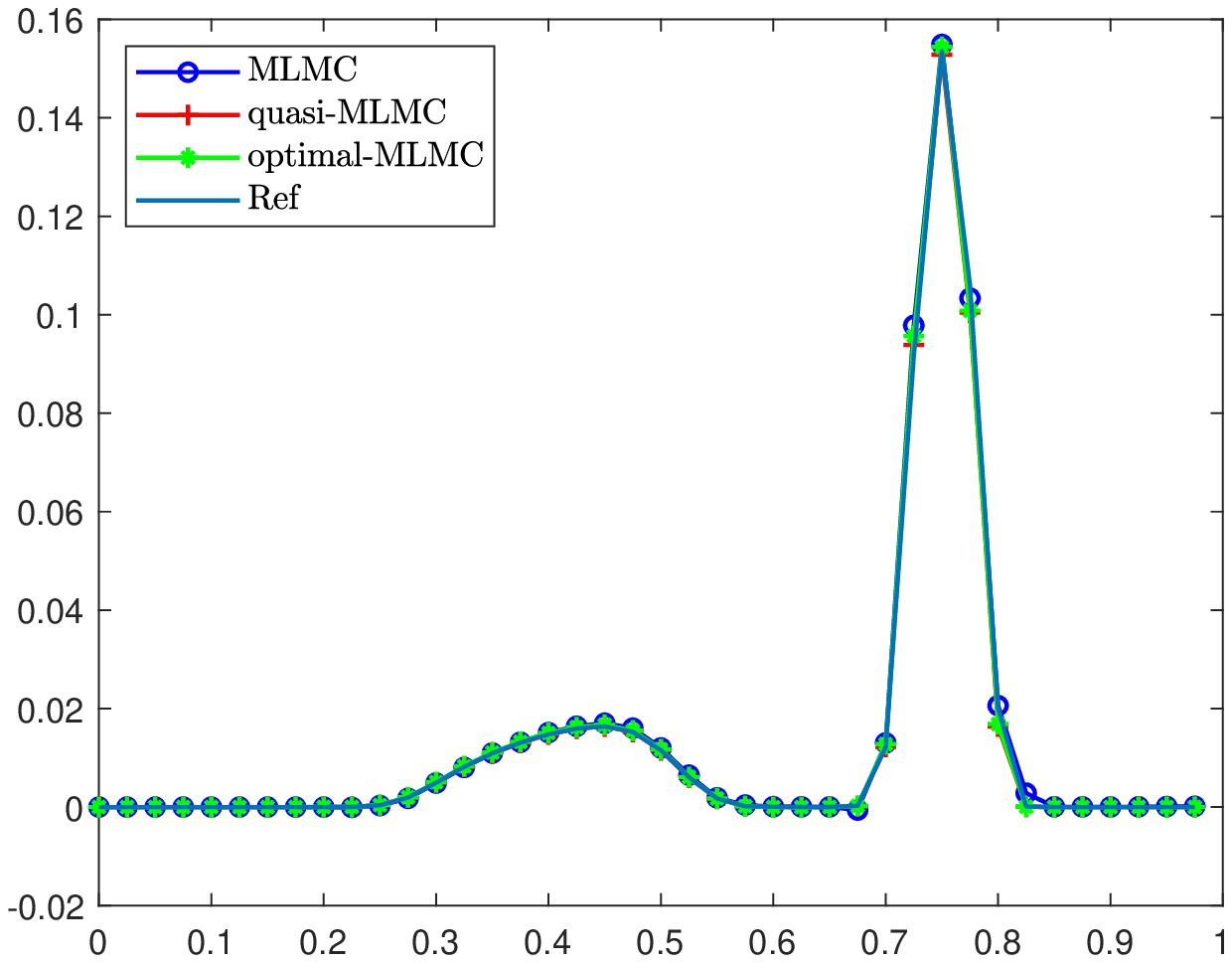}
\includegraphics[width=1.89in]{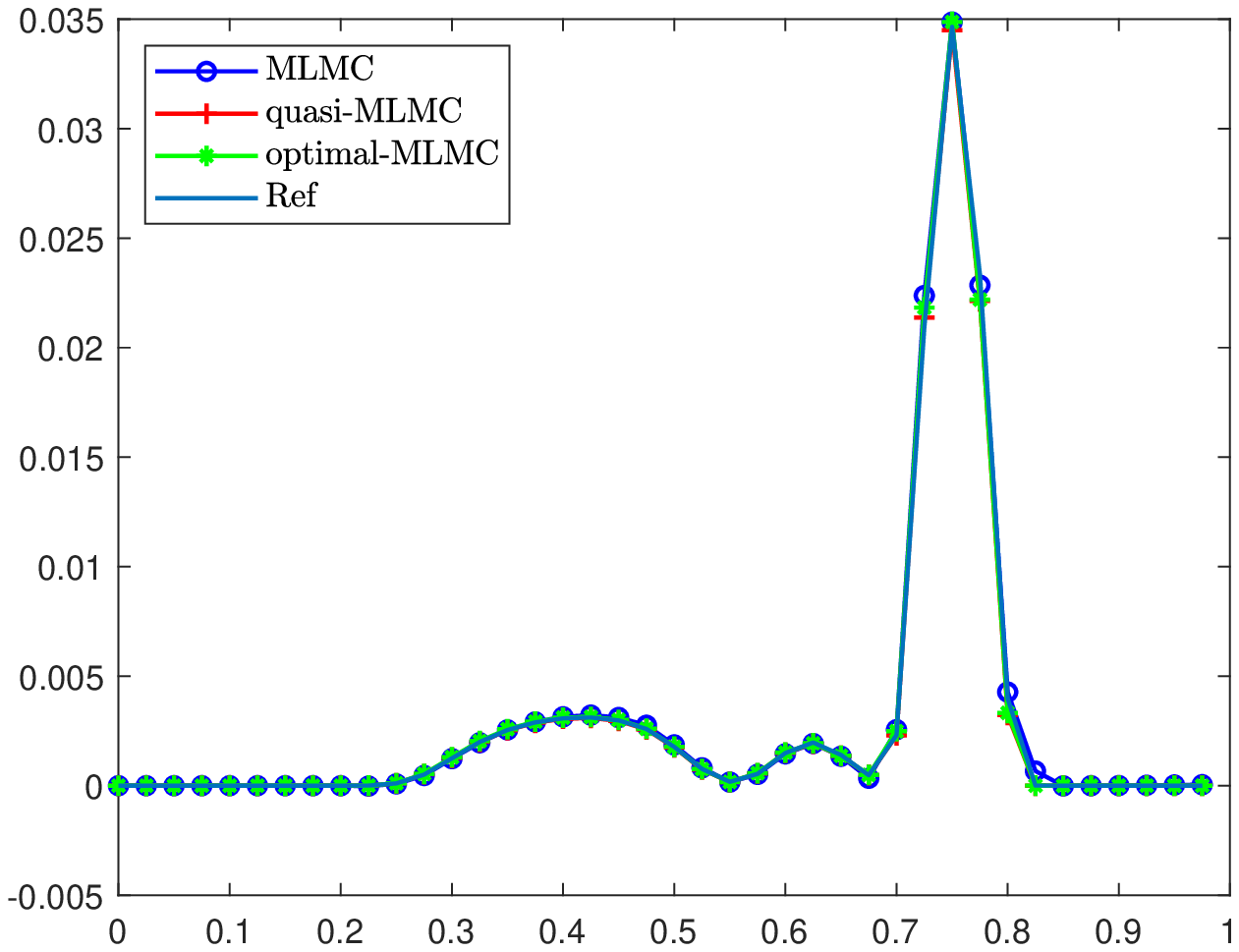}
\includegraphics[width=1.89in]{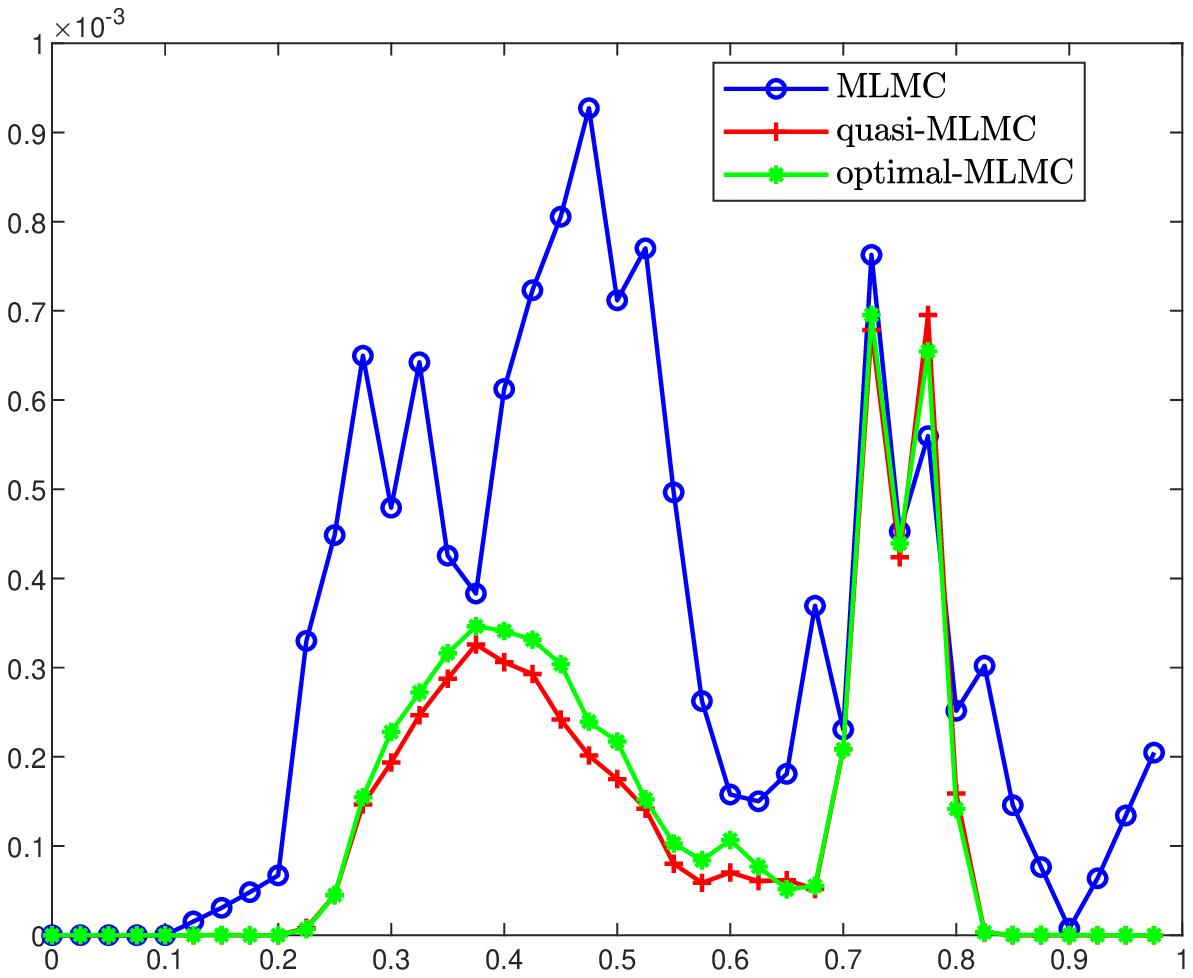}
\includegraphics[width=1.89in]{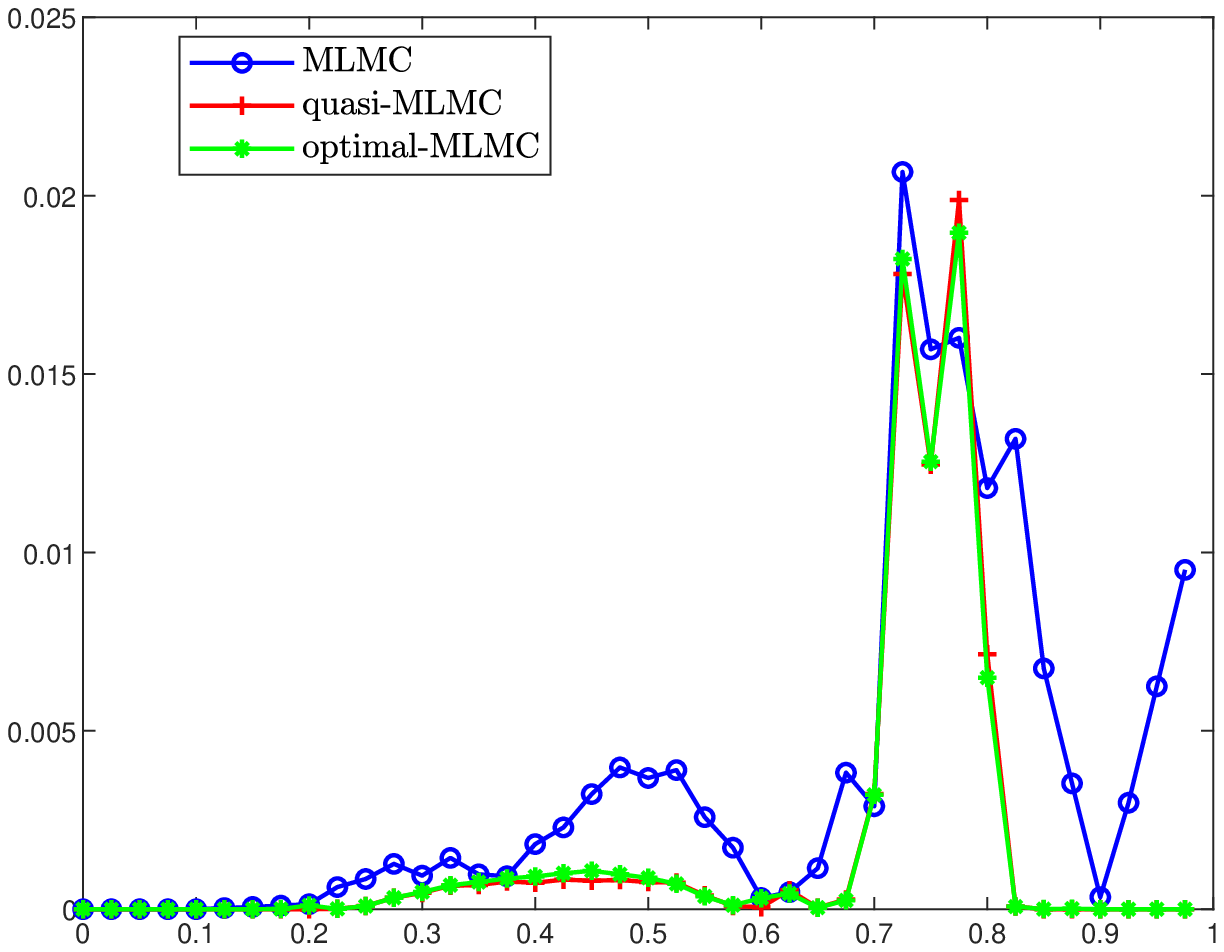}
\includegraphics[width=1.89in]{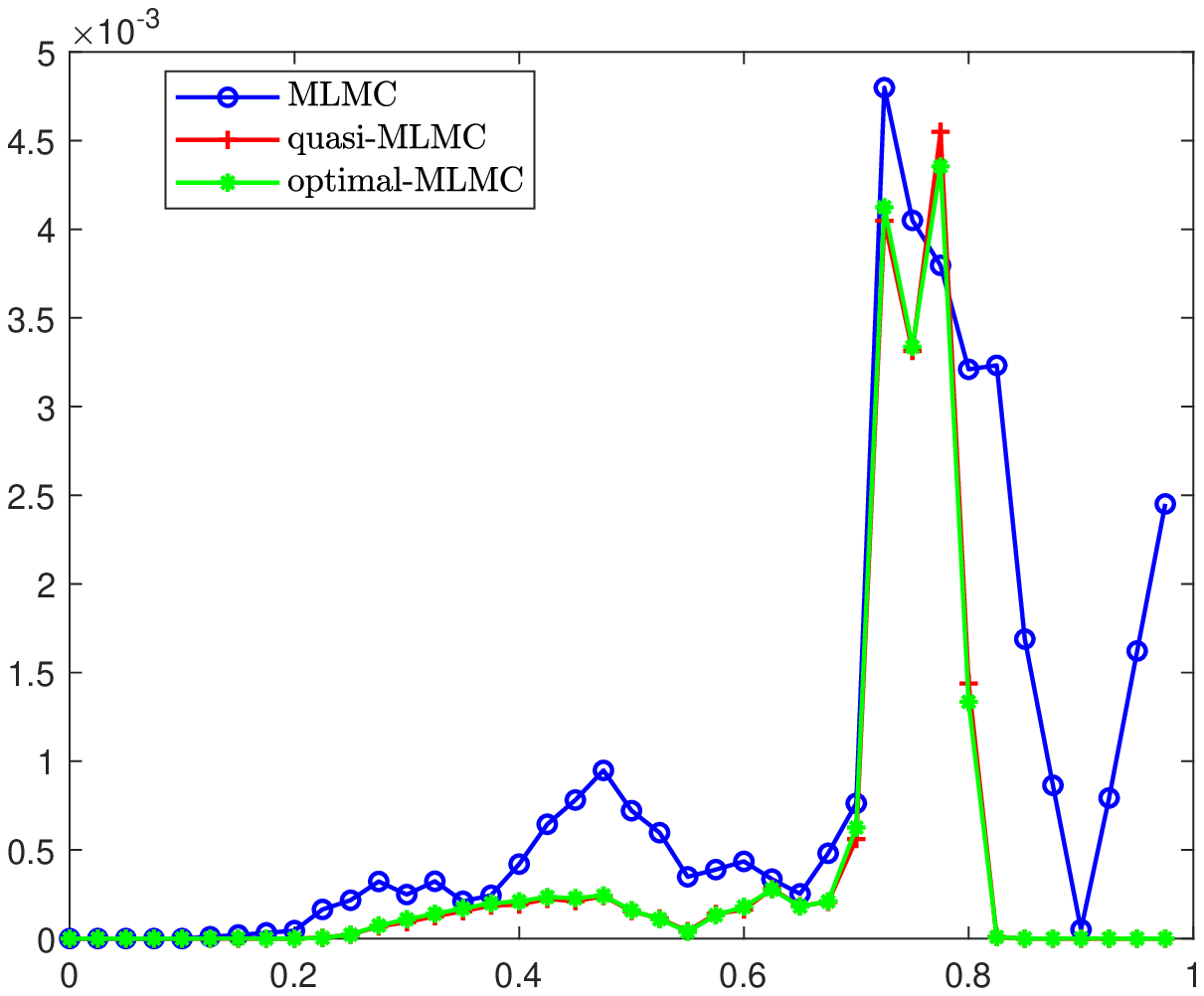}
\caption{Test 2 (I): Approximated variance of density $\mathbb{V}[\rho]$ (left), velocity $\mathbb{V}[U]$ (middle) and temperature $\mathbb{V}[T]$ (right) using MLMC, quasi-optimal MLMC and optimal MLMC methods at time $t=0.15$ (top row). Relative error \cref{def:relativespatialerror} of variance of density (left), velocity (middle) and temperature (right) using three methods (bottom row).}
\label{fig:st1var}
\end{center}
\end{figure}

\begin{figure}[tb]
\begin{center}
\includegraphics[width=1.89in]{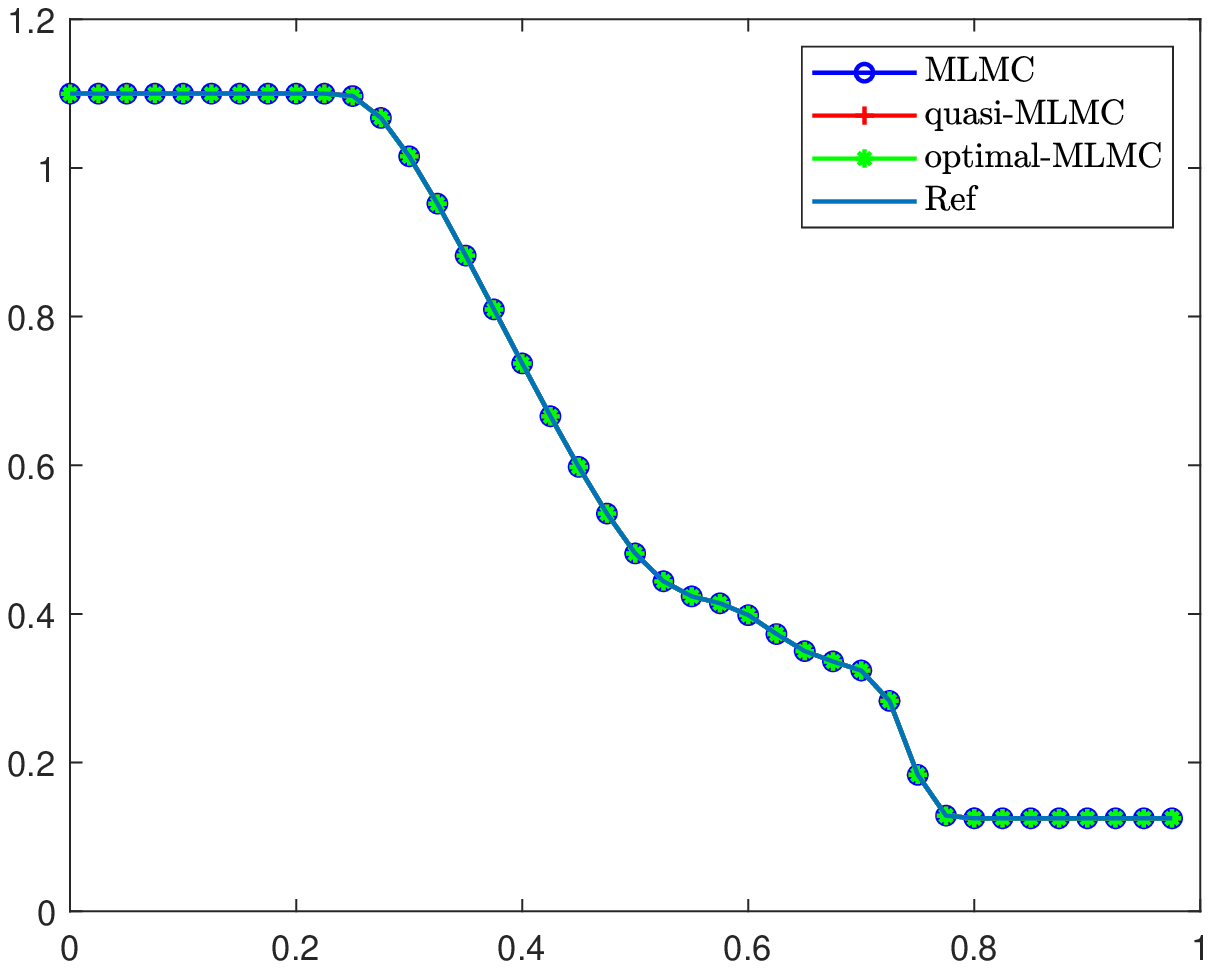}
\includegraphics[width=1.89in]{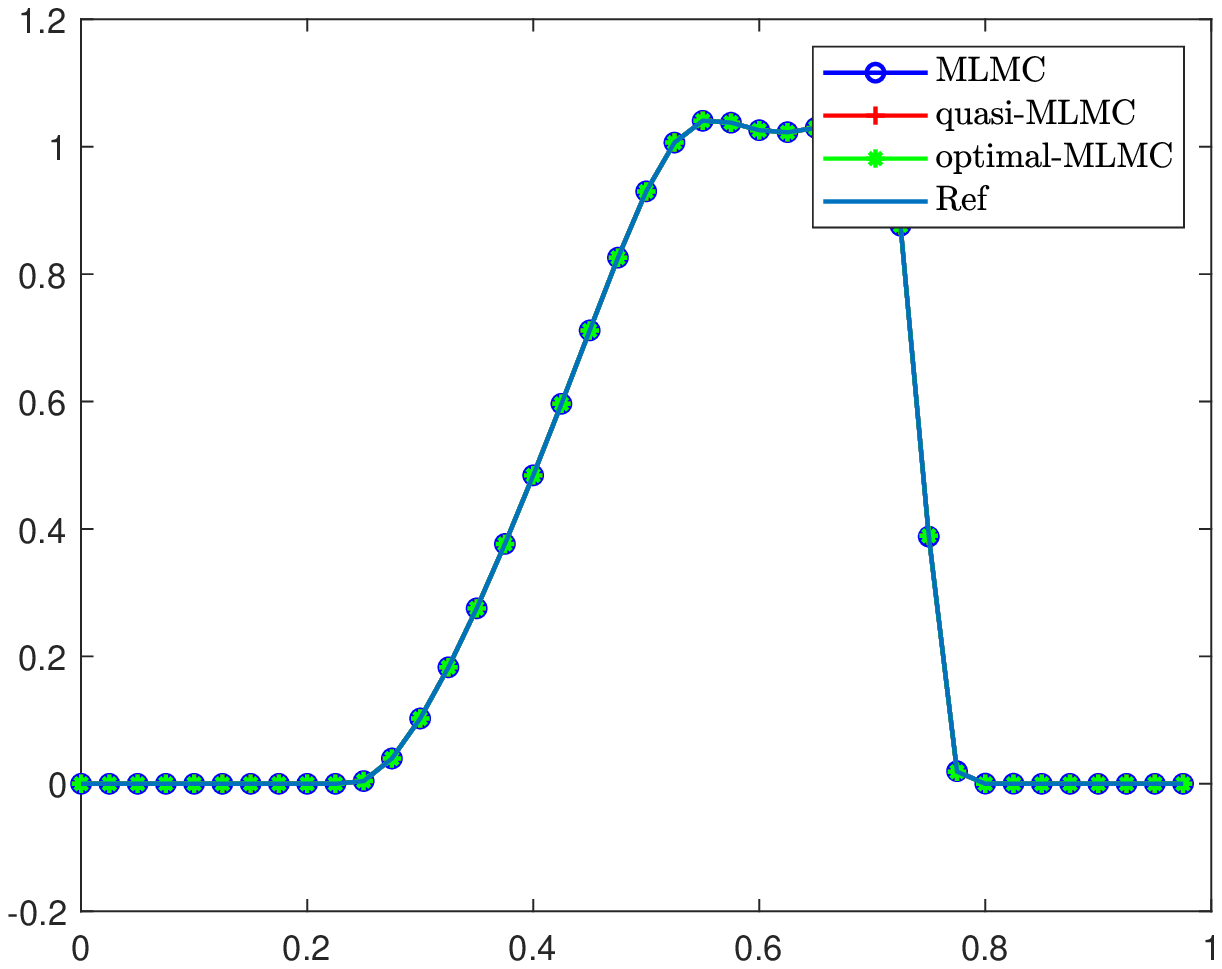}
\includegraphics[width=1.89in]{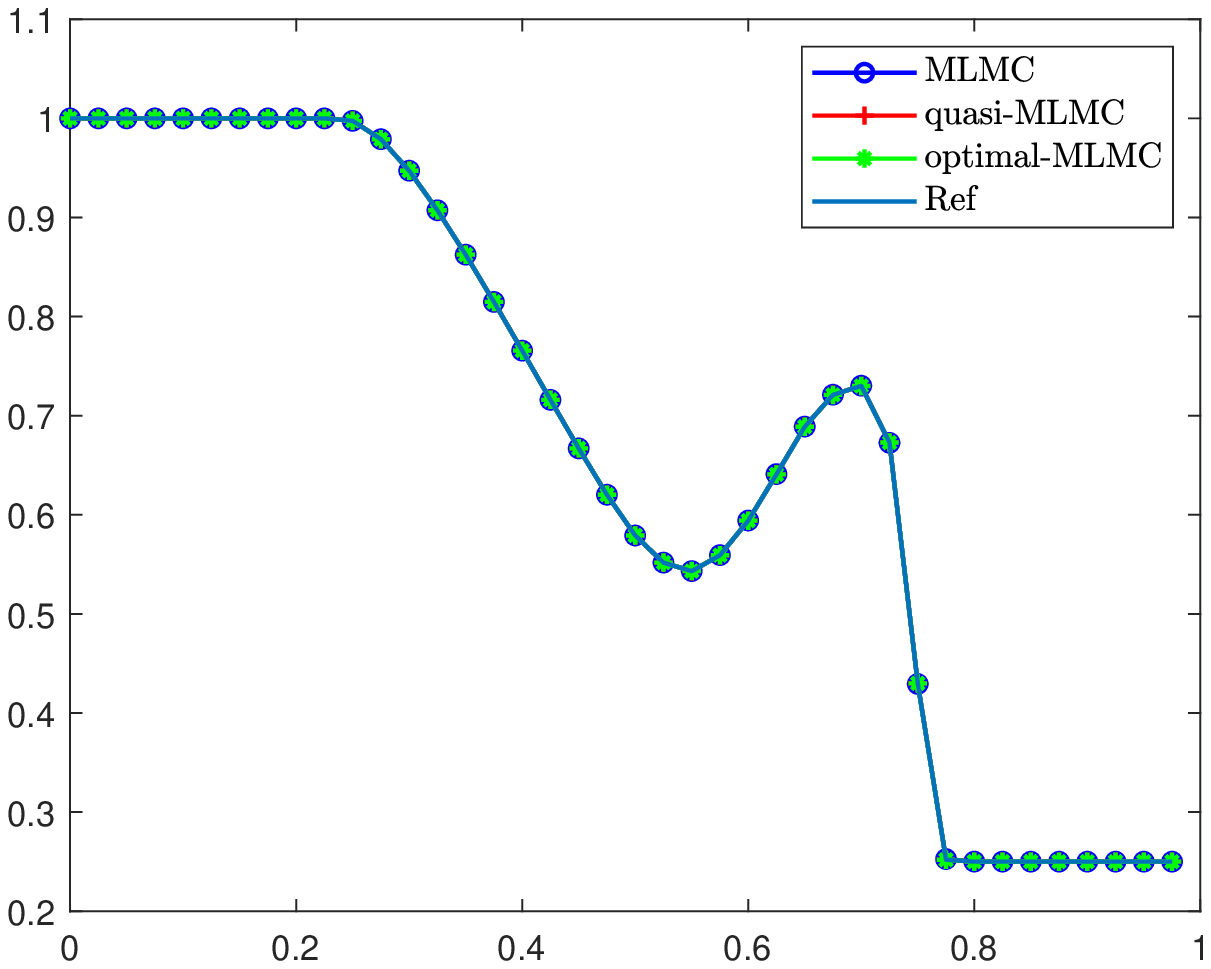}
\includegraphics[width=1.89in]{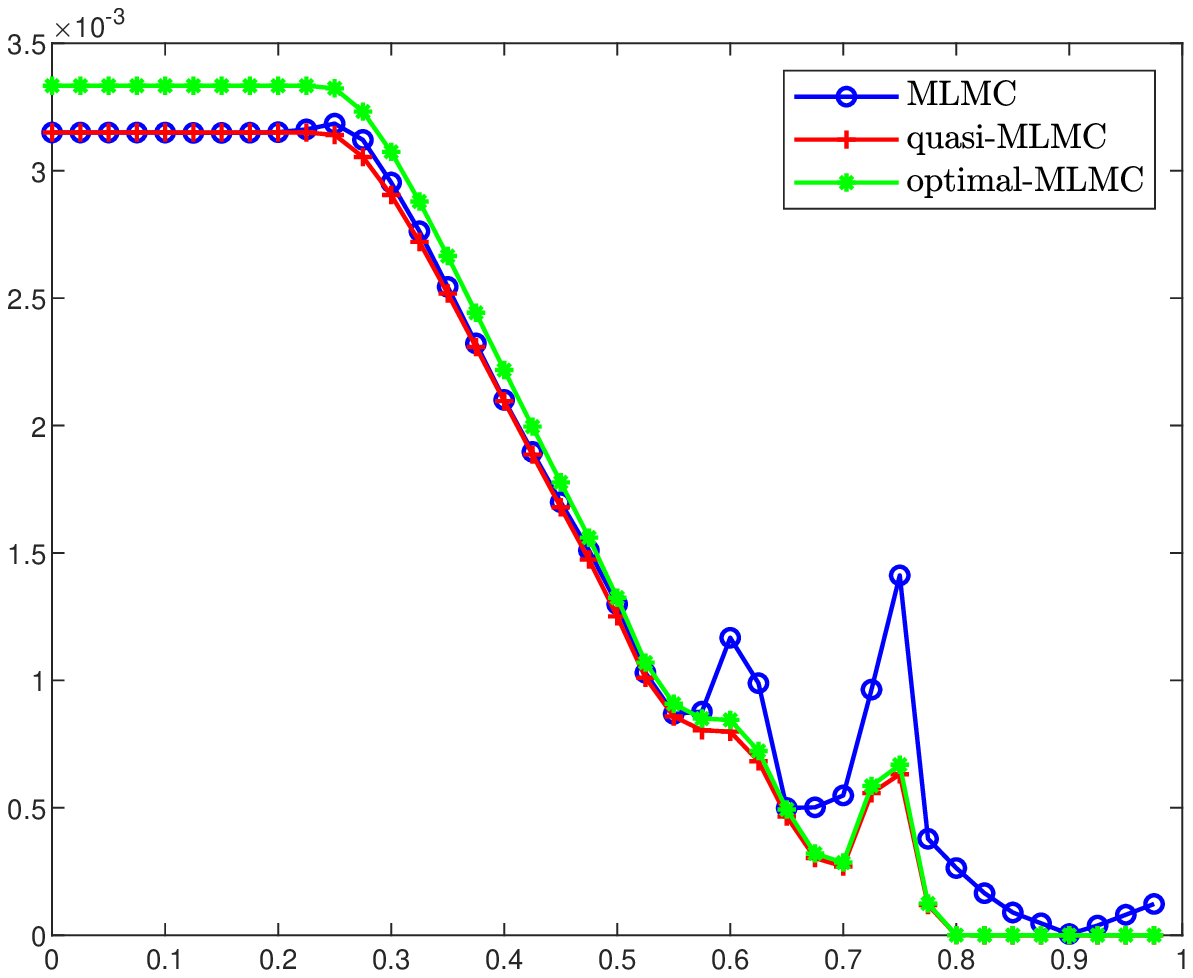}
\includegraphics[width=1.89in]{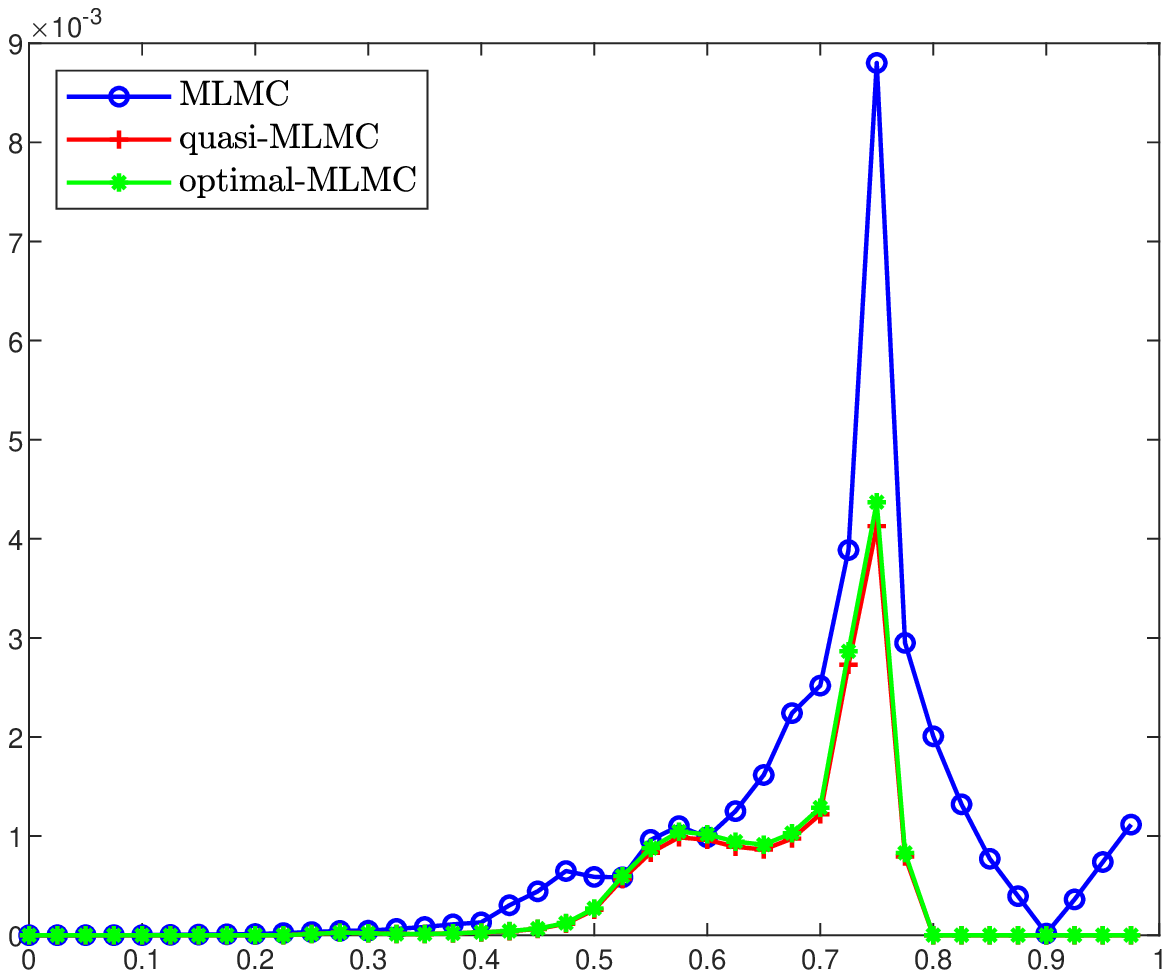}
\includegraphics[width=1.89in]{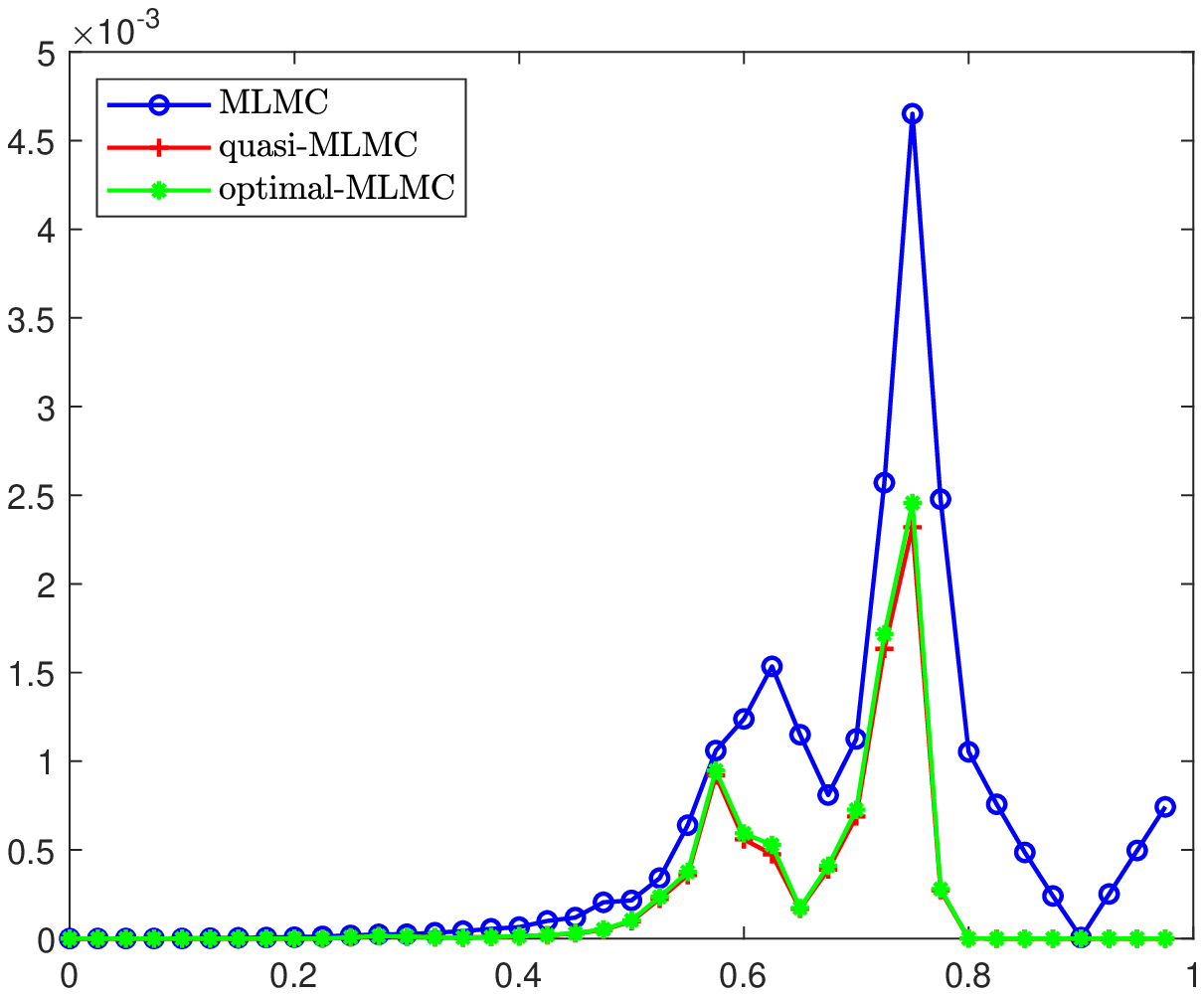}
\caption{Test 2 (II): Approximated expectation of density $\mathbb{E}[\rho]$ (left), velocity $\mathbb{E}[U]$ (middle) and temperature $\mathbb{E}[T]$ (right) using MLMC, quasi-optimal MLMC and optimal MLMC methods at time $t=0.15$ (top row). Relative error \cref{def:relativespatialerror} of expectation of density (left), velocity (middle) and temperature (right) using three MLMC methods (bottom row).}
\label{fig:st2exp}
\end{center}
\end{figure}

\begin{figure}[tb]
\begin{center}
\includegraphics[width=1.89in]{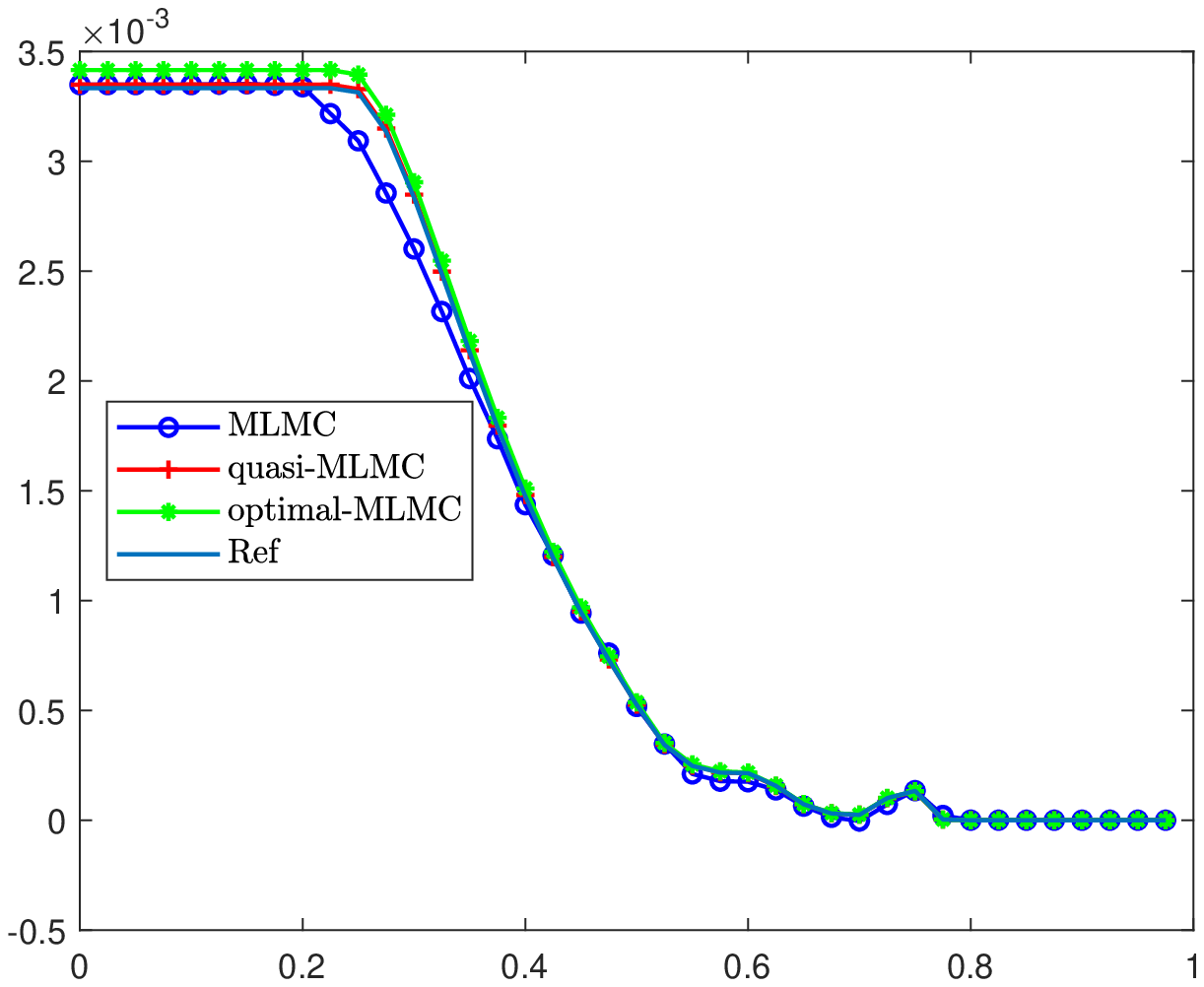}
\includegraphics[width=1.89in]{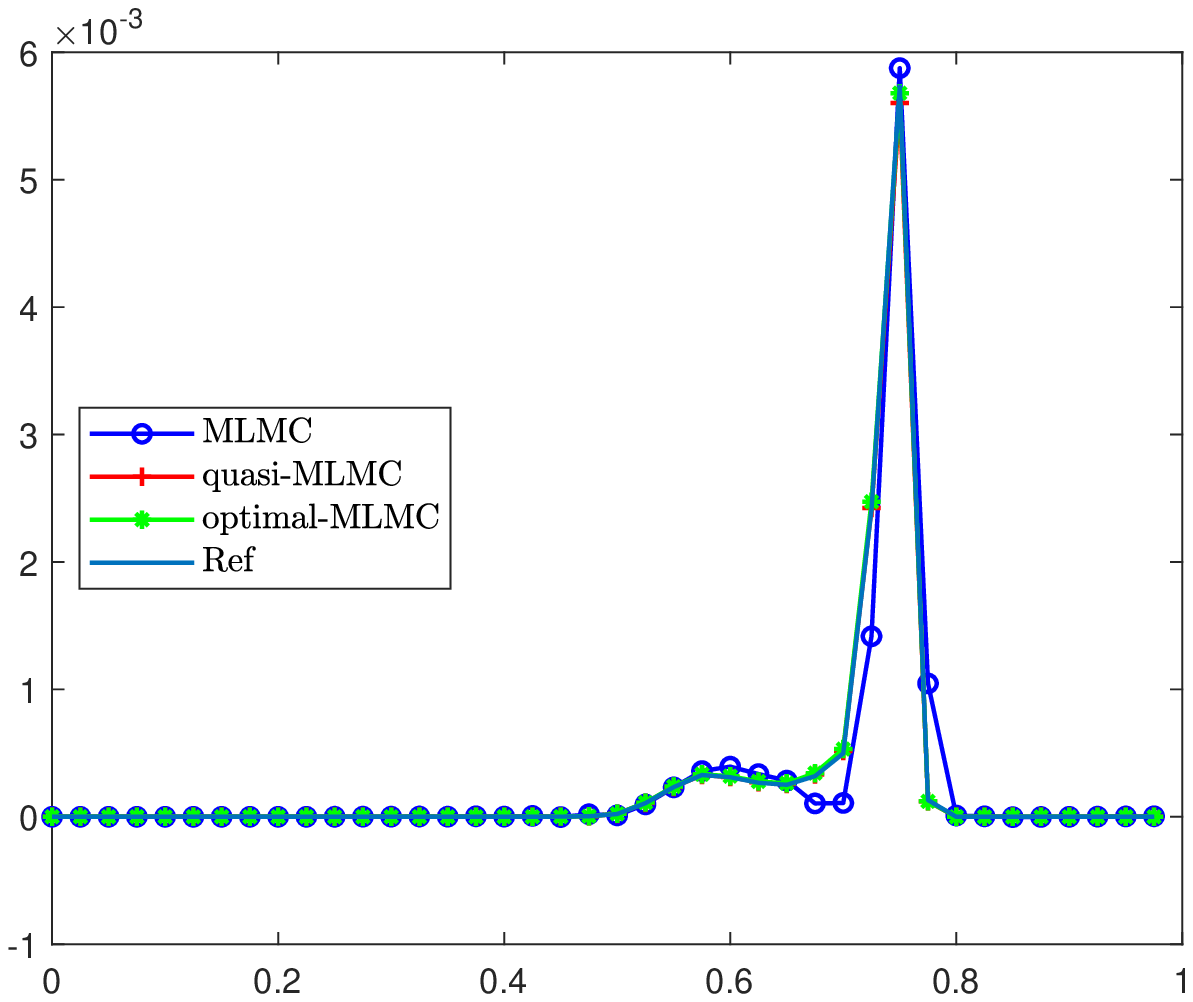}
\includegraphics[width=1.89in]{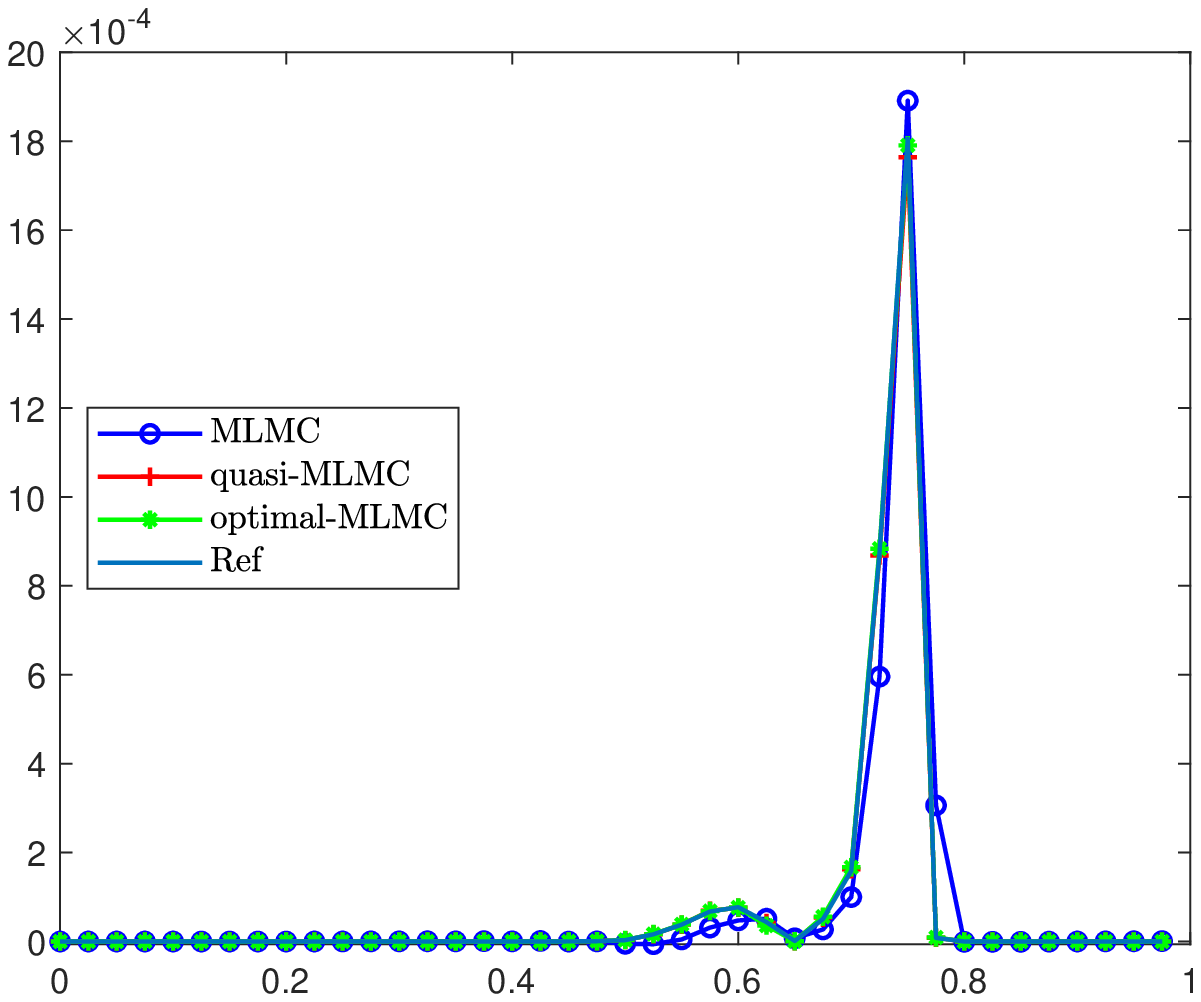}
\includegraphics[width=1.89in]{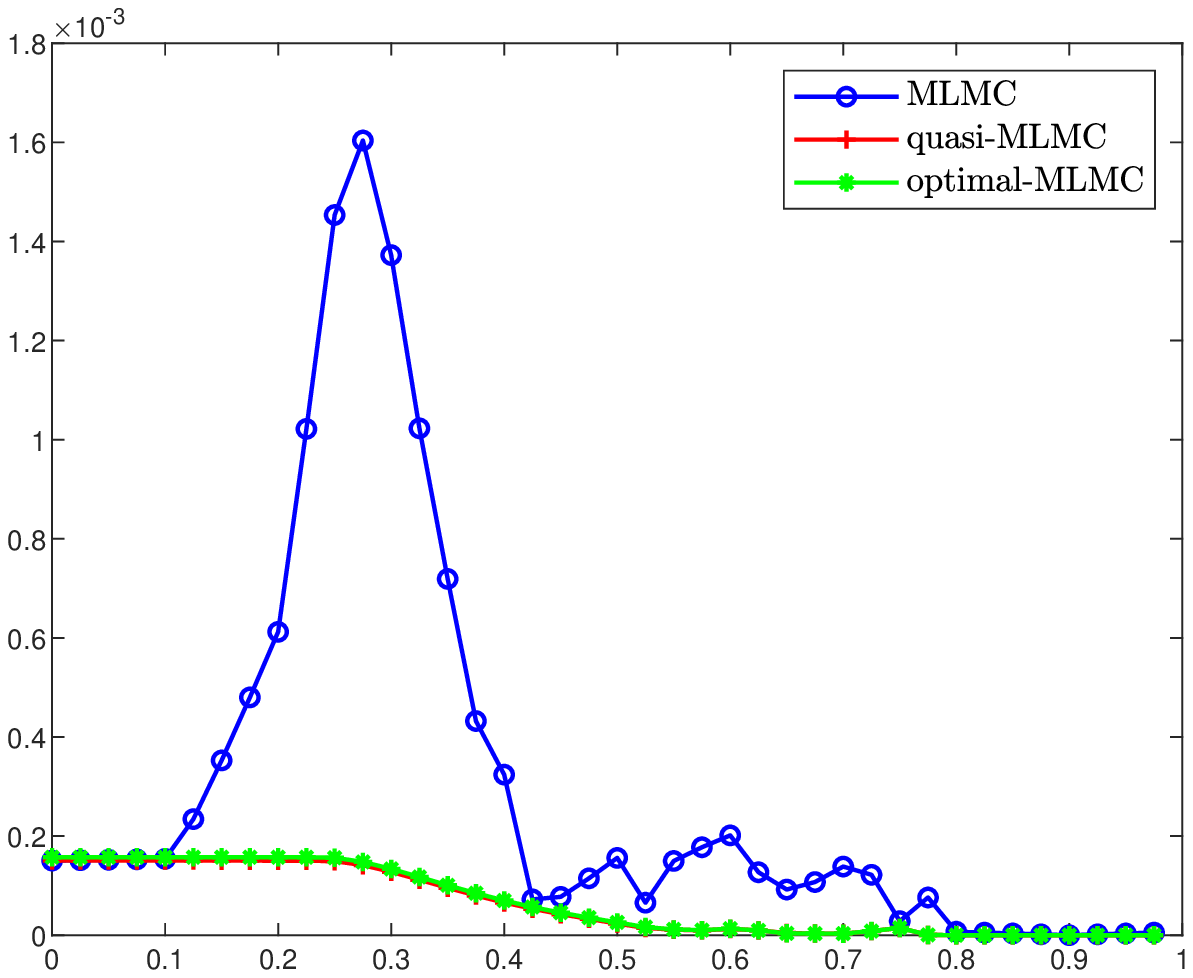}
\includegraphics[width=1.89in]{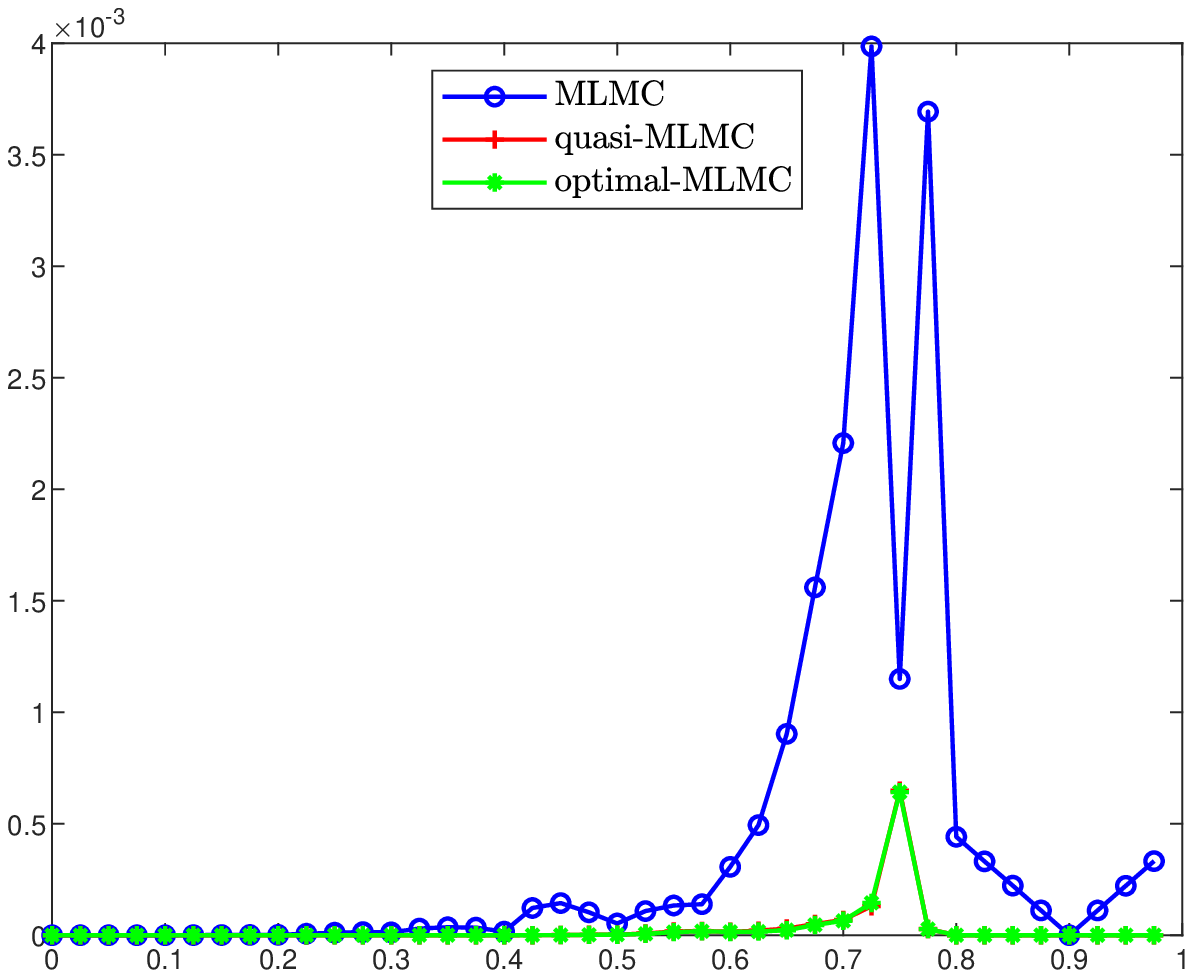}
\includegraphics[width=1.89in]{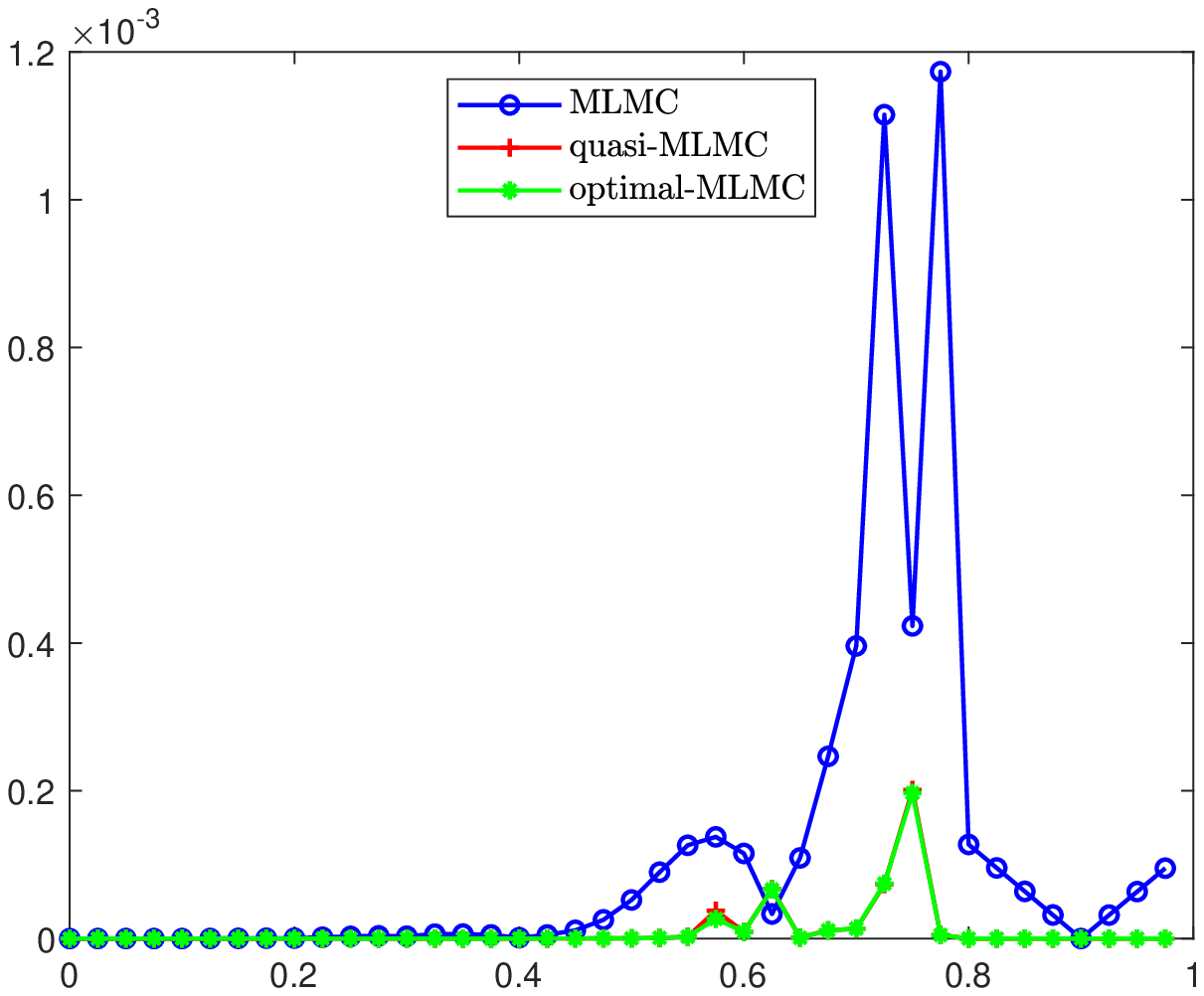}
\caption{Test 2 (II): Approximated variance of density $\mathbb{V}[\rho]$ (left), velocity $\mathbb{V}[U]$ (middle) and temperature $\mathbb{V}[T]$ (right) using MLMC, quasi-optimal MLMC and optimal MLMC methods at time $t=0.15$ (top row). Relative error \cref{def:relativespatialerror} of variance of density (left), velocity (middle) and temperature (right) using three methods (bottom row).}
\label{fig:st2var}
\end{center}
\end{figure}

To better see the difference of three MLMC methods, we plot the values of $\lambda_1$ and $\lambda_2$ in the quasi-optimal and optimal MLMC methods for both tests in \cref{fig:st1lambda} and  \ref{fig:st2lambda}. It is clear that for these problems with shocks/discontinuities the values are far from one in various regions of the computational domain. This is particularly true for the temperature and velocity in agreement with the corresponding errors observed in the previous figures.

\begin{figure}[tb]
\begin{center}
\includegraphics[width=2.2in]{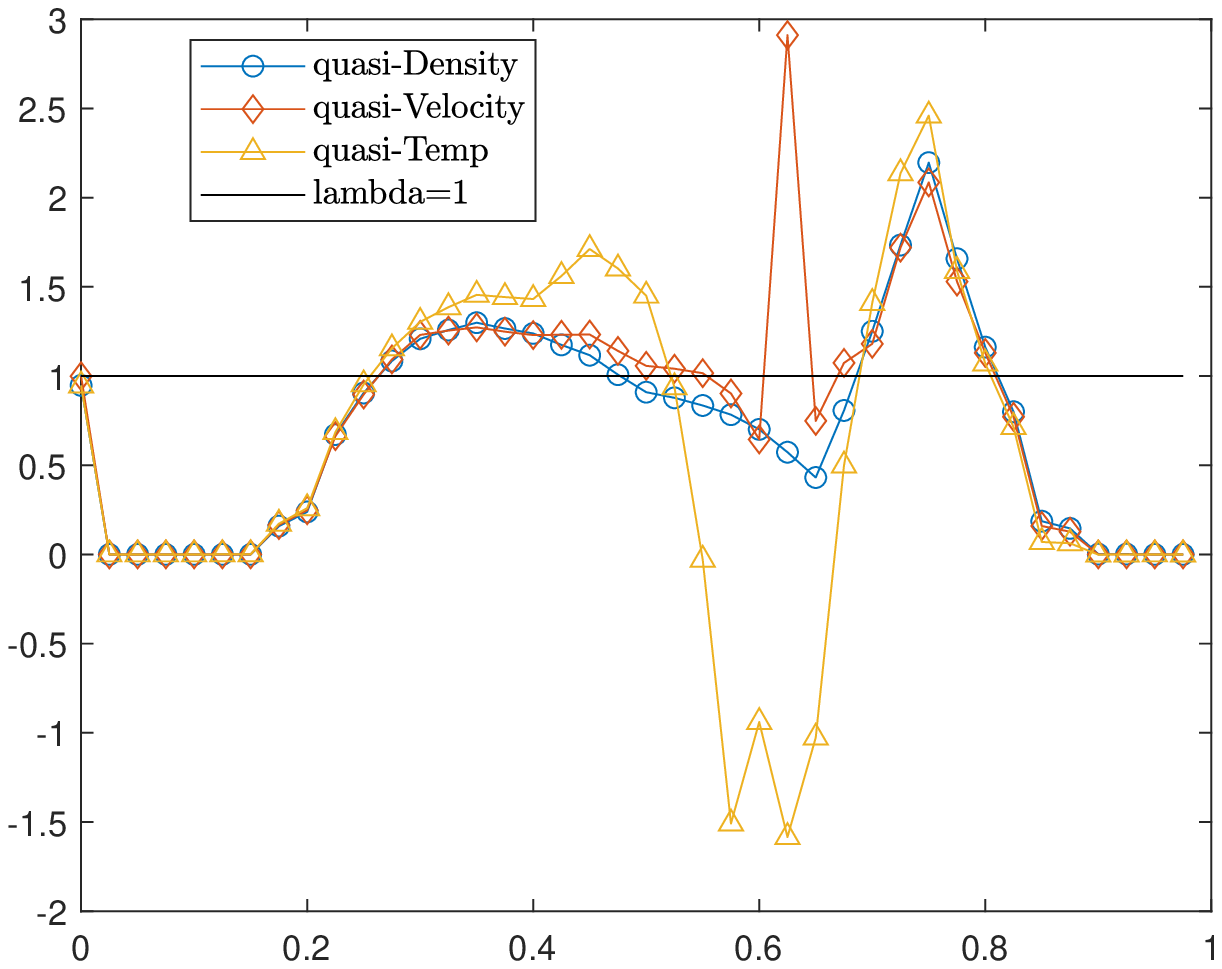}
\includegraphics[width=2.2in]{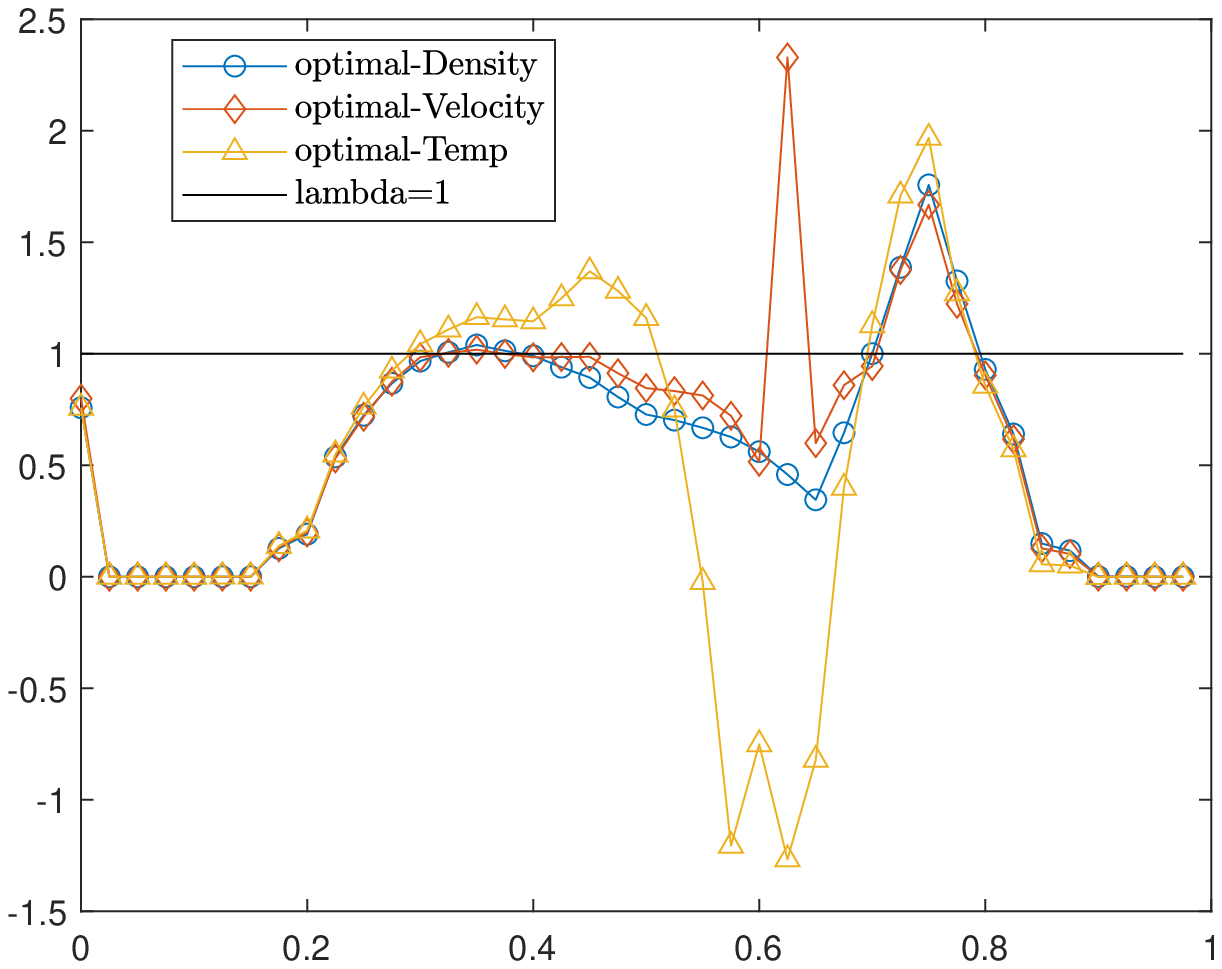}
\includegraphics[width=2.2in]{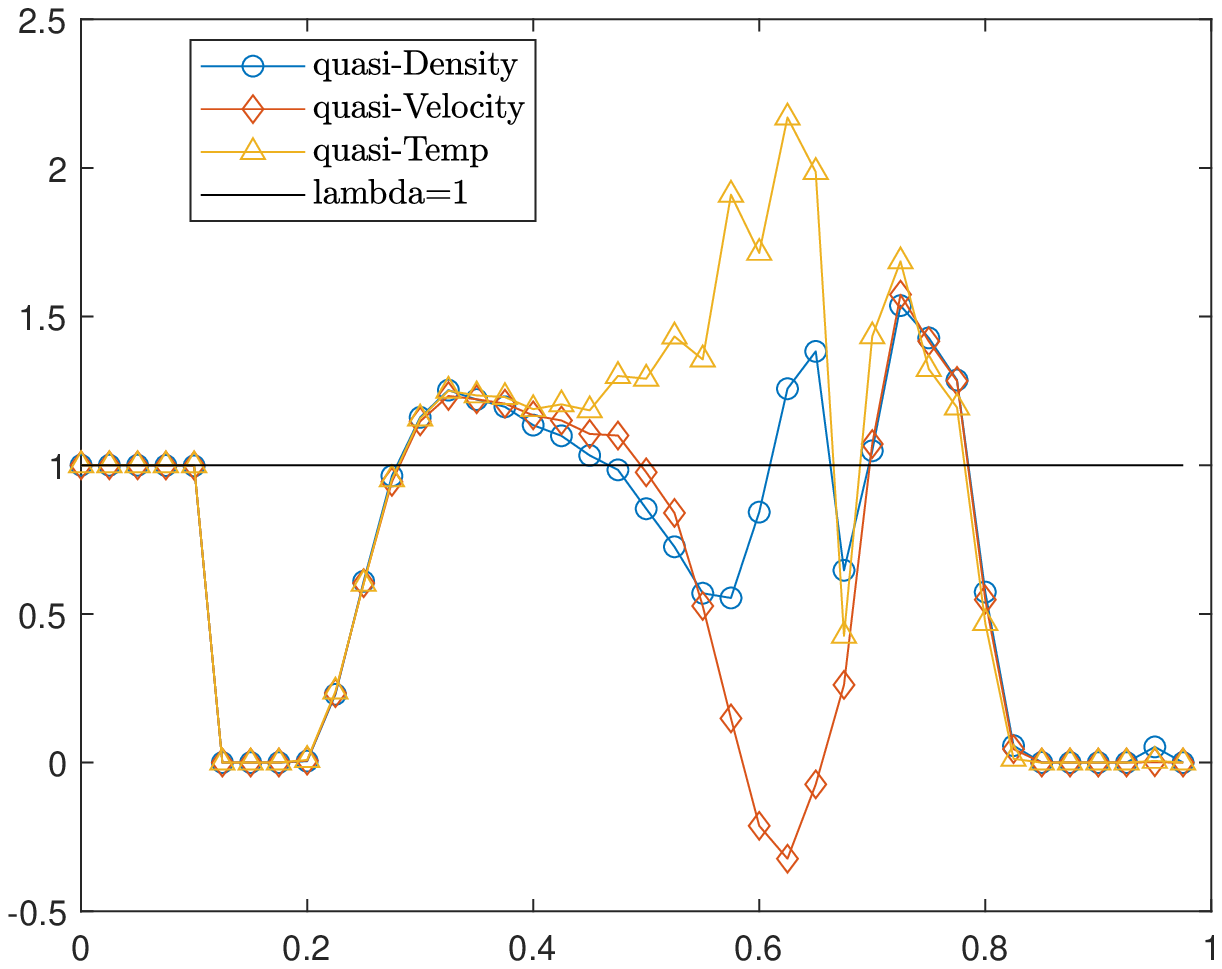}
\includegraphics[width=2.2in]{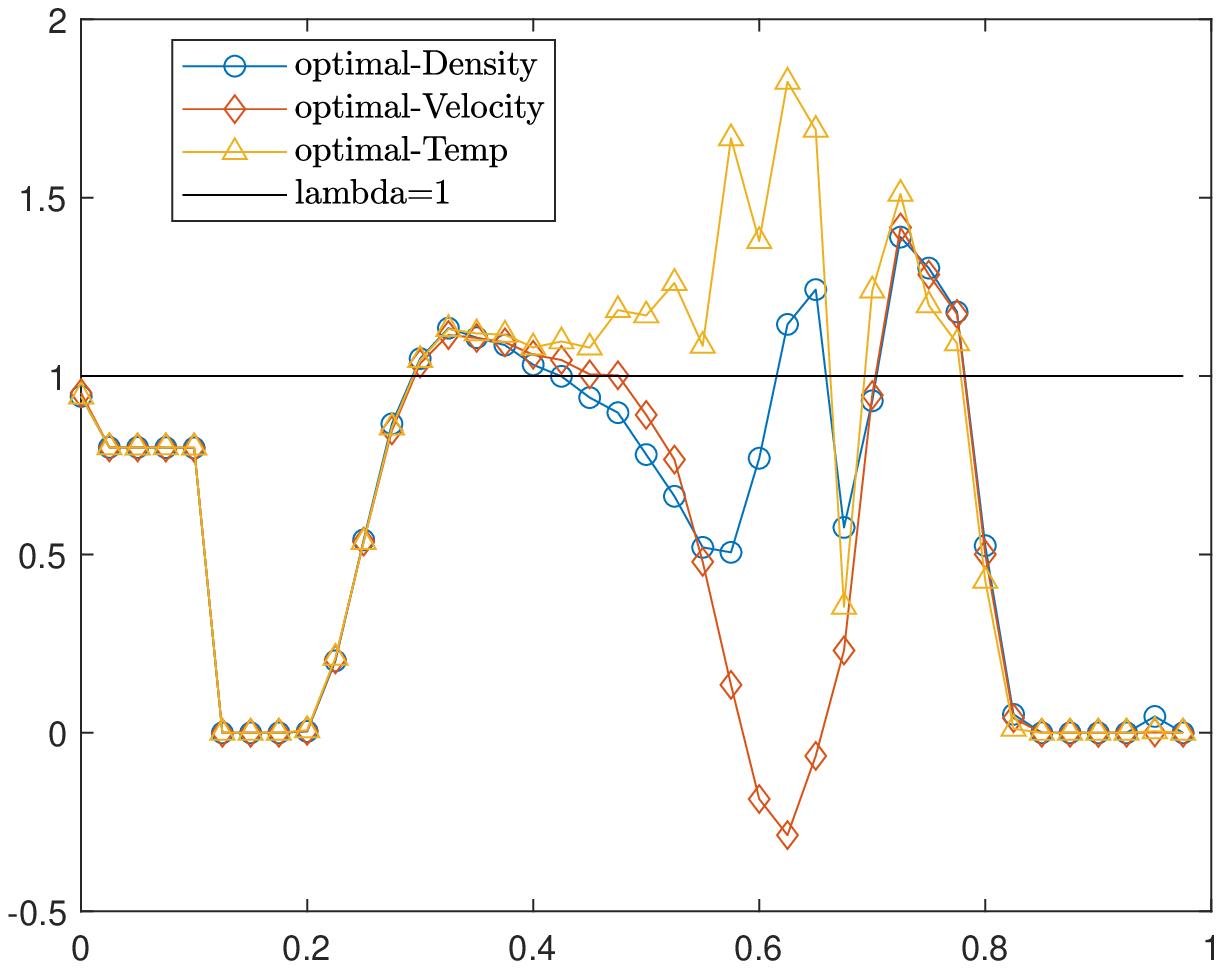}
\caption{Test 2 (I): Values of $\lambda_1$ in quasi-optimal (left) and optimal (right) MLMC methods (top row). Values of $\lambda_2$ in quasi-optimal (left) and optimal (right) MLMC methods (bottom row).}
\label{fig:st1lambda}
\end{center}
\end{figure}

\begin{figure}[tb]
\begin{center}
\includegraphics[width=2.2in]{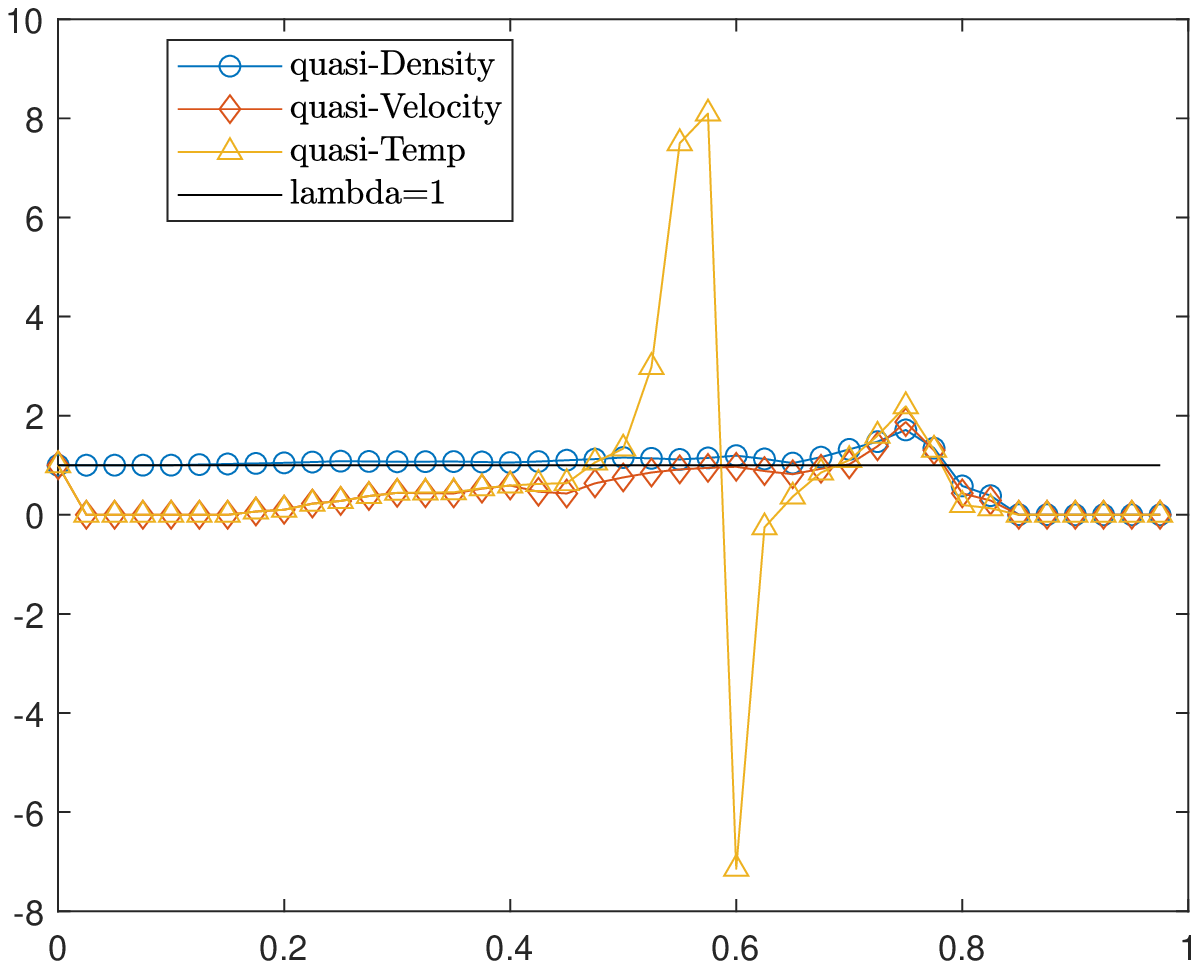}
\includegraphics[width=2.2in]{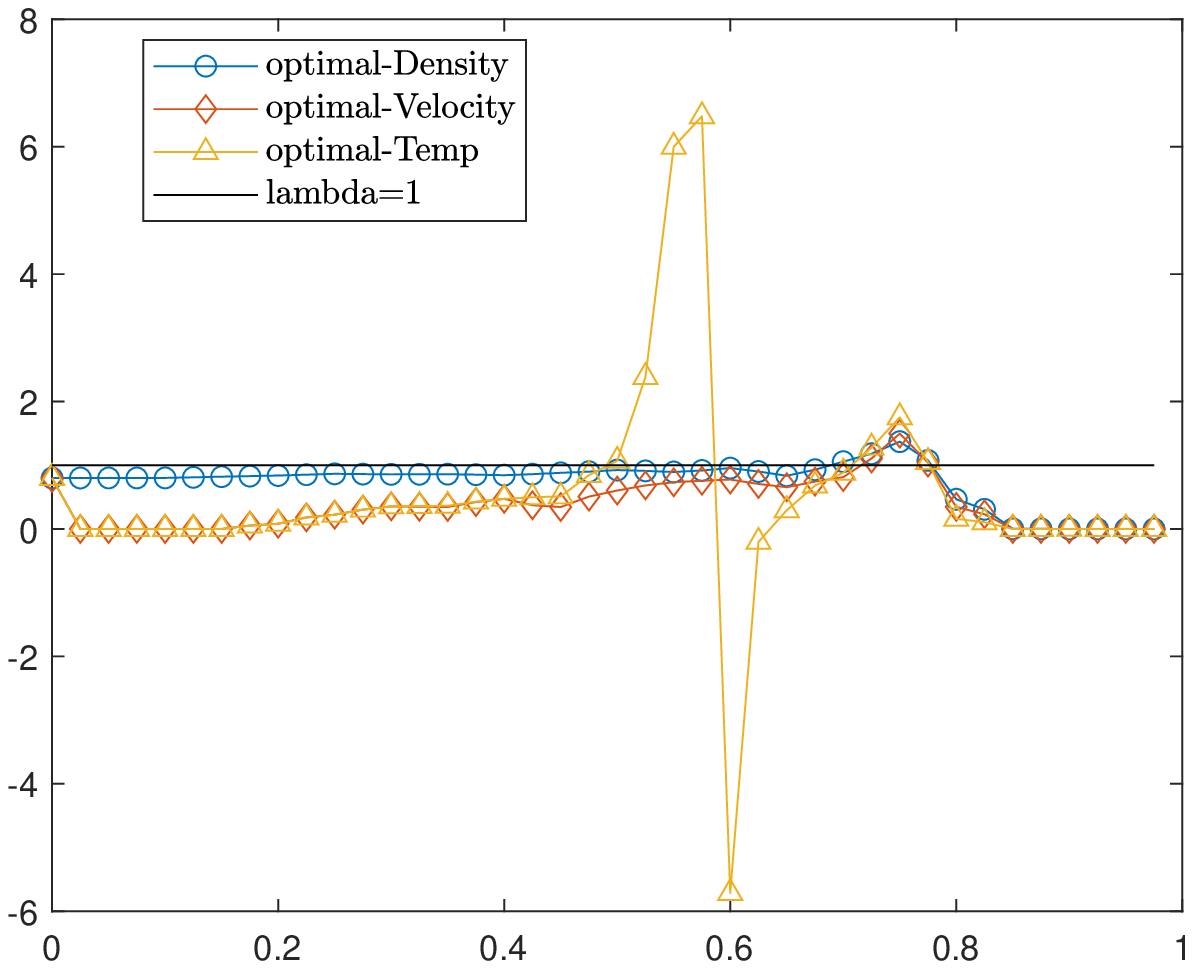}
\includegraphics[width=2.2in]{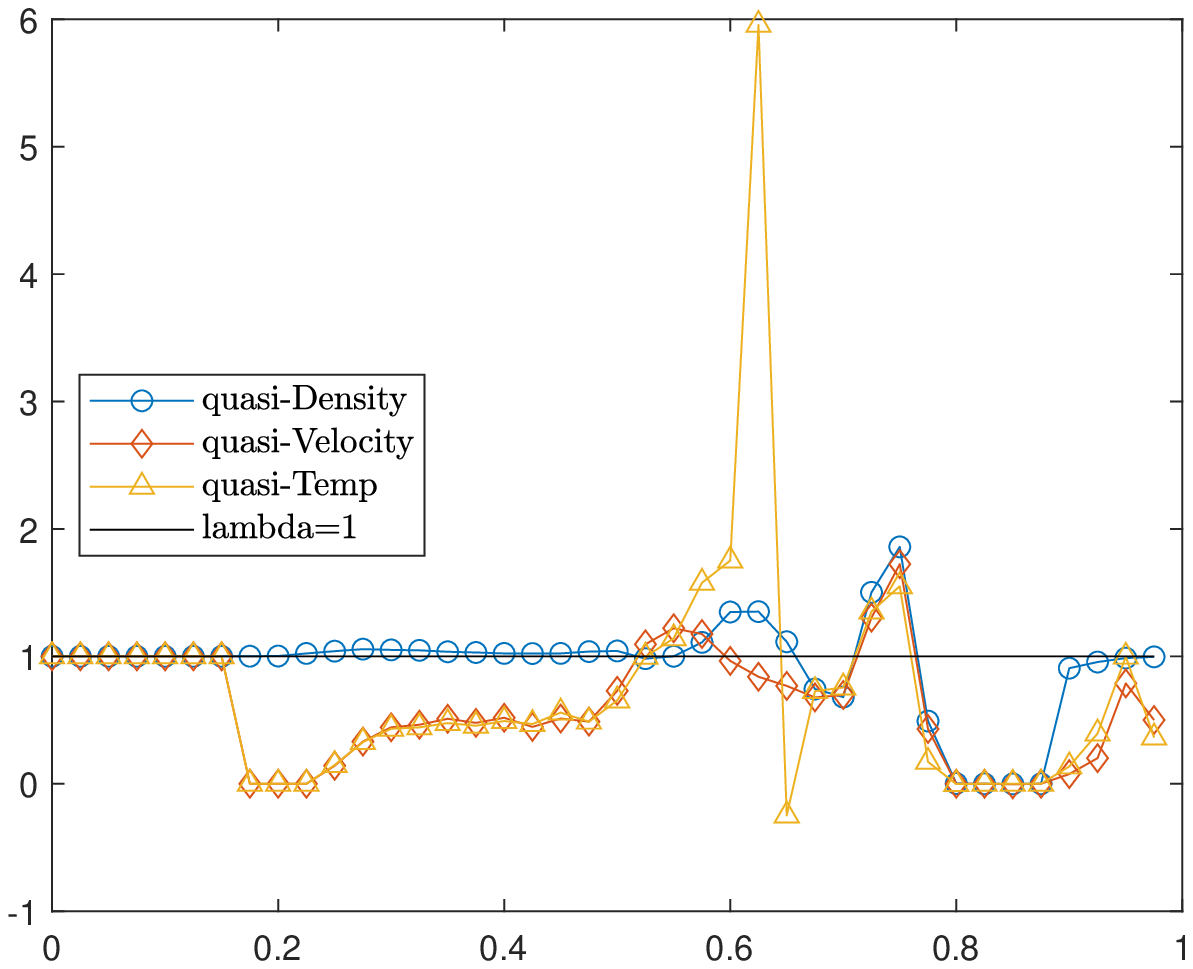}
\includegraphics[width=2.2in]{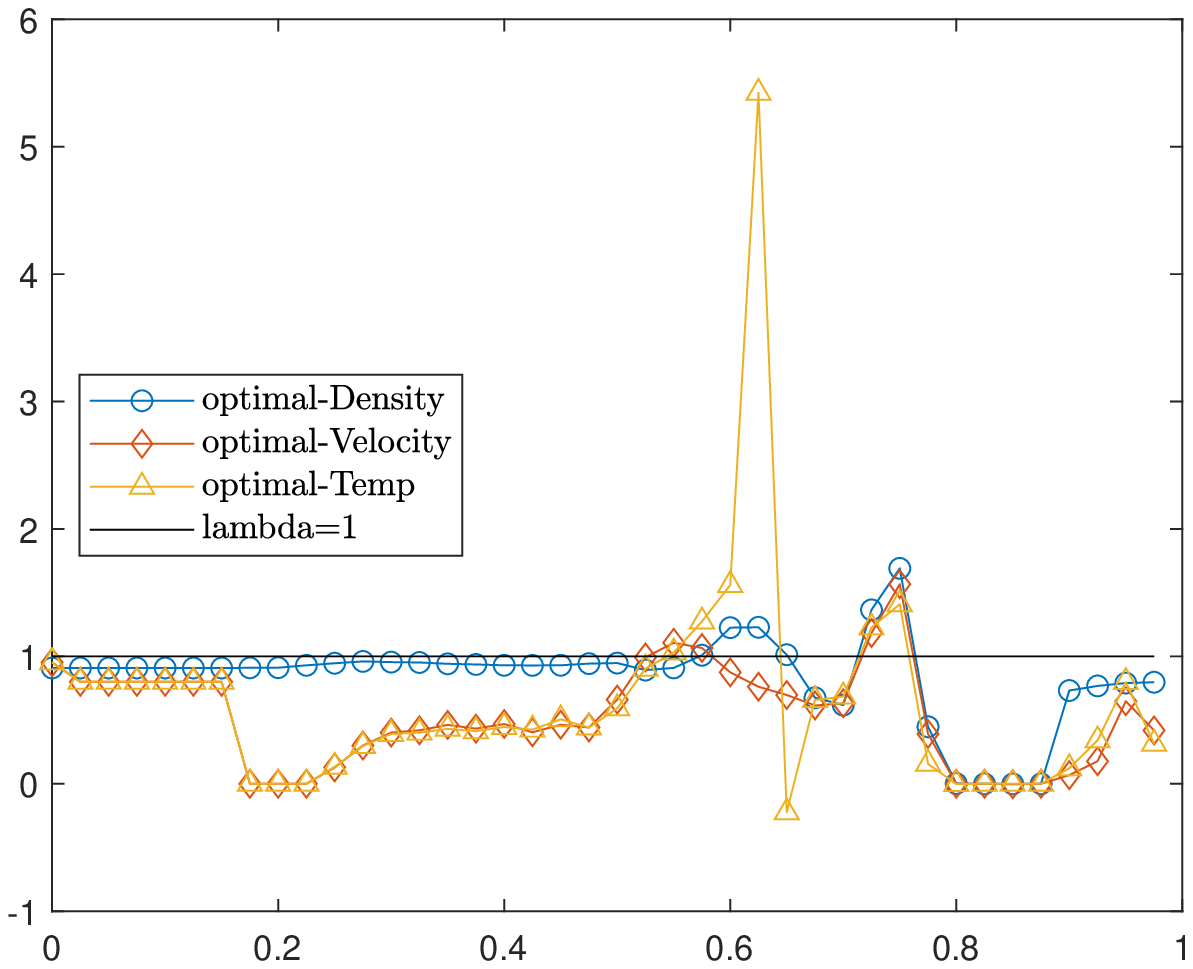}
\caption{Test 2 (II): Values of $\lambda_1$ in quasi-optimal (left) and optimal (right) MLMC methods (top row). Values of $\lambda_2$ in quasi-optimal (left) and optimal (right) MLMC methods (bottom row).}
\label{fig:st2lambda}
\end{center}
\end{figure}

\subsection{Test 3: Sudden heating problem}

In the last test, we consider a problem with random boundary condition. The gas is initially in a constant state with $\rho_0=1$, $\mU_0=(0,0,0)$, $T_0=1$ and $f_0(\vx,\vv)=M_{\rho_0,\mU_0,T_0}$. At time $t=0$, we suddenly change the wall temperature at left boundary of effective spatial domain at $x=0$ to 
\be
T_w(z)=3(T_0+sz), \ s=0.2,
\ee
where the random variable $z$ obeys the uniform distribution on $[-1,1]$. We assume purely diffusive Maxwell boundary condition at $x=0$ and homogeneous Neumann boundary condition at $x=1$. The Knudsen number is set as $\varepsilon=0.1$. This is a classical benchmark test in kinetic theory. With the sudden rise of the wall temperature, the gas close to the wall is heated and accordingly the pressure rises sharply and pushes the gas away forming a shock propagating into the domain. 

We compare the three MLMC methods using parameters: mesh sizes $N_1=10, N_2=20, N_3=40$, and number of samples $M_1=1280, M_2=320, M_3=80$ (these parameters are chosen based on a similar test as in previous examples, we omit the detail). The results are shown in \cref{fig:sdexp} and \cref{fig:sdvar}. Again the control variate MLMC methods outperform the standard MLMC in all simulations.

\begin{figure}[tb]
\begin{center}
\includegraphics[width=1.89in]{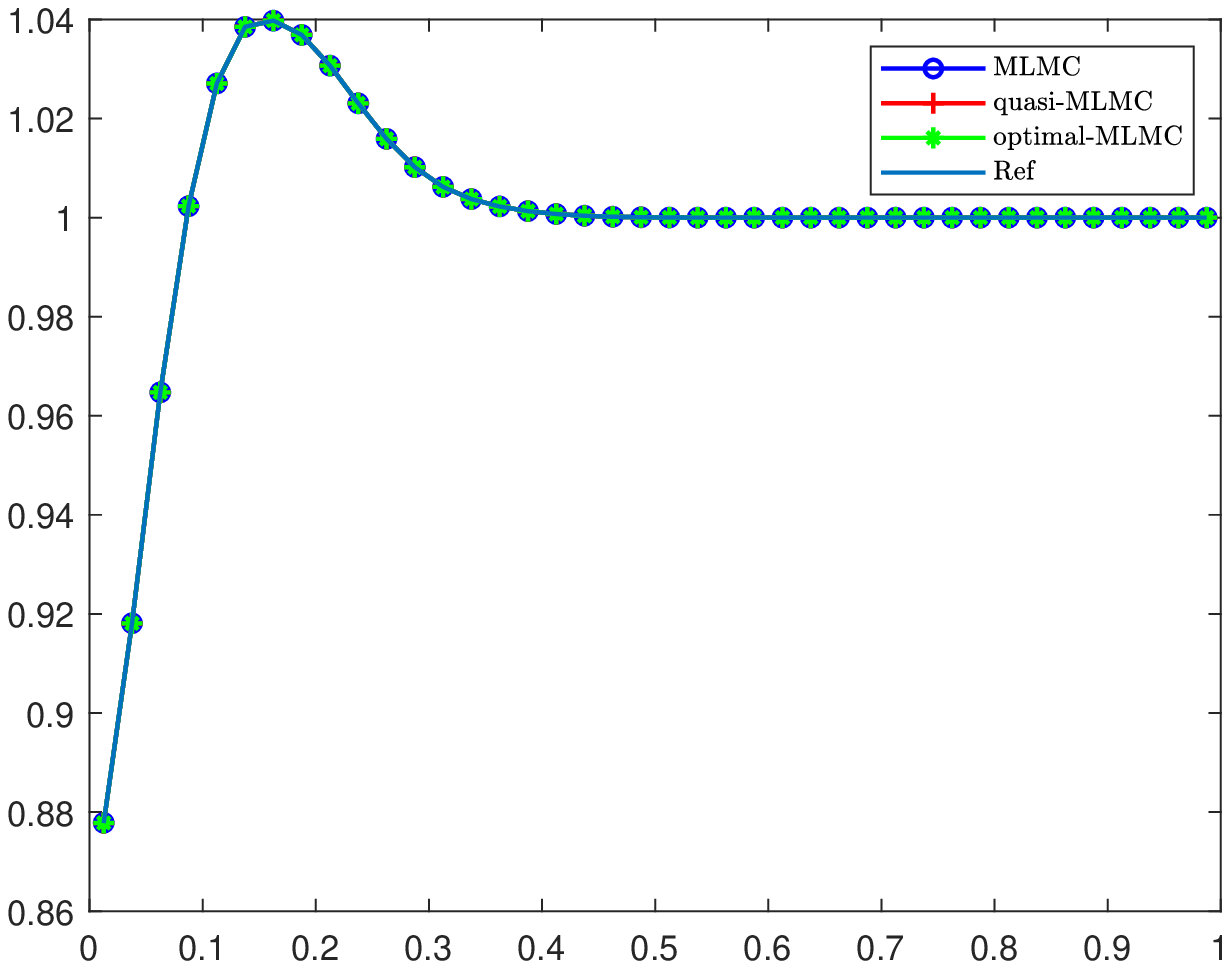}
\includegraphics[width=1.89in]{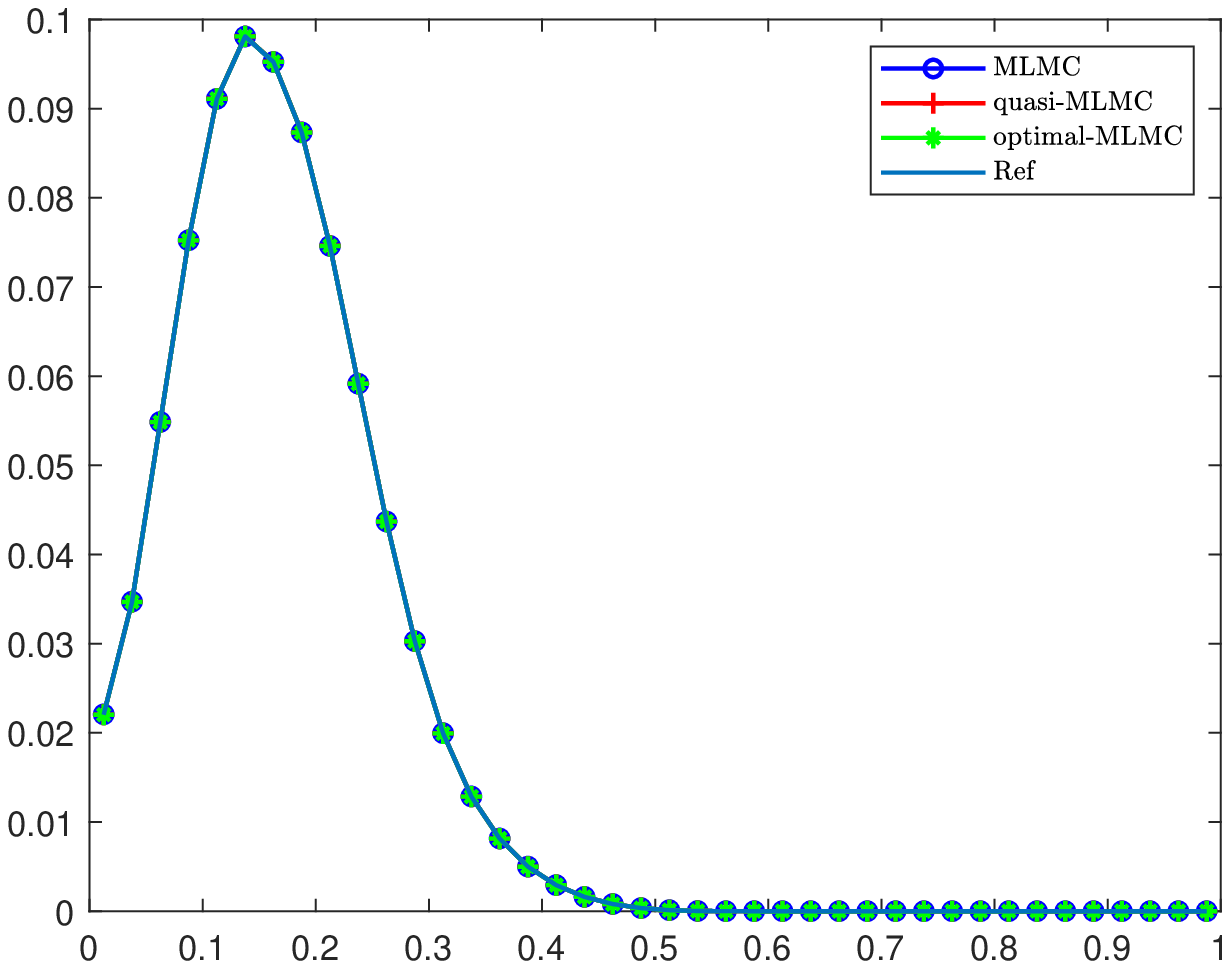}
\includegraphics[width=1.89in]{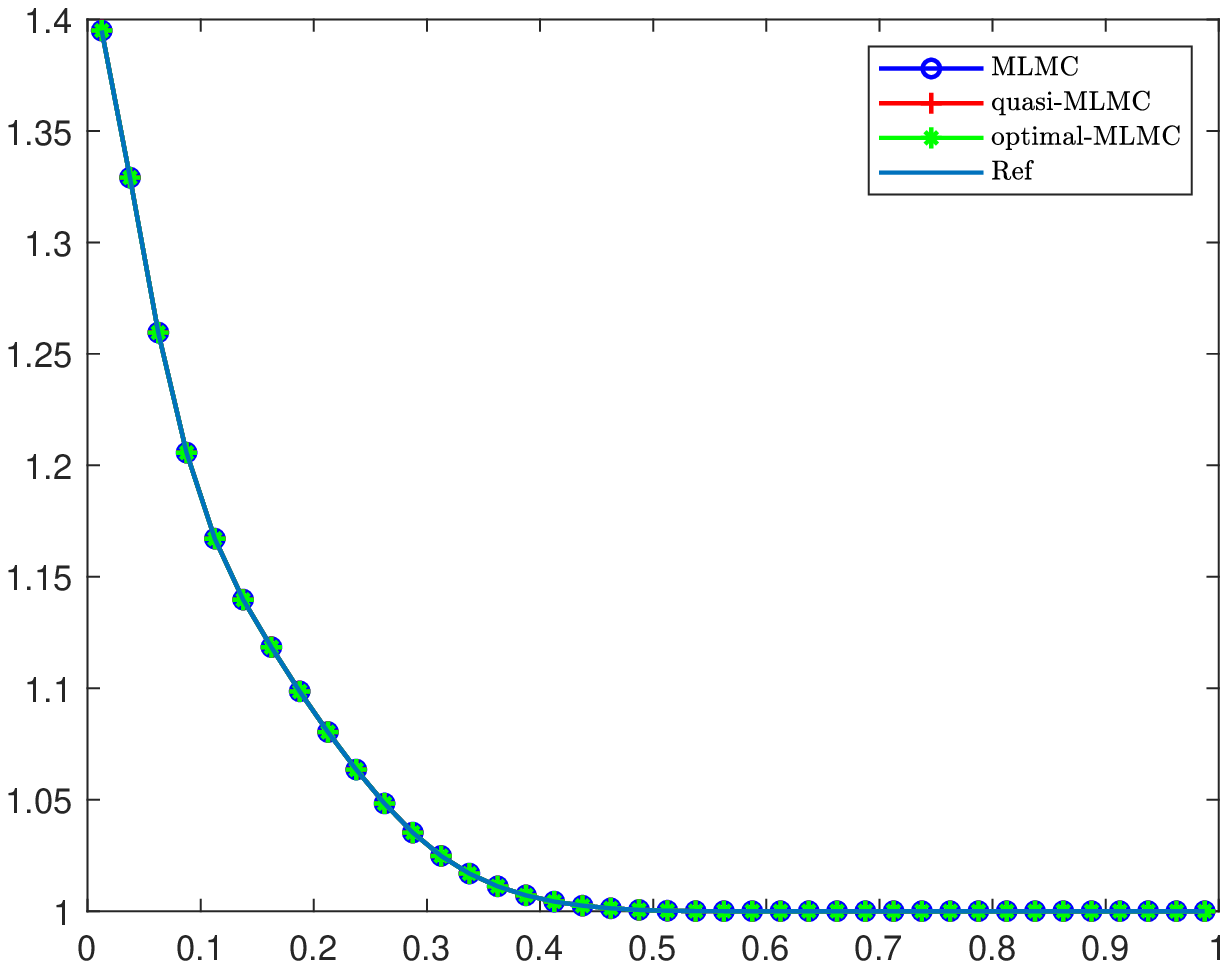}
\includegraphics[width=1.89in]{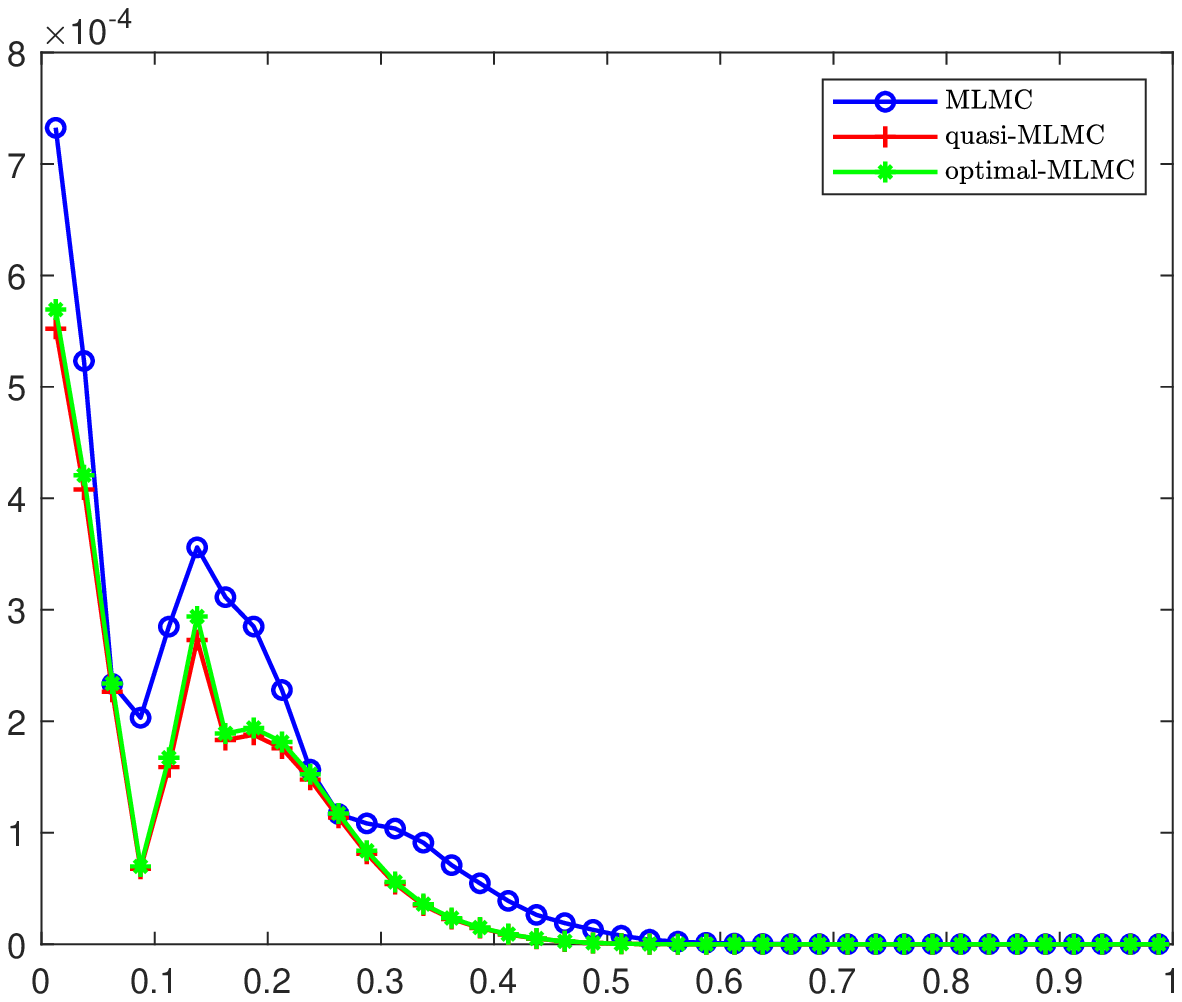}
\includegraphics[width=1.89in]{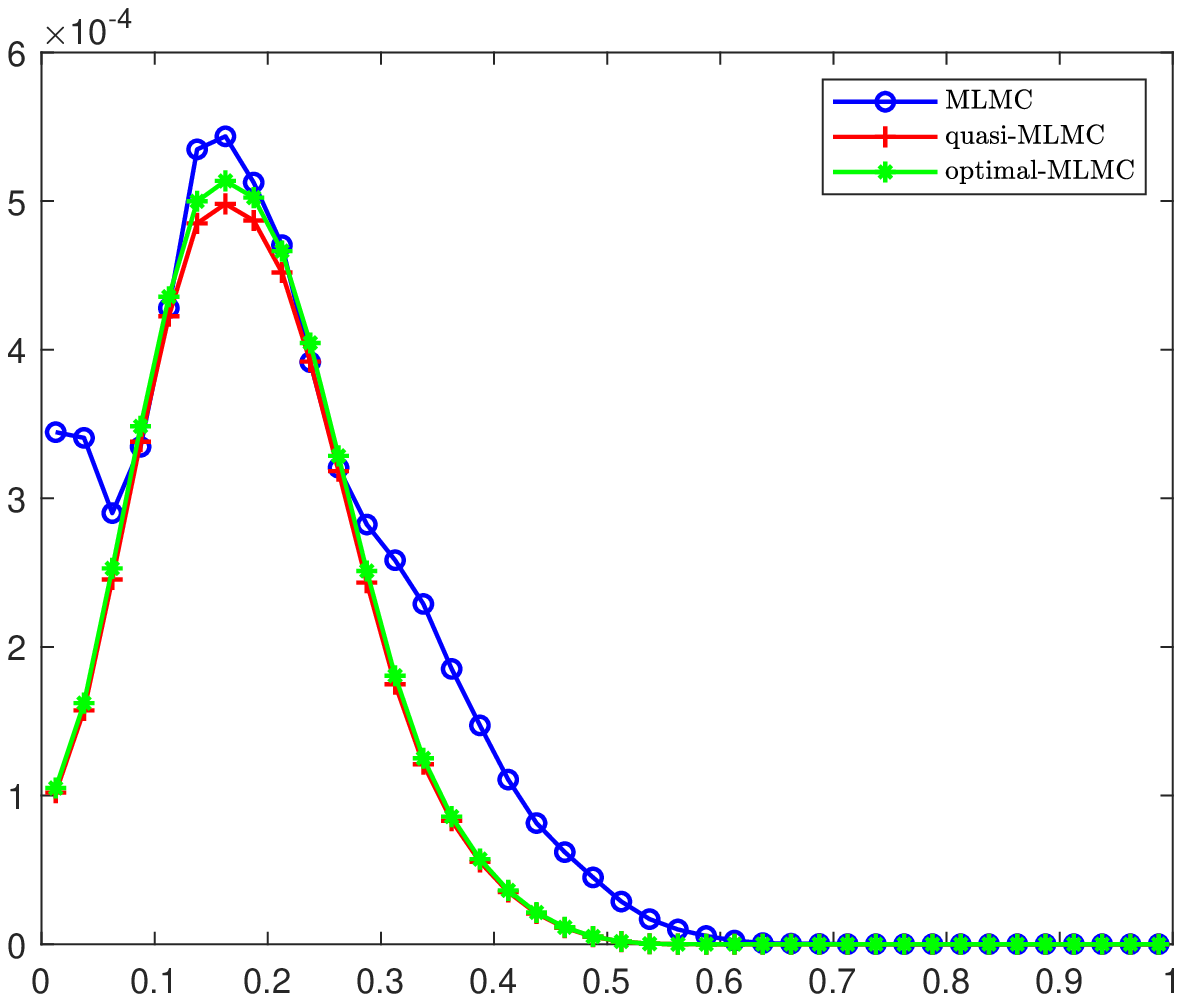}
\includegraphics[width=1.89in]{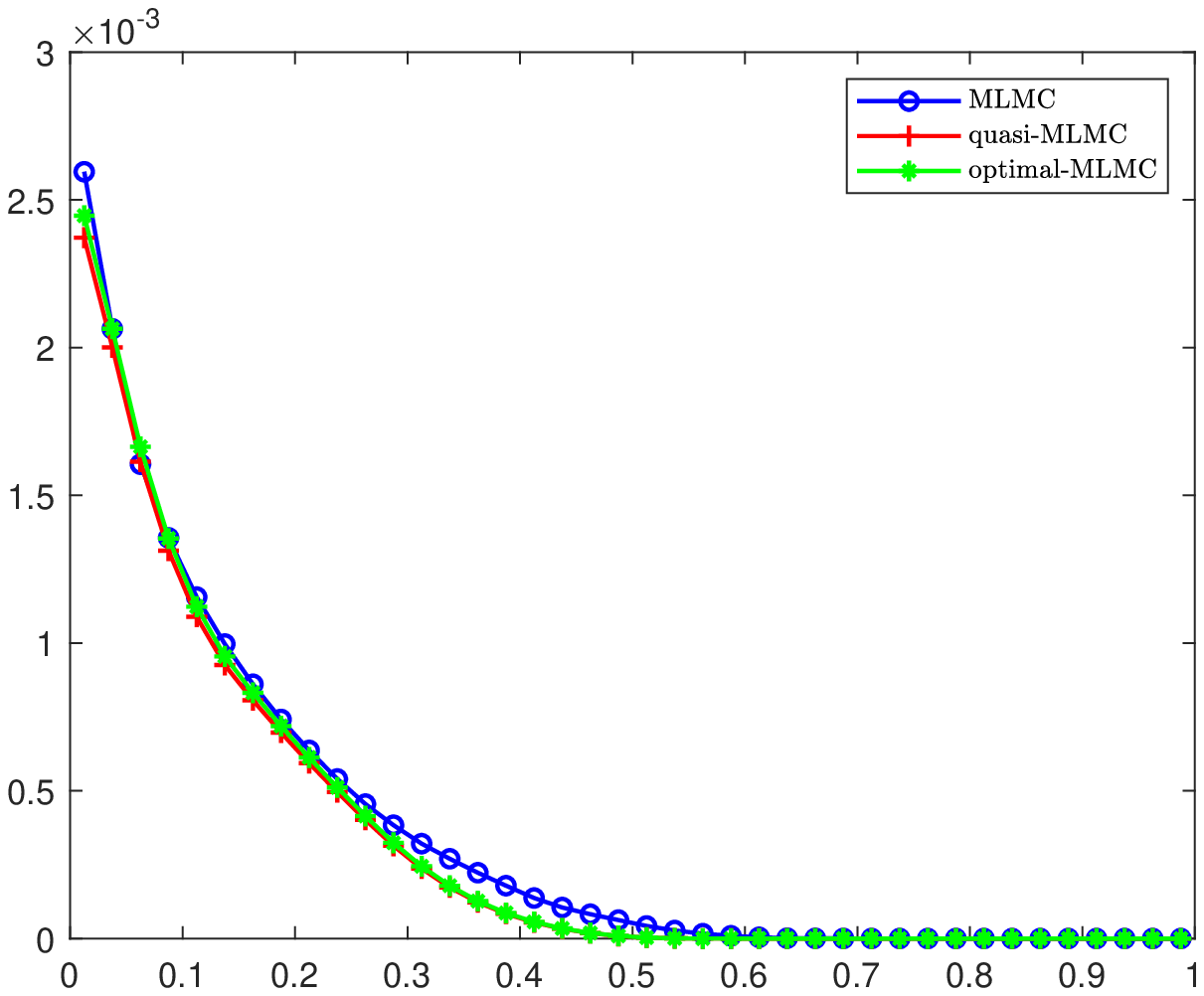}
\caption{Test 3: Approximated expectation of density $\mathbb{E}[\rho]$ (left), velocity $\mathbb{E}[U]$ (middle) and temperature $\mathbb{E}[T]$ (right) using MLMC, quasi-optimal MLMC and optimal MLMC methods at time $t=0.1$ (top row). Relative error \cref{def:relativespatialerror} of expectation of density (left), velocity (middle) and temperature (right) using three methods (bottom row).}
\label{fig:sdexp}
\end{center}
\end{figure}

\begin{figure}[tb]
\begin{center}
\includegraphics[width=1.89in]{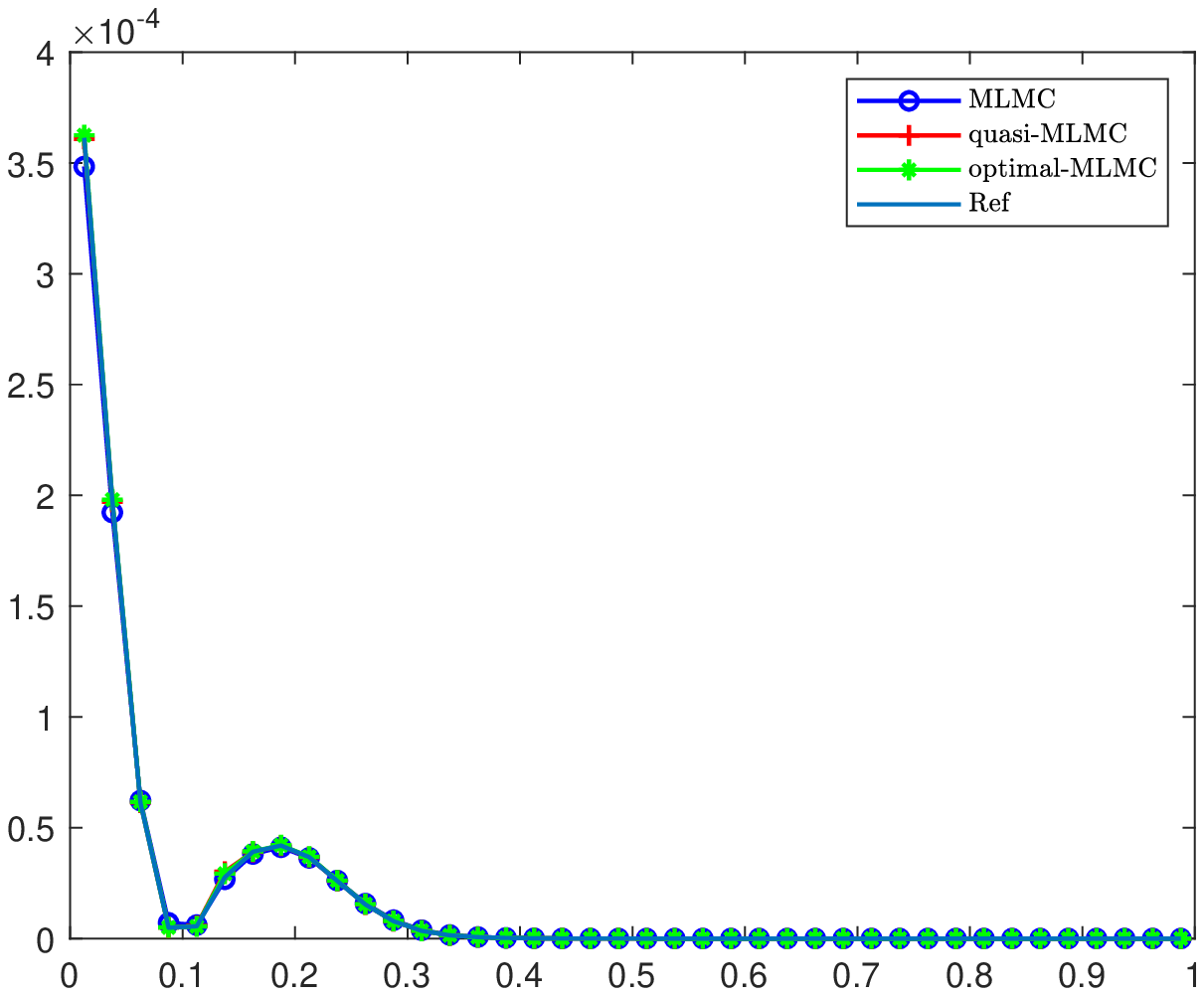}
\includegraphics[width=1.89in]{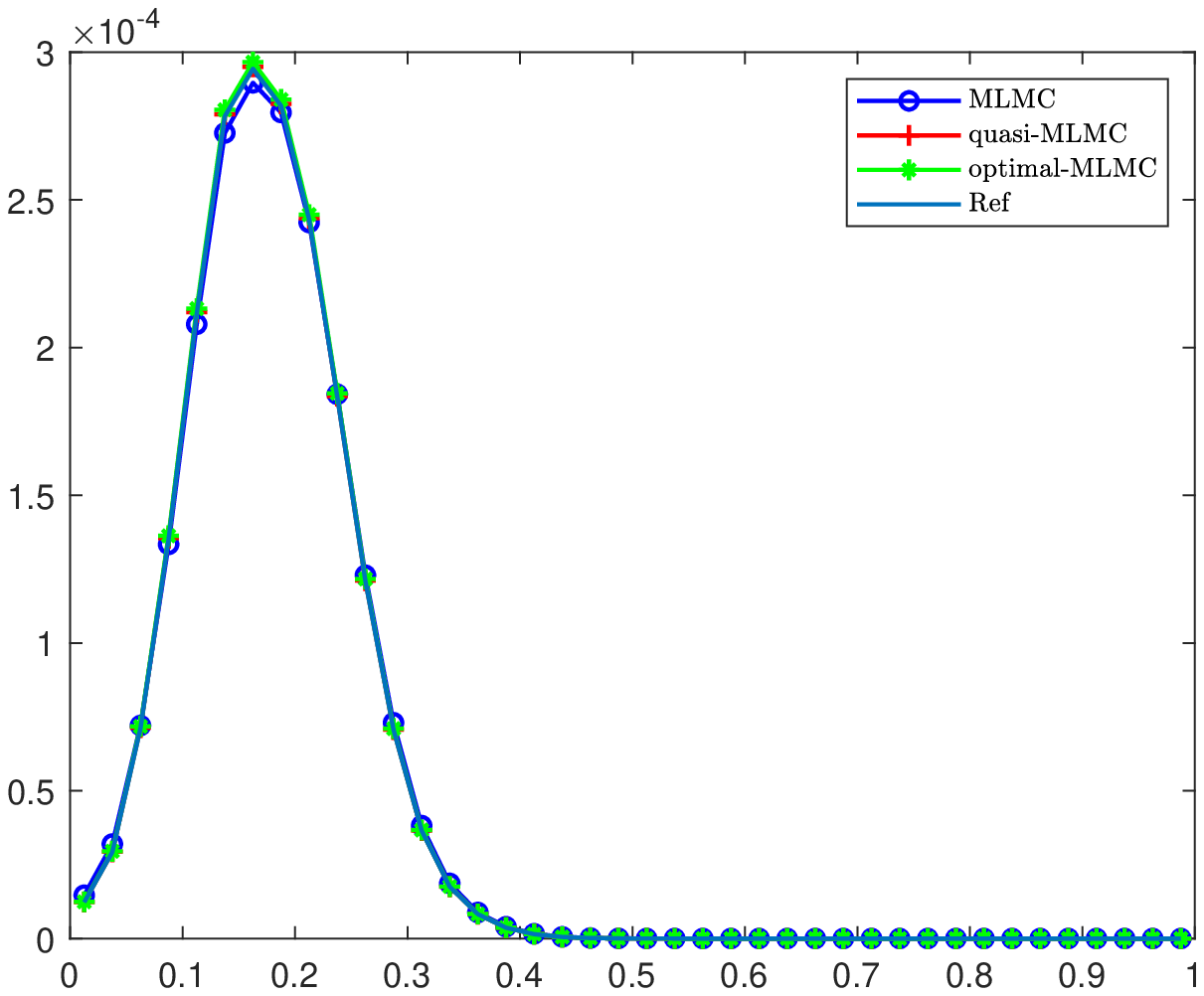}
\includegraphics[width=1.89in]{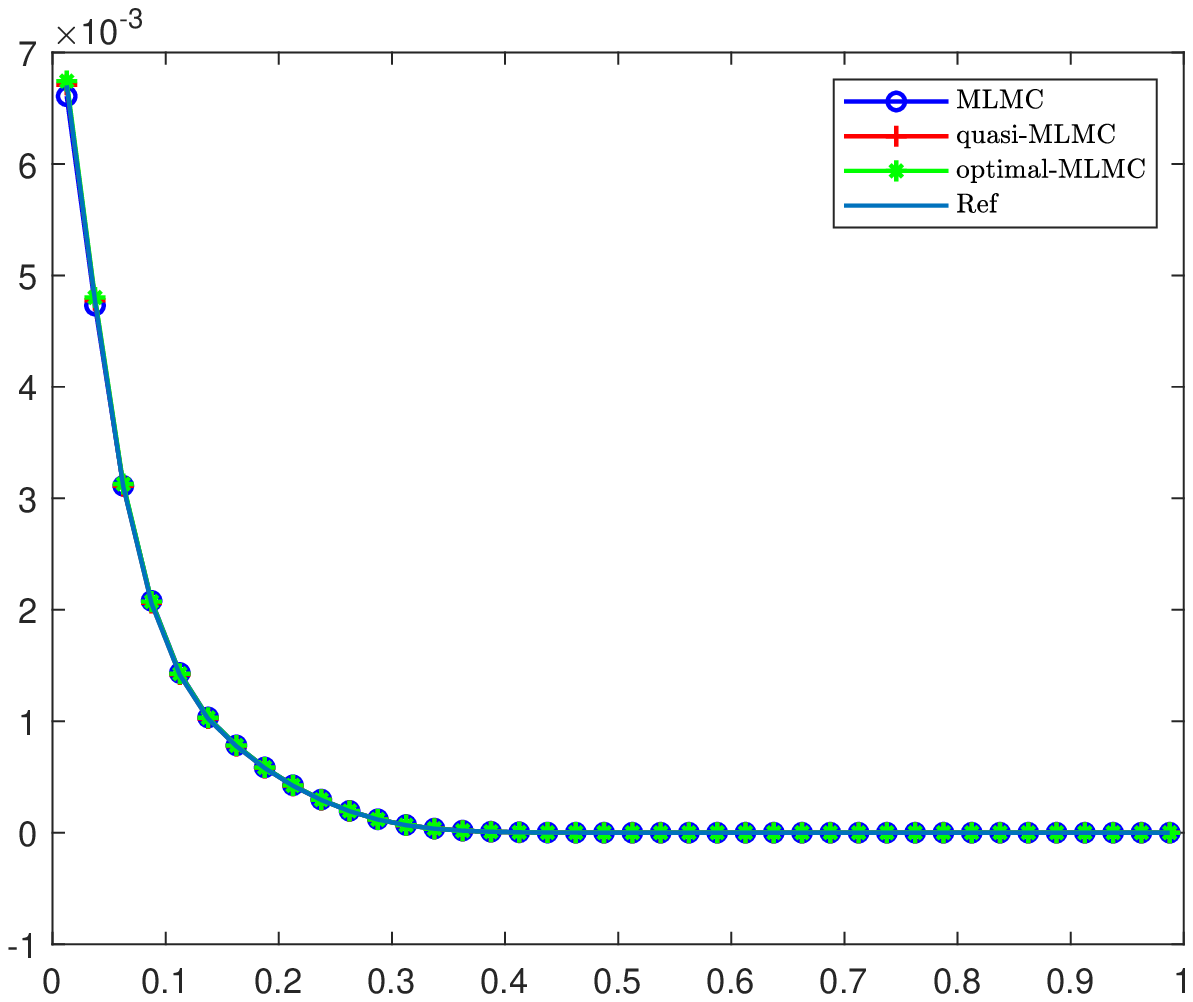}
\includegraphics[width=1.89in]{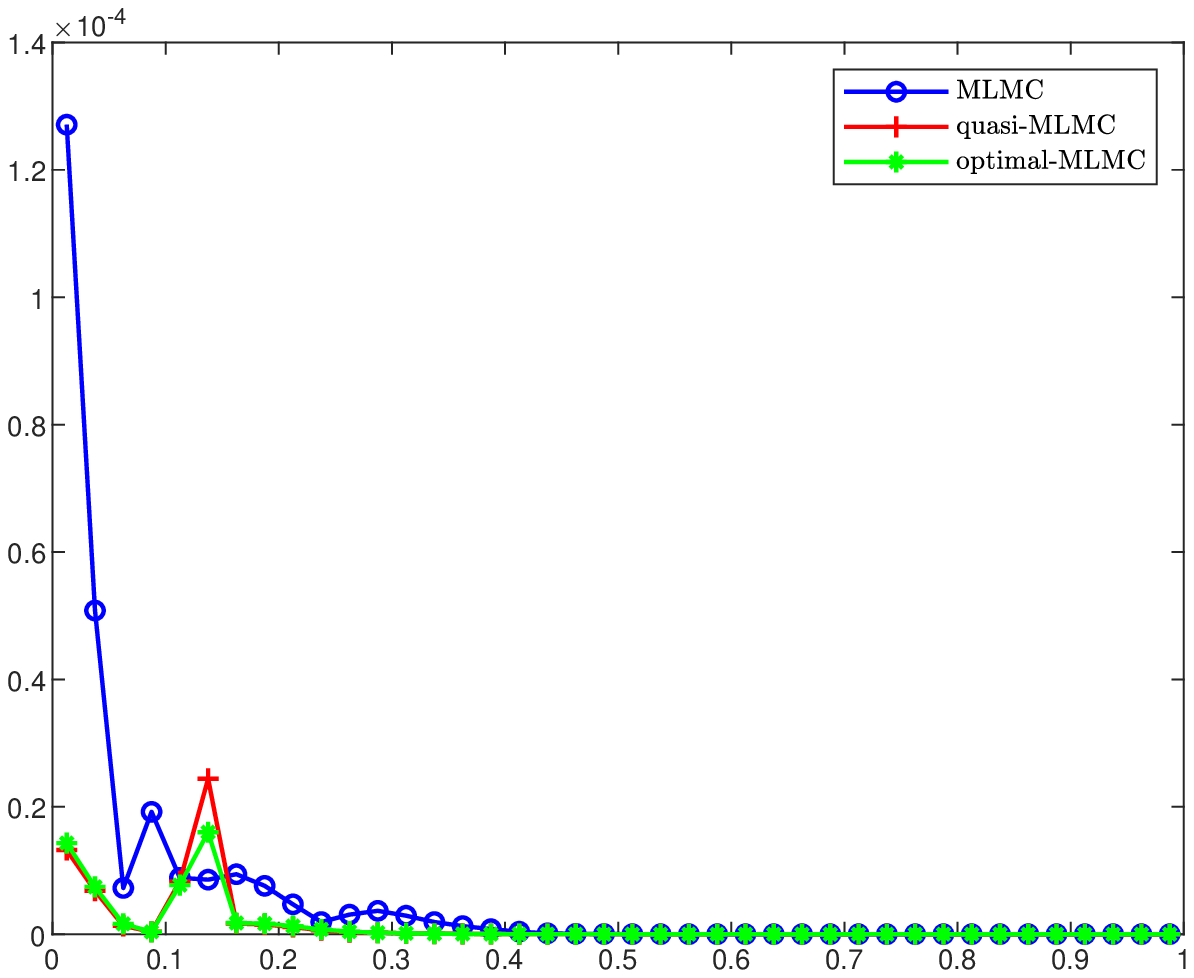}
\includegraphics[width=1.89in]{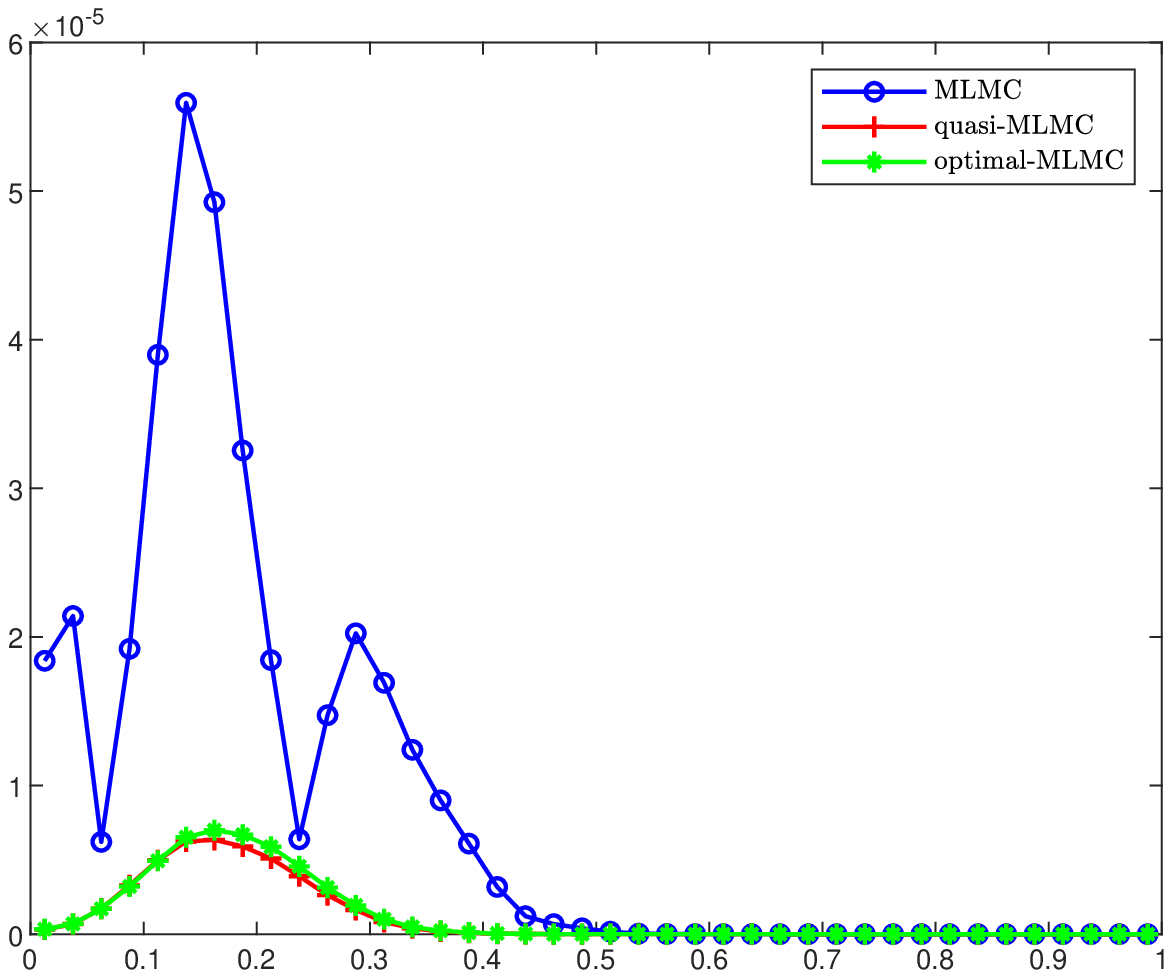}
\includegraphics[width=1.89in]{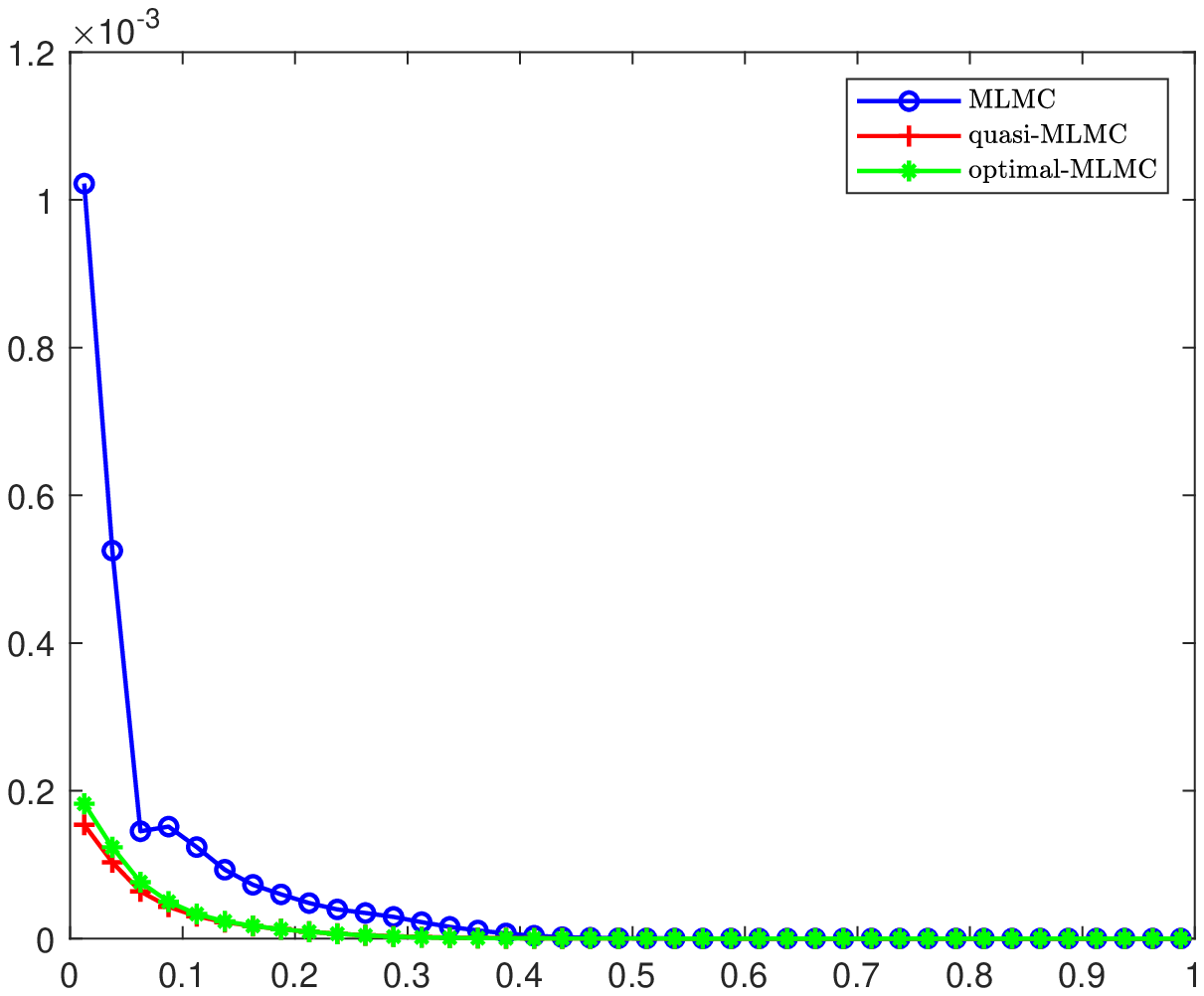}
\caption{Test 3: Approximated variance of density $\mathbb{V}[\rho]$ (left), velocity $\mathbb{V}[U]$ (middle) and temperature $\mathbb{V}[T]$ (right) using MLMC, quasi-optimal MLMC and optimal MLMC methods at time $t=0.1$ (top row). Relative error \cref{def:relativespatialerror} of variance of density (left), velocity (middle) and temperature (right) using three methods (bottom row).}
\label{fig:sdvar}
\end{center}
\end{figure}

%%%%%%%%%%%%%%%%%%%%%%%%%%%%%%%%%%%

\section{Conclusions}
\label{sec:conclusions}

We have introduced a control variate multilevel Monte Carlo method for the BGK model of the Boltzmann equation with uncertainty. Well-posedness of the BGK equation with random parameters, consistency and convergence analysis for various MC type methods are established. Extensive numerical results confirm that the MLMC methods perform much better than the standard MC, and the control variate MLMC is capable to provide further improvement over the conventional MLMC, in particular for problems close to fluid regimes and in presence of discontinuities, where the fidelity degree of the various levels is reduced and traditional gPC-SG based methods may fail. On the other hand, we haven't observed any gain using a global optimal strategy in the variance reduction with respect to a local quasi-optimal strategy based on two subsequent levels, which is subject to future investigation. The approach here developed naturally extends to other kinetic equations of Boltzmann type which combines deterministic discretizations in the phase space with Monte Carlo sampling in the random space.

%%%%%%%%%%%%%%%%%%%%%%%%%%%%%%%%%%%

\appendix

\section{Dimension reduction method and deterministic solver for the BGK equation}
\label{appendix}

In this Appendix, we briefly describe the dimension reduction method adopted to reduce the computational complexity of the BGK equation and the details of the numerical methods used to discretize time, physical space and velocity space. Since the Monte Carlo methods are non-intrusive, our discussion will be based on the deterministic equation (\ref{eqn:bgk}) for simplicity. 

\subsection{The Chu reduction method}

The BGK equation (\ref{eqn:bgk}) lives in six-dimensional phase space whose computation can be extremely expensive. Under certain homogeneity assumptions, one can reduce the dimension using the so-called Chu reduction \cite{Chu1965}.

Let $\vx=(x_1,x_2,x_3)$, $\vv=(v_1,v_2,v_3)$, and $\mU=(U_1,U_2,U_3)$. If the physical domain $D\in \mathbb{R}^3$ is effective only in one dimension and homogeneous in the other two dimensions (e.g., a tube), it is reasonable to assume the following:
\be
\label{ass:chureduction}
\bal
\p_{x_2}f=\p_{x_3}f=0, \quad U_2=U_3=0.
\eal
\ee
Then the equation (\ref{eqn:bgk}) becomes
\be
\label{eq: bgkchu}
\p_{t} {f(x_1,v_1,v_2,v_3,t)}  +v_1  \p_{x_1} f(x_1,v_1,v_2,v_3,t)=\frac{1}{\varepsilon}\left(M[f]-f(x_1,v_1,v_2,v_3,t)\right),
\ee
where \be
M[f](x_1,v_1,v_2,v_3,t)=\fl {\rho(x_1,t)} {{(2\pi T(x_1,t))}^{\fl 3 2}}\exp\left({-\fl {{(v_1-U_1(x_1,t))}^2+v_2^2+v_3^2}{2T(x_1,t)}}\right).
\ee
The Chu reduction proceeds by introducing two distribution functions:
\begin{align}
\label{def:phi}
&\phi(x_1,v_1,t):=\iint_{\mathbb{R}^2} f(x_1,v_1,v_2,v_3,t) \, {\rm d}v_2{\rm d}v_3,\\
\label{def:psi}
& \psi(x_1,v_1,t):=\iint_{\mathbb{R}^2} \left(\fl 1 2 v_2^2 + \fl 1 2 v_3^2\right)f(x_1,v_1,v_2,v_3,t) \, {\rm d}v_2{\rm d}v_3.
\end{align}
It is then easy to derive that $\phi$ and $\psi$ satisfy the following system:
\begin{align}
\label{eqn:1dbgkchuphi}
&\p_{t} {\phi(x_1,v_1,t)}  +v_1  \p_{x_1} \phi(x_1,v_1,t)=\frac{1}{\varepsilon}\left(M_{\phi}(x_1,v_1,t)-\phi(x_1,v_1,t)\right),\\
\label{eqn:1dbgkchupsi}
&\p_{t}  {\psi(x_1,v_1,t)}  +v_1  \p_{x_1} \psi(x_1,v_1,t)=\frac{1}{\varepsilon}\left(M_{\psi}(x_1,v_1,t)-\psi(x_1,v_1,t)\right),
\end{align}
where
\begin{align}
\label{eqn:1dmaxchuphi}
&M_{\phi}(x_1,v_1,t):=\iint_{\mathbb{R}^2} M[f] \,{\rm d} v_2 {\rm d} v_3=\fl {\rho(x_1,t)} {\sqrt{2\pi T(x_1,t)}}\exp\left({-\fl {(v_1-U_1(x_1,t))^2}{2T(x_1,t)}}\right),\\
\label{eqn:1dmaxchupsi}
&M_{\psi}(x_1,v_1,t):=\iint_{\mathbb{R}^2} \left(\fl 1 2 v_2^2 +\fl 1 2 v_3^2\right)M[f] \, {\rm d} v_2  {\rm d} v_3= {T(x_1,t)}M_{\phi}.
\end{align}
Denoting $\int_{\mathbb{R}} \cdot  \rd{v_1}=\langle \cdot \rangle$, it is easy to see the following relation holds
\be \label{eqn:discretemacroquant}
\bal
&\rho=\int_{\mathbb{R}}  \phi \rd{v_1}=\int_{\mathbb{R}}  M_{\phi} \rd{v_1},\\
&m=\rho U_1=\int_{\mathbb{R}}  v_1\phi \rd{v_1}=\int_{\mathbb{R}}  v_1M_{\phi} \rd{v_1},\\
&E=\frac{1}{2}\rho U_1^2+\frac{3}{2}\rho T=\int_{\mathbb{R}} \left( \fl 1 2 v_1^2\phi +\psi\right) \rd{v_1}=\int_{\mathbb{R}} \left( \fl 1 2 v_1^2 M_{\phi}+M_{\psi} \right)\rd{v_1}.
\eal
\ee
                  
Now our task is to solve the reduced 1D BGK system \cref{eqn:1dbgkchuphi}-\cref{eqn:1dbgkchupsi}. 

\subsection{The fully discrete scheme}  

The fully discrete scheme used to solve \cref{eqn:1dbgkchuphi}-\cref{eqn:1dbgkchupsi} consists of three components: velocity discretization, time discretization, and spatial discretization.
                                     
\subsubsection*{Velocity discretization}

In the velocity space, we follow the discrete velocity method (see Section 4.1.1 in \cite{GPT} or \cite{Mieussens2000} for example), which satisfies a discrete entropy decay property.

We first truncate the infinite velocity domain into a bounded interval $[-R, R]$ and then discretize it using $N_v$-point Gauss quadrature with $(\xi_k,w_k)$, $k=1,2,\dots,N_v$ as abscissae and weights. To obtain $M_\phi$, $M_\psi$ from $\phi$ and $\psi$, normally one could use the relation in (\ref{eqn:discretemacroquant}), where the continuous integral is replaced by the Gauss quadrature. However, due to the domain truncation error, the resulting moments are not sufficiently accurate. To remove this error, we assume 
\be
{M}_\phi=\exp(\alpha_1 +\alpha_2 v_1+\alpha_3 v_1^2), \quad M_\psi=-\frac{1}{2\alpha_3}M_\phi,
\ee
and determine $\alpha_1$, $\alpha_2$, $\alpha_3$ such that
\be
\label{eqn:newton}
\bmat{\langle M_{\phi}\rangle\\\langle v_1M_{\phi}\rangle\\ \langle\fl 1 2 v_1^2 M_{\phi}+M_\psi\rangle}=\bmat{\langle {\phi}\rangle\\\langle v_1{\phi}\rangle\\ \langle\fl 1 2 v_1^2 {\phi}+{\psi}\rangle}:=\bmat{\rho\\ m \\ E},
\ee
where $\langle u(v_1) \rangle:=\sum_{k=1}^{N_v} u(\xi_k)w_k$ denotes the quadrature sum in the interval $[-R,R]$. The above nonlinear system is solved by the Newton-Raphson algorithm. 

\subsubsection*{Time discretization}

Due to the possibly stiff collision term, we use the implicit-explicit Runge-Kutta (IMEX-RK) scheme \cite{dimarcopareschiIMEX,Puppo} for the time discretization. In particular, we employ the second-order IMEX-RK scheme proposed in \cite{Hu2018}, which is positivity preserving and asymptotic preserving (preserving the Euler limit without $\Delta t$ resolving $\varepsilon$). 

Specifically, we discretize (\ref{eqn:1dbgkchuphi}) and (\ref{eqn:1dbgkchupsi}) as
\be \label{scheme2}
\begin{split}
&\phi^{(i)}=\phi^n-\Delta t \sum\limits_{j=1}^{i-1}\tilde{a}_{ij}v_1 \p_{x_1} \phi^{(j)}+\Delta t \sum\limits_{j=1}^{i}a_{ij}\fl {1} \varepsilon \left(M_\phi^{(j)}-\phi^{(j)}\right), \quad i=1,\dots,\nu,\\
&\psi^{(i)}=\psi^n-\Delta t \sum\limits_{j=1}^{i-1}\tilde{a}_{ij}v_1 \p_{x_1} \psi^{(j)}+\Delta t \sum\limits_{j=1}^{i}a_{ij}\fl {1} \varepsilon \left(M_\psi^{(j)}-\psi^{(j)}\right), \quad i=1,\dots,\nu,\\
&\phi^{n+1}=\phi^{(\nu)}+\alpha{\Delta t}^2\fl {1}{{\varepsilon}^2}    \left(M_\phi^{n+1}-\phi^{n+1}\right),\\
&\psi^{n+1}=\psi^{(\nu)}+\alpha{\Delta t}^2\fl {1}{{\varepsilon}^2}    \left(M_\psi^{n+1}-\psi^{n+1}\right),
\end{split}
\ee
where the values of the coefficients \(\tilde{a}_{ij},a_{ij},\alpha\) are given in Section 2.6.1 of \cite{Hu2018}. To implement the above scheme explicitly, we first solve the moment system 
\be
\bal
&\bmat{\langle \phi^{(i)} \rangle \\ \langle v_1\phi^{(i)} \rangle \\  \langle \frac{1}{2}v_1^2 \phi^{(i)} +\psi^{(i)} \rangle}=\bmat{\langle \phi^{n} \rangle \\ \langle v_1\phi^{n} \rangle \\  \langle \frac{1}{2}v_1^2 \phi^{n} +\psi^{n} \rangle}
-\Delta t \sum\limits_{j=1}^{i-1}\tilde{a}_{ij} \bmat{\langle v_1\partial_{x_1}\phi^{(j)} \rangle \\ \langle v_1^2\partial_{x_1}\phi^{(j)} \rangle \\  \langle \frac{1}{2}v_1^3 \partial_{x_1}\phi^{(j)} +v_1 \partial_{x_1}\psi^{(j)} \rangle}, \!\!\!\quad i=1,\dots,\nu,\\
&\bmat{\langle \phi^{n+1} \rangle \\ \langle v_1\phi^{n+1} \rangle \\  \langle \frac{1}{2}v_1^2 \phi^{n+1} +\psi^{n+1} \rangle}=\bmat{\langle \phi^{(\nu)} \rangle \\ \langle v_1\phi^{(\nu)} \rangle \\  \langle \frac{1}{2}v_1^2 \phi^{(\nu)} +\psi^{(\nu)} \rangle},
\eal
\ee
which is obtained by taking the moments of (\ref{scheme2}) and using (\ref{eqn:newton}). Hence we can obtain $\rho^{(i)}$, $m^{(i)}$ and $E^{(i)}$ first, and use them to define $M_{\phi}^{(i)}$ and $M_{\psi}^{(i)}$. Finally we solve (\ref{scheme2}) to get $\phi^{(i)}$ and $\psi^{(i)}$.

\subsubsection*{Spatial discretization}

In the physical space, we use the second order MUSCL finite volume scheme \cite{VanLeer79}.

Here we take the following first order in time scheme for $\phi$ as an illustration (suppose it is evaluated at velocity point $v_1=\xi_k$):
\be
\label{eqn:general1dbgk}
\frac{\phi^{n+1}_k(x_1)-\phi^n_k(x_1)}{\Delta t}  + \xi_k \partial_{x_1} \phi^n_k(x_1)=\frac{1}{\varepsilon}\left((M_{\phi})^{n+1}_k(x_1)-\phi^{n+1}_k(x_1)\right).
\ee
Suppose $x_1\in[a,b]$ and $[a,b]$ is divided into $N_x$ uniform cells with size $\Delta x=(b-a)/N_x$, where $a=x_{\frac 12}$, $b=x_{N_x+\frac 12}$. In the cell $[x_{j-\frac 12},x_{j+\frac 12}]$, define the cell average as
\be
\phi_{j,k}^n:=\fl 1 {\Delta x} \int_{x_{j-\fl 1 2}}^{x_{j+ \fl 1 2}} \phi^n_k(x_1) \ {\rm d}x_1.
\ee
Then integrating (\ref{eqn:general1dbgk}) over $[x_{j-\frac 12},x_{j+\frac 12}]$ yields
\be
\fl {\phi^{n+1}_{j,k}-\phi^{n}_{j,k}}{\Delta t}+\fl {{F_{j+\fl 1 2,k}^n}-{F_{j-\fl 1 2,k}^n}}{\Delta x}=\fl 1 {\varepsilon}\left( (M_\phi)_{j,k}^{n+1}-\phi^{n+1}_{j,k}\right), 
\ee
where $(M_\phi)_{j,k}^{n+1}:=(M_\phi)_k^{n+1}(x_j)$. Note that we have replaced the cell average of $(M_{\phi})^{n+1}_k$ by its point value at cell center $x_j$ (the error introduced by this is $O(\Delta x^2)$ which does not destroy the overall order of the method). \({F_{j+\fl 1 2,k}^n} \) is the flux at interface \(x_{j+\fl 1 2}\) and is defined as
\be
{F_{j+\fl 1 2,k}^n}=\max(0,\xi_k)\phi^n_{l,j,k}+\min(0,\xi_k)\phi^n_{r,j+1,k},
\ee
with the left interface and right interface values  \(\phi_{l,j,k}^n,\phi_{r,j,k}^n\)  given by
\be
\left\{
\begin{aligned}
\phi_{l,j,k}^n&=\phi_{j,k}^n+\fl 1 2 \Delta x \sigma_{j,k}^{n} ,\\
\phi_{r,j,k}^n&=\phi_{j,k}^{n}-\fl 1 2 \Delta x \sigma_{j,k}^{n},
\end{aligned}
\right.
\ee
where \(\sigma_{j,k}^{n}\) is the slope of the linear reconstruction and is chosen to be the MC limiter (\(\theta =2\)):
\be
\sigma_{j,k}^{n}=\text{minmod} \left (\fl {\phi_{j+1,k}^n-\phi_{j-1,k}^n}{2 \Delta x},\theta\left(\fl {\phi_{j,k}^n-\phi_{j-1,k}^n}{\Delta x}\right),\theta\left(\fl {\phi_{j+1,k}^n-\phi_{j,k}^n}{\Delta x}\right)\right ).
\ee

\bibliographystyle{siamplain}
\bibliography{ref}

\end{document}